\documentclass{amsart}
\usepackage{amsmath}
\usepackage{amsfonts}
\usepackage{amssymb}
\usepackage{graphicx, enumitem, centernot}
\usepackage{hyperref}
\def\st{\textrm{st}}
\def\N{\mathbb{N}}
\def\E{\mathbb{E}}
\def\R{\mathbb{R}}

\def\P{\mathcal{P}}
\newcommand*{\slim}{\rm{Lim} \ }
\newcommand*{\unright}{\underset{N}{\longrightarrow}}
\providecommand{\U}[1]{\protect \rule{.1in}{.1in}}
\newtheorem{theorem}{Theorem}[section]

\newtheorem{convention}[theorem]{Convention}
\newtheorem{corollary}[theorem]{Corollary}

\newtheorem{definition}[theorem]{Definition}
\newtheorem{example}[theorem]{Example}

\newtheorem{lemma}[theorem]{Lemma}
\newtheorem*{notation}{Notation}

\newtheorem{proposition}[theorem]{Proposition}
\newtheorem{remark}[theorem]{Remark}

\DeclareMathOperator{\Ima}{im}

\begin{document}
	\title{On flexible sequences}
	\author{Bruno Dinis}
\address[B. Dinis]{Departamento de Matem\'{a}tica, Faculdade de Ci\^{e}ncias da
Universidade de Lisboa, Campo Grande, Ed. C6, 1749-016, Lisboa, Portugal}
\email{bmdinis@fc.ul.pt}
\author{Tran Van Nam}
\address[T.V. Nam]{The University of Danang, Campus in Kontum
704 Phan Dinh Phung street, Kontum City, Kontum province, Vietnam.}
\email{vannamtran1205@gmail.com}
\author{Imme van den Berg}
\address[I.P. van den Berg]{Departamento de Matem\'{a}tica, Universidade de \'{E}vora, Portugal}
\email{ivdb@uevora.pt}
\thanks{The first author acknowledges the support of the Centro de Matem\'atica, Aplica\c c\~oes Fundamentais e Investiga\c c\~ao Operacional~/ Funda\c c\~ao da Faculdade de Ci\^encias da Universidade de Lisboa via the grant~UID/MAT/04561/2013 and a postdoc-grant from Erasmus Mundus Mobility with Asia--East~14.  \\
The second author acknowledges a PhD-grant of Erasmus Mundus Mobility with Asia--East~14.}
\subjclass[2010]{03H05, 40A05}
\keywords{External numbers, flexible sequences, convergence, nonstandard analysis}

	\begin{abstract}
		In the setting of nonstandard analysis we introduce the notion of flexible sequence. The terms of flexible sequences are external numbers. These are a sort of analogue for the classical \emph{O$ (\cdot ) $} and \emph{o$ (\cdot ) $} notation for functions, and have algebraic properties similar to those of real numbers. The flexibility originates from the fact that external numbers are stable under some shifts, additions and multiplications. We introduce two forms of convergence, and study their relation. We show that the usual properties of convergence of sequences hold or can be adapted to these new notions of convergence and give some applications.
		
	\end{abstract}
	\maketitle
	\tableofcontents

	\section{Introduction}
	
	In this article we introduce the notion of flexible sequences and study a general form as well as convergence properties and behavior under operations. In our context flexibility deals with stability under some shifts, additions and multiplications, and is a matter of order of magnitude, or degree. The term ``flexible'' is borrowed from \cite{Justino}, where the propagation of errors in matrix calculus is modeled similarly. Our objectives are to extend this error analysis to sequences, and to make rigorous and precise certain ways of speaking in asymptotics, such as ``converging to $o(\epsilon)$'' where $\epsilon$ is arbitrarily small, or ´´neglecting small fluctuations´´. We were inspired by the \textit{ars negligendi} of Van der Corput \cite{VdCorput}; this is the art of neglecting some values which are are small or unimportant with respect to other aspects of some problem.
	
	Van der Corput introduced as a tool \emph{neutrices}, referring to additive commutative groups of functions without unity. 
	Here we use the (scalar) neutrices introduced in \cite{koudjetithese,koudjetivandenberg} in the setting of nonstandard analysis. Then a neutrix is an additive convex subgroup of the (nonstandard) real numbers. Except for the obvious neutrices $\{0\}$ and the set of real numbers itself, neutrices are external sets. Some simple examples of neutrices are $\oslash $, the external set of all real infinitesimals, and $\pounds $, the external set of all limited real numbers (numbers bounded in absolute value by a standard number). An \emph{external number} is the algebraic sum of a neutrix and a real number. A \emph{flexible sequence} is a sequence of external numbers. By convexity and the group property, neutrices, hence also external numbers are stable under some shifts, additions and multiplications indeed, and there is a large, in fact infinite variety of neutrices \cite{neutrixdecompositie}.
	
	In the setting of external numbers more algebraic properties and properties of order are valid than in the context of Van der Corput's functional neutrices. External numbers satisfy, to a large extent, the algebraic and analytic properties of the real numbers \cite{koudjetithese,koudjetivandenberg,dinisberg 2011}, including a Generalized Dedekind completeness property (see \cite{vdbnaa}). Structures with these properties were called \emph{complete arithmetical solids} in \cite%
	{dinisberg 2016}(see also \cite{dinisbergind1,dinisbergind2}). This means in particular that we may use the algebraic properties of neutrices to study properties of convergence. 
	
	We study two types of convergence to external numbers. When $N$ is a neutrix, $N$-convergence expresses convergence to an external number with a tolerance $N$, i.e. we approach the external number within distance $\varepsilon  $ for all $\varepsilon > N  $. \emph{Strong convergence} says that the sequence enters the external number in (nonstandard) finite time, meaning that from a certain natural number onwards the terms of the sequence are included in the limit set; observe that the sequence still may exhibit fluctuations which are less than the neutrix of the external number. Due to the fact that the convergence may happen in finite time, also a refinement is considered, taking into consideration the segment on which the convergence actually happens. 
	
	A main result is that $N$-convergence rather universally implies strong convergence. Loosely translated into classical terms, this means that if we approach a $O(\cdot)$, or a $o(\cdot)$-neighborhood of some object, we enter this neighborhood.
	The proof needs a thorough investigation on the foundations of flexible sequences, as external sets of nonstandard analysis. We do so in the context of Nelson's axiomatic approach Internal Set Theory $\mathsf{IST}$ \cite{Nelson}, in the bounded form permitting to deal with external sets presented by Kanovei and Reeken \cite{KanoveiReeken}. In this setting strong convergence results are provable for definable sequences. 
	The strong convergence theorems are particularly useful for the study of to what extent the properties of limits of sequences of real numbers are valid or can be adapted to flexible sequences. In the last part of this article we consider Cauchy properties and present some applications. They deal with the Borel-Ritt Theorem on asymptotic expansions, the problem of matching in singular perturbations, and a concept of near-stability. 
	
	Section \ref{Section_nature_external_sequences} starts by investigating the form that flexible sequences may take. We use the relatively low complexity of external sets with internal elements which is a consequence of Nelson's Reduction Algorithm \cite{Nelson} and a form of Saturation \cite{Nelsonsyntax} valid for ordered sets. We prove a Representation Theorem for \emph{precise} sequences, i.e external sequences of real numbers. External sequences of external numbers happen to admit an internal partial choice sequence. This is also true for sequences defined on an (external) initial segment, and then the domain of such a choice function enters in the complement of the segment, giving rise to the notion of \emph{internal bridge}. These internal choice sequences are a tool in proving convergence theorems, for they satisfy classical notions of convergence.
	
	In Section \ref{Section e-Sequences} we define the notions of $N$-convergence, strong convergence, and convergence with respect to an initial segment. We consider relevant properties of the ordering and of absolute values of external numbers in Subsection  \ref{Section Order relations}. In Subsection \ref{Subsection infinite sequences} we derive some useful simplifications: the neutrix $ N $ may be taken minimal, and then the neutrix of the limit may be taken equal to this minimal neutrix. This subsection ends with the strong convergence theorem for infinite sequences. Convergence and strong convergence with respect to an initial segment is defined in Subsection \ref{Subsection initial segment}, which also contains a strong convergence theorem for these notions.
	
	Profiting from the foundational results of Section \ref{Section_nature_external_sequences} the strong convergence theorems are proved in Section \ref{subsequences}.
	
	Section \ref{Section Operations Limits} studies the fundamental properties of convergence in the context of flexible sequences. Subsection \ref{Boundedness and monotonicity} deals with boundedness and ordering and Subsection \ref{Operations} with the behavior under operations, where we note that the size of the neutrix in the $ N $-convergence may change. The strong convergence theorems appear to be particularly useful.
	
	In Section \ref%
	{Section Cauchy} we study $N$-Cauchy sequences and the relation between $%
	N$-Cauchy sequences and $N$-convergence.
	
	In the final Section \ref{Section applications} we give some applicatioms. A nonstandard version of the Borel-Ritt theorem happens to be a consequence of the fact that $N$-Cauchy convergence with respect to an initial segment implies $N$-convergence with respect to this initial segment. Subsection \ref{Subsection Matching Principles} interprets the matching problem for singular perturbations in the setting of $N$-convergence, and shows that a modified form of the strong convergence theorem yields a general matching theorem for sequences, which can be adapted to functions. Some examples are given. Subsection \ref{subsection stability} defines solutions to recurrent relations with external numbers, notions of approximate stability and asymptotic stability, and gives conditions for such stability to hold.

	\section{On the nature of external sequences}\label{Section_nature_external_sequences}
	
	We define flexible sequences and investigate the form that these external sequences and some relevant
	external sets may take. 
	
	\begin{convention}\label{convention}
	We will always assume that external sets $ E $ of internal elements are definable in $ \mathsf{IST} $, and that they can be defined by a \emph{bounded} formula; this means that every variable is bounded by a standard set. An axiomatics for external sets is contained in \cite{KanoveiReeken}.
	\end{convention}
	By Nelson's Reduction Algorithm such an external set $ E $ can be written in the form 
	\begin{equation}\label{external set}
	E =\bigcup \limits_{\mathrm{st} (x)\in X}\bigcap \limits_{\mathrm{st}(y)\in Y}I_{xy},  
	\end{equation}%
	where $X$ and $Y$ are standard sets and $I$ is an internal set-valued mapping.
	If $ J $ is an internal set-valued mapping, a set of the form
	\begin{equation*}
	G_{x}=\bigcup \limits_{\mathrm{st}(x)\in X}J_{x}  
	\end{equation*}
	is called a \emph{pregalaxy}, and a \emph{galaxy} if it is not internal. We denote by $^{\sigma}X$ the set that consists of the standard elements of a standard set $X$. It is a galaxy if $X$ is infinite. A set of the form
	\begin{equation*}
	H_{x}=\bigcap \limits_{\mathrm{st}(x)\in X}J_{x}  
	\end{equation*}
	is called a \emph{prehalo}, and a \emph{halo} if it is not internal. In general, an external set which is not internal is called \emph{strictly external}. 
	The somewhat trivial fact that no strictly external set is internal is called the \emph{Cauchy Principle}. The so-called \emph{Fehrele Principle} \cite{Diener-Van den Berg} states that no halo is a galaxy. Both these principles act as \emph{Permanence Principles}, enabling to show that a property which is of a different nature than the set on which it is established, still holds somewhere outside this set.
	
	Let  $C$ be an initial segment of $ \mathbb{N} $.
	If $C$ is internal it is just a discrete interval $\{0,...,c\}$ with $c=\max C$. If $C$ is a galaxy, it may be written in the form%
	\begin{equation*}
	C=\bigcup \limits_{\st(z)\in Z}\{0,...,c_{z}\},
	\end{equation*}%
	where $Z$ is a standard ordered set and $c:Z\rightarrow 
	%TCIMACRO{\U{2115} }%
	%BeginExpansion
	\mathbb{N}
	%EndExpansion
	$ is increasing, at least on the standard elements of $Z$. If $C$ is a halo, it may be written in the form%
	\begin{equation*}
	C=\bigcap \limits_{\st(w)\in W}\{0,...,c_{w}\},
	\end{equation*}%
	where $W$ is a standard ordered set and $c:W\rightarrow 
	%TCIMACRO{\U{2115} }%
	%BeginExpansion
	\mathbb{N}
	%EndExpansion
	$ is decreasing, at least on the standard elements of $W$.	
	
	Let $ X,Y $ be standard sets. Sometimes it is useful to suppose that for fixed $ x\in X $, the sets $ I_{xy} $ of \eqref{external set} are closed under finite intersections. Then
	\begin{equation}\label{finite instersection}
		\forall^\mathrm{stfin}z\subseteq Y \exists^\mathrm{st}\overline{y}\in Y\left( \bigcup \limits_{\mathrm{st}(y)\in z}I_{xy}=I_{x\overline{y}}\right) .
	\end{equation} 
	Let $ \Phi$ be a possibly external property. The \emph{Saturation Principle} \cite{Nelson} states that
	\begin{equation*}
	\forall^\mathrm{st} x\in X\exists y \in Y\Phi(x,y)\Leftrightarrow\exists \widetilde{y}\left(\widetilde{y}:X\rightarrow  Y \wedge \forall^ \mathrm{st} x\in X \Phi\left(x,\widetilde{y}(x)\right)\right).
	\end{equation*}
	We often use Saturation in the following form.
		
	\begin{lemma}\label{cofinal sets}
		Let $ H\subseteq \N  $ be a convex prehalo, let  $ X $ be a standard set and $ f:X \rightarrow \P(\N) $ be a definable mapping such that $ f(x)\subseteq H $ for all standard $ x\in X $.
		\begin{enumerate}
			\item\label{cs1} If for all standard $ x\in X $ there exists $ n\in H$  such that $ f(x)<n $, then there exists $ h\in H $ such that $ f(x)<h $ for all standard $ x\in X $.
			\item\label{cs2} If $ H=\bigcup _{\mathrm{st}(x)\in X}f_{x}  $, then there exists a standard $ x\in X $ such that $ f(x) $ is cofinal in $ H $.
		\end{enumerate}
	\end{lemma}
	\begin{proof}
		\ref{cs1}. By the Saturation Principle there exists an internal function $ \widetilde{n } :X\rightarrow \N$ such that $ \widetilde{n }(x)\in H $ and $ f(x)< \widetilde{n }(x)$ for all standard $ x\in X $. By the Fehrele Principle there exists $ h\in H $ such that $ f(x)<h $ for all standard $ x\in X $.
		
		\ref{cs2}. This is a direct consequence of Lemma \ref{cofinal sets}.\ref{cs1}.
	\end{proof}
	
We denote by $\mathbb{E}$ the set of external
	numbers as constructed in \cite[Section 6]{dinisberg 2016}. An external number $\alpha$ is always of the form $\alpha=a+A$, where $a$ is a (possibly nonstandard) real number and $A$ is a neutrix, called the \emph{neutrix part} of $\alpha$. We may write $A=N(\alpha)$. 
	An external number $\alpha$ is said to be \emph{zeroless} if $0 \notin \alpha$.

	\begin{definition}
		A \emph{flexible} sequence $(u_{n})$ is an external mapping  $u:\mathbb{N} \rightarrow 	\mathbb{E} 	
		$. A flexible sequence $(u_{n})$ is called \emph{precise}, if $u_{n}\in \mathbb{R}$ for all $ n\in \mathbb{N} $. We call $(u_{n})$ definable if its graph
		\begin{equation*}
		\Gamma(u):=\{ (n,y): n \in \mathbb{N} \wedge y \in u_{n} \} 
		\end{equation*} has the form
		\begin{equation}
		\Gamma(u) =\bigcup \limits_{\mathrm{st} (x)\in X}\bigcap \limits_{\mathrm{st}(y)\in Y}I_{xy},  \label{Formule u}
		\end{equation}%
		where $X$ and $Y$ are standard sets and $I:X\times Y\rightarrow \mathcal{P}(\N \times \R)$ is an internal mapping.  
	\end{definition}

As a consequence of Convention \ref{convention} all flexible sequences are tacitly supposed to be definable indeed.

Assume that $(u_n)$ is a precise sequence. In \eqref{Formule u} put for all standard $x\in X$ 
	\begin{equation*}
	Q_{x}=\bigcap \limits_{\mathrm{st}(y)\in Y}I_{xy}.  \label{Formule Q}
	\end{equation*}%
	Then 
	\begin{equation}\label{uunionQ}
	u=\bigcup \limits_{\mathrm{st}(x)\in X}Q_{x}.  
	\end{equation}%
	Since $(u_n)$ is functional, for all standard $x\in X$ the prehalo $Q_{x}$ is
	functional. For every standard $x\in X$, we define the prehalo $\Delta_{x}$ by
	\begin{equation}
	\Delta _{x}=\text{dom}(Q_{x})\text{.}  \label{Formule Delta}
	\end{equation}

Let $ n\in \N $. Assume that \eqref{finite instersection} holds. If $ u_{n} $ is a pregalaxy, at this  $ n $ the representation \eqref{Formule u} can be refined. We will use the functional notation also for $ I_{xy}(n) $, $ Q_{x}(n) $, etc. to indicate values taken at $ n $. We show that there exists a standard $ \overline{y}\in Y $ such that
\begin{equation}\label{Uxyn}
u_n=\bigcup _{\mathrm{st}(x)\in X}I_{x\overline{y}}(n).   
\end{equation}
To prove this, suppose that $ Q_{x}(n) $ is cofinal in $ u_{n} $ for some  standard $x\in X$. By convexity $ Q_{x}(n) = u_{n} $, in contradiction with the Fehrele Principle. Hence there exist $ s,t\in u_{n} $ such that $ Q_{x}(n)\subseteq [s,t] $. By Idealization there exists standard $\overline{y}\in Y $ such that $  I_{x\overline{y}}(n)\subseteq [s,t] \subset u_{n}$. So
\begin{equation*}
u_{n}\subseteq\bigcup _{\mathrm{st}(x)\in X}I_{x\overline{y}}(n) \subseteq  \bigcup _{\mathrm{st}(x)\in X}\bigcap_{\mathrm{st}(y)\in Y} I_{xy}(n) \subseteq u_{n},   
\end{equation*}
which implies \eqref{Uxyn}.

	We extend the notion of sequence to external functions $u:E \rightarrow \mathbb{E} 	
	$, where $ E $ is some (definable) external subset of $ \mathbb{N} $.
	
	We prove the following theorem of representation for external sequences of real numbers.

	\begin{theorem}[Representation Theorem]\label{Representationsequence}
		Let $ (u_n)$ be a precise sequence. If $ (u_n) $ is not internal, it can neither be a galaxy, nor a halo. Then there exist a standard set $ X $, an internal family of (internal) sequences $ \left( v_{x}\right) _{x \in X} $ and a family of prehalos  $ \left( \Delta_{x}\right) _{x \in X} $ such that
		\begin{equation}
		u=\bigcup \limits_{\mathrm{st}(x)\in X}v_{x\upharpoonright\Delta_{x}}.  \label{Formule uv}
		\end{equation}
	\end{theorem}
	
	The theorem implies that a truly external sequence of internal elements must be of full complexity ($ \Sigma_{2}^{\st}=\Pi_{2}^{\st} $), but with properties coming from internal sequences. Indeed, the external sequence originates from an internal family of internal sequences, for which we consider only sequences with standard indices, each sequence restricted to a prehalic domain. This suggests that typically such an external sequence is defined by cases. We present here a simple example.

	Let $u:\N	\rightarrow	\R	$ be defined by%
	\begin{equation}\label{caseu}
	u_{n}= 
	\begin{cases}
	1 &   n\text{ }\in \text{}^{\sigma }%
	\N
	\\ 
	0 &   n\text{ }\in \N 	\cap \centernot{\infty}%
	\end{cases}%
	; 
	\end{equation}%
	here formula (\ref{Formule uv}) takes the form%
	\begin{equation}\label{unionintersectionu}
	u=\bigcup \limits_{\st(m)\in 
		%TCIMACRO{\U{2115} }%
		%BeginExpansion
		\mathbb{N}
		%EndExpansion
	}\bigcap \limits_{\st(n)\in 
		%TCIMACRO{\U{2115} }%
		%BeginExpansion
		\mathbb{N}
		%EndExpansion
	}(I_m\times \{1\} \cup J_n\times \{0\}),
	\end{equation}%
	where $I_m=\{k \in \N : 0 \leq k \leq m\}$ and $J_n=\{p \in \N : p \geq n\}$.
	On the halo $
	\mathbb{N} \cap \centernot{\infty}$ the external sequence $(u_n)$ is the restriction of, say, the internal sequence which is identically zero. Also, the domains $\{0,...,m\} $ for $ \st(m) $ are internal, and their union  $^{\sigma }
	\mathbb{N}$ is a galaxy. On this galaxy the external sequence $(u_n)$ is the restriction of, say, the internal sequence which is identically one. 
	So, it may very well happen that some of the prehalic domains  $ \Delta_{x} $ are internal, and give way to a union which is a galaxy, say $ G $. As we will see in Lemma \ref{Lemma internal extension} below, the restrictions of the external sequence $(u_n)$ to these internal domains are also internal, so their union is a galaxy, which in turn is the restriction of an internal sequence to the galaxy $ G $. We recognize in the example above indeed that the representation  \eqref{unionintersectionu} in terms of unions of prehalos reduces to the two cases given in \eqref{caseu}.
		
	By \eqref{uunionQ}, to prove Theorem \ref{Representationsequence}, we need to show that the prehalos $ Q_{x} $ have internal extensions, and as such are the restrictions to the prehalos $ \Delta _{x}$ of these internal extensions.
	Next lemma says that galactic sequences and halic sequences are internal on
	every internal subdomain, and have an internal extension  indeed.
	
	\begin{lemma}
		\label{Lemma internal extension}Let $E$ be an nonempty subset of $%
		%TCIMACRO{\U{2115} }%
		%BeginExpansion
		\mathbb{N}
		%EndExpansion
		$ and $J\subseteq E$ be internal. Let $u:E \rightarrow 
		%TCIMACRO{\U{211d} }%
		%BeginExpansion
		\mathbb{R}
		%EndExpansion
		$ be a sequence, which is either a galaxy or a halo. Then $u_{\upharpoonright
			J}$ is internal, and there exist $K$ such that $E\subseteq K\subseteq \mathbb{N}	$ and an internal sequence $v:K\rightarrow 	\mathbb{R}$ such that $v_{\upharpoonright E}=u$.
	\end{lemma}
	
	\begin{proof}
		We consider first the case where $u$ is a galaxy. In this case we may suppose that the sets $Q_{x}$ of formula (\ref{uunionQ})  are internal. Then also the sets $\Delta_{x} $ defined by (\ref%
		{Formule Delta}) are internal. Let $J\subseteq E$ be internal. Now%
		\begin{equation*}
		\forall k\in J\exists ^{\st}x\in X(k\in \Delta _{x}).
		\end{equation*}%
		So by Idealization%
		\begin{equation*}
		\exists ^\mathrm{st fin}z\subseteq X\forall k\in J\exists ^{\st}x\in z(k\in \Delta
		_{x}).
		\end{equation*}%
		Then $J\subseteq \bigcup _{x\in z}\Delta _{x}$. For $x\in z$, put $J_{x}=\Delta
		_{x}\cap J$, then $J_{x}$ is internal, hence also $u_{\upharpoonright
			J_{x}}=Q_{x{\upharpoonright J_{x}}}$ is internal. Because $J=\bigcup _{x\in
			z}J_{x}$, we have $u_{\upharpoonright J}=\bigcup _{x\in z}u_{\upharpoonright
			J_{x}}$, which is internal as it is a standard finite union of internal sets.
		
		To construct an internal extension of $u$, we note that by the above%
		\begin{equation*}
		\begin{split}
			\forall ^\mathrm{st fin}z &\subseteq X\exists v\forall x\in z \\
			(v &:%
			%TCIMACRO{\U{2115} }%
			%BeginExpansion
			\mathbb{N}
			%EndExpansion
			\rightarrow 
			%TCIMACRO{\U{211d} }%
			%BeginExpansion
			\mathbb{R}
			%EndExpansion
			\text{ is a partial function such that }\forall n\in \Delta
			_{x}(v_n=Q_{x_{\upharpoonright J_{x}}}(n)).
			\end{split}
		\end{equation*}%
		By Idealization%
		\begin{equation*}
		\begin{split}
			\exists v\forall ^\mathrm{st}x &\in X \\
			(v &:%
			%TCIMACRO{\U{2115} }%
			%BeginExpansion
			\mathbb{N}
			%EndExpansion
			\rightarrow 
			%TCIMACRO{\U{211d} }%
			%BeginExpansion
			\mathbb{R}
			%EndExpansion
			\text{ is a partial function such that }\forall  n\in \Delta
			_{x}(v_n=Q_{x_{\upharpoonright J_{x}}}(n)).
			\end{split}
		\end{equation*}%
		Put $K=$ dom$(v)$. Then $v$ is an internal real-valued function, such that $%
		v_n=u_n$ for all $n\in E$, and $E\subseteq K\subseteq 
		%TCIMACRO{\U{2115} }%
		%BeginExpansion
		\mathbb{N}
		%EndExpansion
		$.
		
		We turn now to the case where $u$ is a halo. We construct first an internal
		extension. We may write $u=\bigcap _{\st(y)\in Y}I_{y}$, where $Y$ is a standard
		set and $I:Y\rightarrow \mathcal{P}(\mathbb{N}\times\mathbb{R})$ is an internal mapping. Put for every $n\in 
		%TCIMACRO{\U{2115} }%
		%BeginExpansion
		\mathbb{N}
		%EndExpansion
		$ and $y\in Y$%
		\begin{equation*}
		I_{y}(n)=\left \{ a\in \R : (n,a)\in I_{y} \right \} .
		\end{equation*}%
		Let $n\in \N$ be given, and suppose that%
		\begin{equation*}
		\forall ^\mathrm{st fin}z\subseteq Y\exists a\in 
		%TCIMACRO{\U{211d} }%
		%BeginExpansion
		\mathbb{R}
		%EndExpansion
		(a\in I_{y}(n)\wedge a\neq u_{n}).
		\end{equation*}%
		By Idealization%
		\begin{equation*}
		\exists a\in 
		%TCIMACRO{\U{211d} }%
		%BeginExpansion
		\mathbb{R}
		%EndExpansion
		\forall ^{\st}y\in Y(a\in I_{y}(n)\wedge a\neq u_{n}).
		\end{equation*}%
		Then $(n,a)\in u$, in contradiction with the fact that $u$ is functional. We
		conclude that there exists a standard finite $z\subseteq Y$ such that 
		\begin{equation}
		u_{n}=\bigcap \limits_{y\in z}I_{y}(n).
		\label{Formule eindige doorsnijding functioneel}
		\end{equation}%
		For arbitrary standard finite $z\subseteq Y$ we define 
		\begin{equation*}
		w_{z}=\left \{ \left(n,\bigcap \limits_{y\in z}I_{y}(n)\right): n\in 
		\N		\wedge \left \vert \bigcap \limits_{y\in z}I_{y}(n)\text{ }\right \vert =1\right \} .
		\end{equation*}%
		Also we define%
		\begin{equation*}
		w=\bigcup \left \{ w_{z}: z\subseteq Y\text{ standard finite}		\right \} .
		\end{equation*}%
		Let $n\in E$. By (\ref{Formule eindige doorsnijding functioneel})
		for some standard finite $z\subseteq Y$ we have $u_{n}=\bigcap _{y\in z}I_{y}(n)
		$. Then $w_{z}(n)=u_{n}$. Let $z^{\prime }\subseteq Y$ be also standard
		finite and such that $\left \vert \bigcap _{y\in z^{\prime }}I_{y}(n)\right \vert =1$%
		. Now $u_{n}=\bigcap _{\st(y)\in Y}I_{y}(n)\subseteq \bigcap _{y\in z^{\prime
		}}I_{y}(n)$, and then $u_{n}=\bigcap _{y\in z^{\prime }}I_{y}(n)$ because $ u $ is functional. Hence $%
		w_{z^{\prime }}(n)=\bigcap _{y\in z^{\prime }}I_{y}(n)=u_{n}=$ $w_{z}(n)$.
		Hence $w$ is functional, and $u\subseteq w$. Because $u$ is a halo and $w$ a
		pregalaxy, by the Fehrele Principle there exists an internal function $v$
		with $u\subseteq v\subseteq w$. Put $K=$ dom$(v)$. Then $E\subseteq
		K\subseteq 
		%TCIMACRO{\U{2115} }%
		%BeginExpansion
		\mathbb{N}
		%EndExpansion
		$ and $v_{\upharpoonright E}=u$. To finish the proof, let $J\subseteq 
		E$ be an internal set. Then $u_{\upharpoonright J}=v_{\upharpoonright J}$
		is internal.
	\end{proof}
	
	\begin{proof}[Proof of Theorem \ref{Representationsequence}] A strictly external sequence $ u $ defined on the internal set $ \mathbb{N} $ cannot be a galaxy, for then it would be internal by Lemma \ref{Lemma internal extension}. For the same reason it cannot be a halo. Hence it must have full complexity. By \eqref{uunionQ} it has the representation $ u=\bigcup _{\mathrm{st}(x)\in X}Q_{x}  $, where $Q_{x}$ is a functional prehalo, with domain $ \Delta _{x}  $, for all standard $x\in X$. By Lemma \ref{Lemma internal extension} and the Saturation Principle  there exists an internal family of internal sequences $ \left( v_{x}\right) _{x \in X} $ such that $ v_{{x}{\upharpoonright {\Delta_{x}}}}=Q_{x} $ for all standard $x\in X$. This proves \eqref{Formule uv}.
	\end{proof}
	
	We now define subsequences for external sequences, and consider properties of cofinality and coinitiality.
	
	\begin{definition}\label{Def subsequence-ad} Let $(u_n)$ be a flexible sequence. Let $ S\subseteq \N$ be cofinal with $ \N$. Let $v= \{(n, u_n): n\in S\}$. Then $ v $ is called a \emph{subsequence} of $(u_n)$. 
	\end{definition}
	
	If $(u_n)$ is an internal sequence and $ S $ is internal, Definition \ref{Def subsequence-ad} coincides with the conventional notion of subsequence. In this case a subsequence is again an internal sequence and we adopt the usual notation $(u_{m_n})$ to indicate a subsequence of $(u_n)$. 
	If $ S $ is external, we no longer have such a representation.
	
	Cofinality and coinitiality are notions related to ordered sets. With some abuse of language we define cofinality and coinitiality also for sequences. We will show that a definable sequence of external sets has an internal choice sequence. To that end we introduce the following notation.
	
	\begin{notation}
		Let $S\subset \N\times \R$. We denote by $\pi(S)$ the projection of $S$ into $\N$.
	\end{notation}
	
	\begin{definition}\label{cofinitial}
		Let $(C,D)$ be a cut of $\N
		$ into a nonempty initial segment $C$ and a final segment $D$. Let $%
		(u_n)$ be an external flexible sequence and $v$ a subsequence. The sequence $v$ is said to
		be \emph{cofinal} with $u$ on $C$ if 
		\begin{equation*}
		\forall m\in C \exists n\in C(n\geq m \wedge n\in dom(v));
		\end{equation*}
		 if $ D $ is nonempty, the sequence $v$  is said to be \emph{coinitial}
		with $u$ on $D$ if 
		\begin{equation*}
		\forall \nu \in D \exists \mu \in D(\mu \leq \nu \wedge \mu \in dom(v)).
		\end{equation*}
			\end{definition}

	\begin{proposition}\label{exists internal subsequence2} Let $V=\left\{(n, V_{n})  \right\} \subseteq \N \times \P(\R)$ be an external mapping with definable graph, where $ \pi (V)$ is cofinal with $ \N $. Then there exist an internal set $ S\subseteq \N $ cofinal with $ \N $ and an internal mapping $ v:S\rightarrow \R $ such that $ v_{n} \in V_{n} $ for all $ n \in S $.
	\end{proposition}
	
	\begin{proof}
		Let $X, Y$ be standard and $I:X\times Y \rightarrow  \mathcal{P}(\N\times \R)$ be an internal mapping such that $ V=\bigcup_{\st(x)\in X}\bigcap_{\st(y)\in Y} I_{xy}$. For all $ x\in X  $, let $H_x=\bigcap_{\st(y)\in Y} I_{xy}.$ Because $\pi (V)$ is cofinal with $ \N $, there exists $\st(x)\in X$ such that $H_x$ is cofinal with $ \N $. It follows that  $I_{xy}$ is cofinal with $ \N $ for all $\st(y)\in Y$. Hence  $$\forall^\mathrm{st fin} z\subseteq Y\, \exists J\, \forall^{\st}y\in z (J\subseteq I_{xy}), $$ where $J$ is cofinal with $ \N $.  In fact, we can take $J=\bigcap_{\st(y)\in z} I_{xy}$. By Idealization $$\exists J\ \forall^{\st} y\in Y (J\subseteq I_{xy}),$$ where $J$ is cofinal with $\N$. Put $ S=\pi(J) $. Because $J$ is internal, applying the Axiom of Choice, it has an internal selection $ v:S\rightarrow \R $, i.e. $ v_{n} \in J_{n} \subseteq V_{n} $ for all $ n \in S $.
	\end{proof} 	
	
	\begin{corollary}\label{exist internal subsequence}
		Every precise external sequence $(u_n)$ admits an internal subsequence.
	\end{corollary}

	We now turn to sequences which are defined on subsets of $ \N $ which are typically not cofinal with $ \N $.
	
	\begin{definition}
		Let $ C\subseteq \N $ be an initial segment. A \emph{local sequence} is an external mapping $ u:C \rightarrow \E $.
	\end{definition}

	A local sequence still admits an internal choice sequence, but by the Cauchy Principle the domain of this sequence must have nonempty intersection with the complement of $C$. Such a sequence will be called an internal bridge.
	 
	\begin{definition}\label{definition internal bridge}
		Let $ C\subseteq \N $ be an initial segment, not cofinal with $ \N $, let $ D= \N \setminus C $, let $ E\subseteq C $ be cofinal with $ C $ and $ u:E \rightarrow \E $ be a local flexible  sequence. Let $ K $ be an internal set which is cofinal with $ C $ and coinitial with $ D $. An internal sequence $ b:K\rightarrow \R $ such that  $ \{n \in K\cap C: b_{n}\in u_{n} \} $ is cofinal with C is called an \emph{internal bridge} from $ C $ to $ D $ along $ u $.
	\end{definition}	
		
	Taking $b$ as in Definition \ref{definition internal bridge}, we have that $ b_{\upharpoonright(K\cap C)} $ is a local sequence cofinal with $ C $ and $ b_{\upharpoonright(K\cap D)} $  is a local sequence coinitial with $ D $, according to Definition \ref{cofinitial}. Internal bridges are shown to exist in Theorem \ref{Stelling flexible brug}. We prove the theorem first for precise sequences.
	\begin{lemma}
		\label{Lemma brug}Let $(C,D)$ be a cut of $ \mathbb{N} $
		 into a halo $C$ and a galaxy $D$. Let $ E\subseteq \mathbb{N} $ be an external set which is cofinal with $C$ and $u:E\rightarrow \mathbb{R}	$ be a local precise sequence. Then $u$ admits an internal bridge from $C$ to $D$.
	\end{lemma}
\begin{proof}
	We have $u=\bigcup _{\mathrm{st}(x)\in X}Q_{x}$, where for all standard $x\in X$ the external sets $Q_{x}$ are functional, and are prehalos, like their domains $ \Delta _{x} $.
	
		By Lemma \ref{cofinal sets} there exists a standard $\overline{x} \in X$ such that $\Delta \overline{_{x}}$ is cofinal with $C$. Then $u_{\upharpoonright \Delta \overline{_{x}}}$ is
	a prehalo and by Lemma \ref{Lemma internal extension} there exists $K$ with $%
	\Delta \overline{_{x}}\subset K\subseteq \N
	$ and an internal sequence $b:K\rightarrow \R
	$ such that $b_{\upharpoonright \Delta \overline{_{x}}}=u$. Also $K\cap D$ is
	coinitial with $D$. Then $b$ is an internal bridge from $C$ to $D$ along $u$.
\end{proof}
	\begin{theorem}
		\label{Stelling flexible brug}Let $(C,D)$ be a cut of $%
		%TCIMACRO{\U{2115} }%
		%BeginExpansion
		\mathbb{N}
		%EndExpansion
		$ into a nonempty initial segment $C$ and a nonempty final segment $D$. Let $u:C \rightarrow \mathbb{E}$ be a local flexible sequence. Then $u$ admits an internal bridge from $C$ to $D$.
		\end{theorem}
\begin{proof}
	If $C$ is internal, it has a maximal element $m$. Choose $ b_{m} \in u_{m}$. We may define $ b_{m+1} \in \R $ arbitrarily. Then $%
	\{(m,b_{m}),(m+1,b_{m+1})\}$ is an internal bridge from $C$ to $D$ along $	u $.
	
	From now on we assume that $C$ is external. If $ C $ is a galaxy, it has the form
	$C=\bigcup _{\mathrm{st}(z)\in Z}\{0,...,c(z)\}$, where $ Z $ is a standard set and $ c:Z\rightarrow \mathbb{N} $ is internal. For all standard $ z\in Z $ there exists $ v\in u_{c(z)} $, so by the Saturation Principle we obtain an internal function $ \widetilde{v }:Z\rightarrow \mathbb{R} $ such that $ \widetilde{v }(z)\in u_{c(z)} $ for all standard $ z\in Z $. 
	
	For $n\in \Ima(c_{\upharpoonright {}^{\sigma }Z})$ we
	define $b_n=u_{c(z)}$, where $z$ is standard and such that $%
	c(z)=n$; then $b$ is a well-defined function on $\Ima%
	(c_{\upharpoonright {}^{\sigma }Z})$. By the Saturation
	Principle $b$ can be extended to an internal function $b:K\mathbb{%
		\rightarrow }%
	%TCIMACRO{\U{211d} }%
	%BeginExpansion
	\mathbb{R}
	%EndExpansion
	$. Then $\left \{ n\in K: b_{n}\in u_{n}\right \} $ is cofinal
	in $C$. Hence $b$ is an internal bridge from $C$ to $D$ along $u$.

	In the remaining case $ C $ is a halo. We consider separately the case where $ u $ has a cofinal sequence of galaxies, and the case where from some $ c\in C $ onwards all neutrices are prehalos. We have the representation
	\begin{equation*}
		\Gamma(u)=\bigcup _{\mathrm{st}(x)\in X}\bigcap_{\mathrm{st}(y)\in Y} I_{xy}(n), 
	\end{equation*}
	where $X,Y $ are standard sets and $ I:X\times Y\rightarrow \mathcal{P}  (\mathbb{N}\times\mathbb{R})  $ is internal; we may assume that $ I $ is closed under finite intersections.
	
	In the first case, let 
	\begin{equation*}
	F= \left\lbrace n \in C : u_{n} \text{ is a galaxy} \right\rbrace.
	\end{equation*}
		Observe that $ F $, and therefore also $ u_{\upharpoonright F} $, is definable, because, due to \eqref{Uxyn}, for all $ n \in F $  there exists a standard $ y\in Y $ such that
	\begin{equation*}
	u_n=\bigcup _{\mathrm{st}(x)\in X}I_{xy}(n).   
	\end{equation*}
	By Lemma \ref{cofinal sets} there exists $ \overline{y}\in Y $ such that $ \pi( I_{x\overline{y}}) $ is cofinal in the halo $ C $.
	For all standard $x\in X $ we define
	\begin{equation*}
	C_{x}=\left\lbrace n\in F: I_{x\overline{y}}(n)\subset u_n \right\rbrace .   
	\end{equation*}
	Again by Lemma \ref{cofinal sets} there exists a standard $ \overline{x}\in X $ such that $ C_{\overline{x}} $ is cofinal in $ C $. Let  $ M_{x}:C_{\overline{x}}\rightarrow \mathbb{R} $ be defined by
	\begin{equation*}
	M_{\overline{x}}(n) =	\max  I_{\overline{x}\overline{y}}(n)			
		\end{equation*}
		Then $ 	M_{\overline{x}}(n) \in u_{n}$ for all $n\in C_{\overline{x}} $. By Lemma \ref{Lemma brug} there exists an internal bridge $ b $ from $ C $ to $ D $ along $ M $. Then $ b $ is also an internal bridge from $ C $ to $ D $ along $ u $.		
		
		In the second subcase we may assume without restriction of generality that all terms $ u_{n} $ are prehalos. Hence $ u $ is a halo, and has the representation
		\begin{equation*}
	u=\bigcap_{\mathrm{st}(y)\in Y} J_{y},   
		\end{equation*}
		where $Y $ is a standard set and $ J:Y\rightarrow  \mathcal{P}(\mathbb{N}\times \mathbb{R})  $ is internal; we may assume that $ J $ is closed under finite intersections. By the Axiom of Choice, for all standard $ y \in Y$ there exists an internal function $ h:\mathbb{N} \rightarrow \mathbb{R}$ such that $ h_{\upharpoonright \pi (J_{y})} $ is a choice function for $ J_{y} $. Using the finite intersection property, we obtain by Idealization an internal function $ b:\mathbb{N} \rightarrow \mathbb{R}$ such that, whenever $y \in Y$
is standard and $n \in \pi(J_y)$, we have $b_n \in J_y$. This means that	$ b_{\upharpoonright \pi (\bigcap_{\mathrm{st}(y)\in Y} J_{y})} $ is a choice function for $ \bigcap_{\mathrm{st}(y)\in Y} J_{y}=u $. Then $ b $ is an internal bridge from $ C $ to $ D $ along $ u $.
\end{proof}

\section{Notions of convergence for flexible sequences\label{Section e-Sequences}}

In this section we introduce two types of convergence for sequences of external numbers which may be seen as ``outer'' ($ \varepsilon-n_0 $ like) and ``inner''(entrance into the limit set) forms of convergence.  The main results of this section say that except for the case where the sequence is infinite and the limit is precise, outer convergence induces inner convergence. Since outer convergence requires a notion of ordering we consider definitions and some properties of ordering for external numbers.

\subsection{Ordering}\label{Section Order relations}

	\begin{definition}[\cite{koudjetivandenberg,koudjetithese}]
		\label{Def order relations}Let $\alpha=a+A,\beta=b+B \in  \E$. We define (with abuse of notation)
		
		\begin{enumerate}
			\item \label{Def >=}$\alpha \geq \beta \Leftrightarrow \forall x\in \alpha
			\exists y\in \beta \left( x\geq y\right)$.
			
			\item \label{Def >}$\alpha>\beta \Leftrightarrow \forall x\in \alpha \forall
			y\in \beta \left( x>y\right)$.
			
			\item \label{Def <=}$\alpha \leq \beta \Leftrightarrow \forall x\in \alpha
			\exists y\in \beta \left( x\leq y\right)$.
			
			\item \label{Def <}$\alpha<\beta \Leftrightarrow \forall x\in \alpha \forall
			y\in \beta \left( x<y\right)$.
		\end{enumerate}
	\end{definition}
	
	Let $\varepsilon$ be an infinitesimal. We have $1+\varepsilon \pounds >\oslash $, $\varepsilon \leq \oslash $ and $\varepsilon \geq \oslash $.	
	Clearly $\alpha<\beta$ if and only if $\beta>\alpha$. However, we note that $%
	\alpha \geq \beta$ is not equivalent to $\beta \leq \alpha$. For example, we
	have $\oslash \geq \pounds $, yet $\pounds \not \leq \oslash$. Moreover, we have $\oslash \leq \pounds $. Indeed, $\alpha \leq \beta$ and $\alpha \geq
	\beta$ can occur simultaneously without equality if $\alpha \subset \beta$.
	Let $\alpha,\beta \in \E$. It follows from the definition that $\alpha \leq
	\beta$ if and only if $\alpha \geq-\beta$. However, $\alpha\not\leq \beta$ does not imply that $\alpha>\beta$ (nor $\beta< \alpha$). Indeed, it is enough to take for example $\alpha=\pounds$ and $\beta=\oslash$.

	The following characterizations of the order relations of Definition \ref{Def order relations} are a consequence of the the fact that two external numbers are either disjoint or one contains the other \cite[Prop. 7.4.1]{koudjetivandenberg}.
	%The following result gives a characteristic on order relations on external numbers. 
	\begin{proposition}
		\label{Lemma order relations}Let $\alpha,\beta \in \E$. Then
		
		\begin{enumerate}
			\item \label{OR1 <}$\alpha<\beta$ if and only if $\alpha \leq \beta \wedge
			\alpha \cap \beta=\emptyset.$
			
			\item \label{OR2 >}$\alpha>\beta$ if and only if $ \alpha \geq \beta \wedge
			\alpha \cap \beta=\emptyset.$
			
			\item \label{OR3 <=}$\alpha \leq \beta $ if and only if $ \alpha<\beta \vee
			\alpha \subseteq \beta.$
			
			\item \label{OR4 >=}$\alpha \geq \beta$ if and only if $ \alpha>\beta \vee
			\alpha \subseteq \beta.$
		\end{enumerate}
	\end{proposition}

	The order relation ``$\leq$'' is compatible with addition.

	\begin{proposition} \cite[p. 93]{koudjetithese} \label{baotaonthutu}
		Let $\alpha, \beta, \gamma \in \E$. If $\alpha\leq \beta$ then $\alpha+\gamma\leq \beta+\gamma$.
	\end{proposition}
	
	The following two results give conditions that allow to ``change'' the order relation.
	
	\begin{proposition}\label{change order}
		Let $M$ be a neutrix and let $\alpha$ and $\beta$ be two external numbers such that $N(\alpha)\subseteq M$ and $N(\beta)\subseteq M$. Then $\alpha+M\not \leq \beta+M$ implies that $\alpha+M>\beta+M$.
	\end{proposition} 
	
	\begin{proof}
		Observe that $N(\alpha+M)=N(\beta+M)=M$. If $\left(\alpha+M\right)\cap \left(\beta+M\right) \not=\emptyset$,  there exists $x\in \left(\alpha+M\right)\cap \left( \beta+M\right)$. It follows that  $\alpha+M=\beta+M=x+M$, in contradiction with the hypothesis. So $\left(\alpha+M\right)\cap \left(\beta+M\right)=\emptyset$. Since $\alpha+M<\beta+M$ cannot be, we conclude that $\alpha+M>\beta+M$.
	\end{proof}

	\begin{corollary}\label{change order2}
		Let $\alpha,\beta$ be two external numbers such that $N(\alpha)=N(\beta).$ If $\alpha\not\leq \beta$ then $\alpha>\beta.$
	\end{corollary}
	\begin{proof} It follows from Proposition \ref{change order} with $M=N(\alpha)=N(\beta)$. 
	\end{proof}

%	\begin{proposition}\label{Prop order} Let $\alpha=a+A, \beta=b+B, \gamma=c+C \in \E$. We have that	
%		\begin{enumerate}
%			\item \label{quanhethutubaotoan}  If	$\alpha-\beta <\gamma$, then $\alpha+N(\beta)<\beta+\gamma.$
%			\item \label{tinhchatbatdangthuc1} If	 $\gamma<\alpha-\beta$, then $\gamma+\beta< \alpha+N(\beta)$. In particular, $\gamma+\beta<\alpha$.
%			\item \label{tinhchatbatdangthuc2} If $\alpha-\beta <\gamma$ then $\alpha<\beta+\gamma$.
%		\end{enumerate}
%	\end{proposition}
	
%	\begin{proof} \ref{quanhethutubaotoan}. Assume $\alpha-\beta <\gamma$. If $N(\beta)\subseteq N(\gamma)$ we have $\beta+\gamma=c+b+C=b+ \gamma$. Then $\alpha+N(\beta)<\gamma+b=\gamma+\beta$.
%		If $N(\gamma)\subset N(\beta)$, then $N(\gamma+\beta)=B\subseteq A+B=N(\alpha-\beta)$. We prove first that $\alpha+N(\beta)\cap \gamma+\beta=\emptyset$. Suppose that there exists $x\in \alpha+N(\beta)\cap \gamma+\beta$. Then $\alpha+N(\beta)=x+A+B$ and $\beta+\gamma=x+B.$ It follows that $\beta+\gamma\subseteq \alpha+N(\beta)$. In particular, $b+c\in \alpha+N(\beta)$ which implies that $c\in \alpha-\beta<\gamma$, a contradiction. Also, we have $a\in \alpha+N(\beta), b+c\in \beta+\gamma$ and $a<b+c$ by the assumption. Hence $\alpha+N(\beta)< \beta+\gamma$.
%		
%		\ref{tinhchatbatdangthuc1}. The proof is analogous to the proof of Proposition \ref{Prop order}.\ref{quanhethutubaotoan}.
%		
%		\ref{tinhchatbatdangthuc2}. It follows from Proposition \ref{Prop order}.\ref{quanhethutubaotoan} that $\alpha \leq \alpha+N(\beta)<\beta+\gamma$.
%	\end{proof}

	\begin{definition}\label{trituyetdoidinhnghia}
		Let $\alpha=a+A$ be an external number. The \emph{absolute value} of $\alpha$ is defined (with abuse of notation) as follows $$|\alpha|= |a|+A.$$
	\end{definition}
	
	Let $N$ be a neutrix. Then $|N|=N$. Let $\alpha=a+A$. Then clearly, $|\alpha|=\alpha$, if $a>0$ and $|\alpha|=-a+A=-\alpha$, if $a<0$. 
	In the final part of this subsection we study some basic properties of the absolute value.
	
	\begin{proposition}\label{expliciteformulaOfAbsoluteexternalnumber} 
		If $\alpha=a+A \in \E$ is zeroless, then $|\alpha|=\left\{|x| : x\in \alpha\right\}.$
	\end{proposition}
	
	\begin{proof} 
		Let $\xi=\left\{|x| : x\in \alpha\right\}.$ We give the proof for the case where $a>A$. The case where $a<A$ is analogous.  
		Let $x\in |\alpha|$. Then $x=|a|+u$, for some $u\in A$. Since  $a>A$, we have $a+u \in \alpha$ and $x=|a|+u=a+u=|a+u|$ and we conclude that $x\in \xi$. Hence $|\alpha|\subseteq \xi$. On the other hand, let $y\in \xi$. Since $a>A$ we have $y=|a+v|=a+v=|a|+v$, for some $v \in A$. So, $y\in |\alpha|$ and therefore $\xi\subseteq |\alpha|$. Hence $\xi=|\alpha|$.
	\end{proof}
	
	Note that Proposition \ref{expliciteformulaOfAbsoluteexternalnumber} fails to be true if $\alpha$ is a neutrix. 
	
	\begin{corollary}
		Let $\alpha=a+A$ be a zeroless external number. The definition of absolute value of $\alpha$ does not depend on the choice of the representative $a$ of $\alpha.$	
	\end{corollary}
	
	\begin{proof} 
		Assume that $\alpha=b+A$. Then $|\alpha|=|b|+A=\{|x|: x\in \alpha\}=|a|+A$.
	\end{proof} 
	
		Let $\alpha=a+A, \beta=b+B \in \E$. The following triangular inequalities follow directly from the definition of absolute value. We will use them implicitly throughout the article.
		\begin{equation}
			\label{DT1} |\alpha +\beta|\le |\alpha|+|\beta|.
		\end{equation}	
			
		\begin{equation}	
			\label{DT2} ||\alpha|-|\beta||\le |\alpha|-|\beta|.
		\end{equation}

	%	\begin{proof}
	%	\ref{DT1}. We have $|\alpha+\beta|=|a+A+b+B|=|a+b|+A+B\le |a|+|b|+A+B=|\alpha|+|\beta|$.
	%	
	%	\ref{DT2}. We have $||\alpha|-|\beta||=||a|+A-|b|+B|=||a|-|b||+A+B\le |a|-|b|+A+B=|\alpha|-|\beta|$.
	%	\end{proof}

	\begin{proposition}\label{phanchiathanhhaibatdangthuc} 
		Let $\alpha \in \E$ and $\beta$ be a zeroless, positive external number. Then  $|\alpha|<\beta$ if and only if  $-\beta< \alpha<\beta$. 
	\end{proposition}
	\begin{proof}
		We start with the direct implication. Assume first that $\alpha$ is a neutrix. Then $\alpha=|\alpha|<\beta$. Suppose that $-\beta \not<\alpha$. Then there exists $u\in \alpha$ and $v\in -\beta$ such that $u \leq v$. Observe that $-\beta$ is negative, so $v\leq 0$. Since $\alpha$ is a neutrix and $u,0\in \alpha$ we have that $v\in [u, 0]\subseteq \alpha$. It follows that $-v\in \alpha$. Then $-v \in \alpha\cap \beta$, a contradiction since $|\alpha|<\beta$ implies that $\alpha\cap \beta=\emptyset$. Hence $-\beta<\alpha<\beta$. Assume secondly that $\alpha$ is zeroless. Let $u\in \alpha$ and $v\in \beta$. Then $|u|\in |\alpha|$ by Proposition \ref{expliciteformulaOfAbsoluteexternalnumber}. So $|u|<v$, 
		%that is, $-v<u< v$ for all $u\in \alpha$ and $v\in -\beta$ we have $v<u$. 
		hence $-\beta<\alpha.$ Also $\alpha\leq |\alpha|<\beta$. Thus, $-\beta<\alpha<\beta.$
		
		To prove the converse implication assume that $-\beta<\alpha<\beta.$ If $\alpha$ is a neutrix then $|\alpha|=\alpha<\beta$. If $\alpha$ is zeroless, let $y\in |\alpha|$ and $v\in \beta$. Then $y=|u|$ for some $u\in \alpha$ by Proposition \ref{expliciteformulaOfAbsoluteexternalnumber}. Since  $-\beta<\alpha<\beta$, it holds that $-v<u<v$, so $y=|u|<v$. 
		%Since $y\in |\alpha|, v\in \beta$ are arbitrary, 
		We conclude that $|\alpha|< \beta$.
	\end{proof}

\subsection{Convergence for infinite sequences\label{Subsection infinite sequences}}

\begin{definition}\label{Definite convergentie}
	Let $N$ be a neutrix. We say that a flexible sequence $u{:}$ $\mathbb{N}\rightarrow \mathbb{E}
	$ $N$-\emph{converges} to $\alpha \in \mathbb{E}$ if 
	\begin{equation}
	\forall \varepsilon>N\exists n_{0}\in \mathbb{N}\forall n\in \mathbb{N}(n\geq
	n_{0}\Rightarrow \left \vert u_{n}-\alpha \right \vert <\varepsilon).
	\label{N-conv}
	\end{equation}
	We write $u_{n}\unright \alpha$ or $N$-$ \lim u_{n}=\alpha$.  If $ N=N(\alpha) $, we simply say that $ u $ converges to $ \alpha $ and write  $u_{n}\longrightarrow \alpha$ or $ \lim u_{n}=\alpha$. A flexible sequence which is not $N$-convergent to any element $\alpha\in \E$ is called $N$-divergent.
\end{definition}

\begin{definition}\label{def strong convergence}
	Let $(u_n)$ be a flexible sequence and $\alpha$ be an external number.  The sequence $(u_n)$ is said to be \emph{strongly convergent to $\alpha$} if 
	\begin{equation*}
	\exists n_0 \in \N \forall n \in \N (n \geq n_0 \Rightarrow u_n \subseteq \alpha)
	\end{equation*}
	We write  $u_n  \rightsquigarrow  \alpha$ or $\slim  u_n=\alpha$. The external number $\alpha$ is called a \emph{strong limit} of $(u_n)$. 
\end{definition}

\begin{remark}
Clearly every flexible sequence is $\mathbb{R}$-convergent. To avoid this trivial case, whenever we refer to $%
N$-convergence we always assume that $N\neq \mathbb{R}$. 
Similarly, every flexible sequence strongly converges to $\mathbb{R}$, so we consider strong limits to be different from $\mathbb{R}$.
The other extreme case for $N$-convergence is when $N=0$. This corresponds to the usual notion of convergence. In
this case we adopt the usual notation and terminology, i.e. we say that $%
\left( u_{n}\right) $ converges to $\alpha $ and write $u_{n}\longrightarrow
\alpha $ or $\lim u_{n}=\alpha $. 
\end{remark}

\begin{remark}\label{equiv flexible}
	In the condition of $N$-convergence one may
	assume the elements $\varepsilon $ to be \emph{precise}, i.e. to be real
	numbers. In fact, (\ref{N-conv}) is equivalent to:%
	\begin{equation*}
	\forall \varepsilon  \in \mathbb{R}\left( \varepsilon >N\Rightarrow \exists n_{0}\in 
	\mathbb{N}(n\geq n_{0}\Rightarrow \left \vert u_{n}-\alpha \right \vert
	<\varepsilon )\right) . 
	\end{equation*}
	
	This equivalence often allows for simplifications in proofs and in calculations.
\end{remark}

Definition \ref{Definite convergentie} generalizes the usual $ \varepsilon-n_{0} $ criterium for convergence. Here it is asked that for $n \geq n_0$ the terms $ u_{n} $ enter in some interval, which is partly outside the limit set, justifying the term outer convergence. In the case of strong convergence the terms of the sequence reach the limit set within a finite time, leading to a form of inner convergence indeed.  
Let us illustrate Definitions \ref{Definite convergentie} and \ref{def strong convergence} with some examples. The examples also motivate some of the results proved throughout the section.

The first example illustrates the fact that $%
N$-convergent flexible sequences do not have a unique limit. In fact, as shown in Proposition \ref{tinhduynhatcuagioihansailechneutrix}, $N$-convergence is unable to distinguish between elements whose distance is less than or equal to the neutrix $N$. That is, $N$-limits are unique modulo $N$ in the sense that if $(u_n) $ is $N$-convergent to two (possibly distinct) elements $\alpha,\beta \in \E$ then the absolute	value of their difference must be less than $N$, in particular if $u_n \underset{N }\longrightarrow \alpha$, then $u_n \underset{N }\longrightarrow a$, for all $a \in \alpha$ (see Proposition \ref{converComponents}). 

\begin{example} \label{Example not unicity} Consider the flexible sequence defined by $u_{n}=\frac{1}{n}+\oslash $ and let $\delta \simeq 0 $. One
	has both $u_{n} \underset{\oslash }{\longrightarrow} \delta  $ and $u_{n} {\longrightarrow}  \oslash $ because $|u_n-\delta|=|u_n-\oslash|=\left|\frac{1}{n}+ \oslash\right|$.   
\end{example}

Next example illustrates the fact that every sequence is $N$-convergent for some neutrix $N$.
Clearly, if it is possible to approximate a certain quantity with a given imprecision, then it must also be possible to approximate it with a bigger imprecision. Adapting this intuition to flexible sequences we show in Proposition \ref{suhoitucuaneutrixlonhon} that if a flexible sequence is convergent with respect to a given neutrix, then it is convergent with respect to every neutrix which is larger than the original neutrix. Also, there exists a smallest neutrix such that the given sequence is convergent with respect to it. This means that when discussing $N$-convergence we will implicitly be looking for the smallest neutrix $N$ for which the sequence is $N$-convergent.

\begin{example}  \rm Consider the sequence defined by $u_n=(-1)^n$ for all $n\in \N$. Clearly $(u_n)$ is $\pounds$-convergent. We prove that $(u_n)$ is $\oslash$-divergent and since there are no neutrices between $\oslash$ and $\pounds$, we have that $N=\pounds$ is the smallest neutrix that gives $N$-convergence for $(u_n)$. Indeed, suppose that $u_n \underset{\oslash}{\longrightarrow} \alpha$ with $\alpha=a+A$ an external number. We have $N(\alpha)=A\leq \oslash.$ If $\alpha=1+A\subseteq  1+\oslash$, then taking $\varepsilon_0=1$ and choosing $n_0=2n+1>n$ we have $|u_{n_0}-\alpha|=|-1-1+A|=2+A>\varepsilon_0,$ a contradiction. If $\alpha \cap (1+\oslash)=\emptyset$, we take $\varepsilon_0=|1-a|/3>\oslash.$ Since $A\subseteq \oslash$, we have $-\varepsilon_0< A<\varepsilon_0.$ For all $n\in \N$, choosing $n_0=2n>n$ we obtain that $$|u_n-\alpha|=|1-a+A|=|1-a|+A =3\varepsilon_0-\varepsilon_0=2\varepsilon_0>\varepsilon_0,$$ a contradiction. Hence $(u_n)$ is $\oslash$-divergent.	
\end{example}

\begin{proposition} \label{suhoitucuaneutrixlonhon}
	Let $N$ be a neutrix and $(u_n)$ be a flexible sequence that
	$N$-converges to some element $\alpha  \in \E$. Let
	$M$ be a neutrix such that $N\subseteq M$. Then $(u_n) $
	also $M$-converges to $\alpha$. Moreover, there exists a unique neutrix $M_0$ such that $(u_n)$ is $M_0$-convergent and $M_0 \subseteq N$ for every neutrix $N$ such that $(u_n)$ is $N$-convergent. 
\end{proposition}
To prove the previous proposition we recall that every neutrix can be represented as the product of a positive real number and an \emph{idempotent} neutrix, i.e. a neutrix $I$ such that $II=I$. 

\begin{theorem} [\cite{koudjetithese, koudjetivandenberg}]\label{canonical form} 
	Let $A$ be a neutrix. Then there exists a real positive number $t$ and a unique idempotent neutrix $I$ such that $A=t\cdot I.$
\end{theorem}

\begin{proof}[Proof of Proposition \ref{suhoitucuaneutrixlonhon}] Let $\varepsilon >M$. Then $\varepsilon >N.$ Since $(u_n)$ $N$-converges to $\alpha$, there exists $n_0\in \N$ such that  $|u_n-\alpha|<\varepsilon$ for all $n\geq n_0$. Hence $(u_n)$ $M$-converges to $\alpha$.
	Without loss of generality, we may assume that $(u_n)$ is $N$-convergent to a neutrix. Let $L=\{M : M \textup{ is a neutrix and } (u_n) \mbox{ is $M$-convergent}\}$ and $M_0$ be the infimum of $L$.   By Theorem \ref{canonical form} there exist a positive real number $p$ and a unique idempotent neutrix such that $M_0=p\cdot I$. 
	
	Suppose first that $M_0=p\cdot \oslash$ for some positive real number $p$. If $M_0 \notin L$, all elements $ M $ of $ L $ satisfy $ L\supseteq p\cdot  \pounds$, hence $M_0$ cannot be the infimum of $L$, a contradiction. Hence $M_0 \in L$ and so $M_0 \leq N$ for every neutrix $N$ such that $(u_n)$ is $N$-convergent.
	
	Suppose now that $M_0$ is not a multiple of $\oslash$. Suppose that the sequence is not $M_0$-convergent. Let $\varepsilon >M_0$. Note that  $M_0 \subseteq \varepsilon \oslash$, for $\varepsilon \oslash$ is the biggest neutrix less than $ \oslash $, hence $M_0 \subset \varepsilon \oslash$, for $M_0 $ is not isomorphic to $\varepsilon \oslash$. For all $n\in \N$ there exists a element $m\geq n$ such that $\varepsilon \leq u_m$.  Because $\varepsilon \oslash<\varepsilon$, the sequence $u_n$ is not $\varepsilon \oslash$-convergent. So $\varepsilon \oslash\not \in L$. By the first part, for every neutrix $M\in L$ we have $\varepsilon \oslash\subset M$. It follows that $\varepsilon \oslash\subseteq M_0$, a contradiction. Hence $u_n$ is $M_0$-convergent.   
\end{proof}

As already suggested by Proposition \ref{suhoitucuaneutrixlonhon}, the following proposition entails that, concerning $N$-convergence, it is possible to neglect quantities in the terms which are contained in $ N$. 
\begin{proposition}
	\label{equiv e-conv}Let $M,N$ be neutrices such that $M\subseteq N$, let $\alpha \in \E$ and let $%
	\left( u_{n}\right) $ be a flexible sequence. Then the following are
	equivalent
	
	\begin{enumerate}
		\item \label{equiv1}$\left( u_{n}+M\right) \unright \alpha$
		\item \label{equiv3}$u_{n}\unright \alpha$
	\end{enumerate}
\end{proposition}
\begin{proof}
	Assume $\left( u_{n}+M\right)
	\unright \alpha$. Then there exists $n_{0}\in \mathbb{N}$ such
	that for $n\geq n_{0}$ we have $\left \vert \left( u_{n}+M\right) -\alpha
	\right \vert <\varepsilon$. Then, for $n\geq n_{0}$%
	\begin{equation*}
	\left \vert u_{n}-\alpha \right \vert \leq \left \vert \left( u_{n}+M\right)
	-\alpha \right \vert <\varepsilon.
	\end{equation*}
	Hence $u_{n}\unright \alpha$. Assume now that $u_{n}\unright \alpha$. Then there exists $n_{0}\in \mathbb{N}$ such
	that for $n\geq n_{0}$ we have $\left \vert u_{n}-\alpha \right \vert
	<\varepsilon/2$. Clearly $\varepsilon/2>N$. Then%
	\begin{equation*}
	\left \vert \left( u_{n}+M\right) -\alpha \right \vert \leq \left \vert
	\left( u_{n}+N\right) -\alpha \right \vert \leq \left \vert u_{n}-\alpha
	\right \vert +N<\frac{\varepsilon}{2}+\frac{\varepsilon}{2}=\varepsilon.
	\end{equation*}
	Therefore $(u_{n}+M) \unright \alpha$. Hence \eqref{equiv1} and \eqref{equiv3} are equivalent.
\end{proof}

As illustrated by Example \ref{Example not unicity}, $N$-convergence is unable to distinguish between elements whose
distance is less than or equal to the neutrix $N$. As a consequence, the $N$-limits of a flexible sequence are not unique. However, $N$-limits are unique modulo $N$ in the sense that if a flexible sequence $(u_n) $ is $N$-convergent to two (possibly distinct) elements $\alpha,\beta \in \E$ then the absolute value of their difference must be less than $N$. Moreover, if $(u_n)$ is $N$-convergent to some $\alpha \in \E$ then $N$ is the best possible, meaning that then $N(\alpha) \subseteq N$ and the largest limit set is $\alpha + N$.

\begin{proposition}\label{tinhduynhatcuagioihansailechneutrix}	 Let $u:  \N\longrightarrow \E$ be a flexible	sequence, $\alpha,\beta \in \E$ and $M$ be a neutrix such that $u_{n} \underset{M} {\longrightarrow} \alpha$
	and $u_{n} \underset{M}{\longrightarrow} \beta$. Then 
	\begin{enumerate}
		\item \label{tcntgh1} $N\left(  \alpha \right)  \subseteq  M$;
		\item \label{tcntgh2}$\left \vert \alpha-\beta	\right \vert \leq M$;
		\item \label{tcntgh3}$u_n  \underset{M}{\longrightarrow} \alpha+M$;
		\item \label{tcntgh4} $u_n  \underset{M}{\longrightarrow} \gamma$ if and only if $\gamma\subseteq \alpha+M$.
	\end{enumerate}
\end{proposition}

\begin{proof} 
	\ref{tcntgh1}. Suppose that $M\subset N(\alpha)$.  Let $\varepsilon \in \R$ be such that $M<\varepsilon  \leq N\left(
	\alpha \right)  $. Then there exists $n_{0}\in\N \,
	%EndExpansion
	$ such that for $n\geq n_{0}$ we have
	\[
	N\left(  \alpha \right)  \leq N\left(  \left \vert u_n-\alpha \right \vert
	\right)  \leq \left \vert u_n-\alpha \right \vert <\varepsilon \leq N\left(
	\alpha \right)  ,
	\]
	which is a contradiction. Hence  $N\left(  \alpha \right)  \subseteq  M$.
	
	\ref{tcntgh2}. Suppose $u_{n}\underset{M}{\longrightarrow} \alpha$ and $u_{n}  \underset
	{M}{\longrightarrow} \beta$ with $\left \vert \alpha-\beta \right \vert >M$. For 
	$\varepsilon_0 = \frac{\left \vert \alpha-\beta \right \vert }{2}>M$, there exists $n_{0}\in	\mathbb{N}
	%EndExpansion
	$ such that  $\left \vert u_{n}-\alpha \right \vert
	<\varepsilon_0$ for $n\geq n_{0}$ and there exists $n_{1}\in
	\mathbb{N}
	%EndExpansion
	$ such that  $\left \vert u_{n}-\beta \right \vert
	<\varepsilon_0$ for $n\geq n_{1}$. Let $k= \max \left \{  n_{0},n_{1}\right \}$.  Then for $n\geq k$,%
	\[
	\left \vert \alpha-\beta \right \vert \leq \left \vert \alpha-u_{n}+u_{n}%
	-\beta \right \vert \leq \left \vert \alpha-u_{n}\right \vert +\left \vert
	u_{n}-\beta \right \vert <2\varepsilon_0=\left \vert \alpha-\beta \right \vert \text{,}%
	\]
	a contradiction. Hence $\left \vert \alpha-\beta \right \vert \leq M$.
	
	\ref{tcntgh3}. 	Let $\varepsilon >M$. Then $\varepsilon /2>M/2=M$.  Because  $u_n \underset{M}{\longrightarrow} \alpha, $  there exists $n_0 \in \N$ such that  $|u_n-\alpha| < \varepsilon /2$ for all $n\geq n_0$. It follows that $$|u_n-\alpha+M|= |u_n-\alpha|+M< \frac {\varepsilon}{2}+\frac {\varepsilon}{2}=\varepsilon $$ for all $n\geq n_0$. Hence $u_n \underset{M}{\longrightarrow} \alpha+M$.
	
	\ref{tcntgh4}. If $\gamma\subseteq \alpha+M$ then   $|u_n-\gamma|\leq |u_n-\alpha+M|<\varepsilon $ for all $n\geq n_0,$ and hence $u_n	\underset{M}{\longrightarrow} \gamma$.
	Conversely, assume that $u_n	\underset{M}{\longrightarrow} \gamma$. Then by Proposition \ref{tinhduynhatcuagioihansailechneutrix}.\ref{tcntgh1} we have $N(\gamma) \subseteq M$. By Proposition \ref{tinhduynhatcuagioihansailechneutrix}.\ref{tcntgh2} it holds that $|\alpha -\gamma|\subseteq M$. So $\gamma - \alpha \subseteq M$. Then $\gamma \subseteq \gamma +N(\alpha) \subseteq \alpha + M$. 
\end{proof}

In Proposition \ref{tinhduynhatcuagioihansailechneutrix}.\ref{tcntgh2}, if $M=0$, then $\alpha=\beta$, i.e. we have uniqueness of the limit. This was to be expected because as mentioned above this case corresponds to the usual notion of convergence.

Let $(u_n)$ be a flexible sequence of the form $(a_n)+(A_n)$, where for each $n \in \N$ the element $ a_n $ is precise and $ A_n=N(u_n) $. Then we call the precise sequence $(a_n)$ a \emph{sequence of representatives}  of the sequence $(u_n)$. Let $N$ be a neutrix. In Proposition \ref{seq representatives well defined} we will show that a different choice for the sequence of representatives still leads to the same $N$-limit. First, we show that a flexible sequence $\left( u_{n}\right) $ is $N$-convergent to $\alpha $ if and only if both $(A_n)$ is $N$-convergent to the neutrix of $\alpha$ and a sequence of representatives of $(u_n)$ is $N$-convergent to an element of $ \alpha $.

\begin{proposition}
	\label{converComponents} Let $\left( u_{n}\right) $ be a flexible sequence such that for all $n\in \mathbb{N}$, $u_{n}=a_{n}+N(
	u_{n})$ where $(a_{n})$ is a sequence of representatives. Let $N$
	be a neutrix and let $\alpha=a+N(\alpha) \in \mathbb{E}$. Then $%
	u_{n}\unright \alpha$ if and only if $a_{n} \unright\ a$ and $N(u_{n}) \unright N(\alpha)$. In particular if $u_{n}\underset{N } {\longrightarrow} a$ and $N(\alpha) \leq N$, then $u_{n}\underset{N }{\longrightarrow} \alpha$.
\end{proposition}

\begin{proof}
	Assume first that $a_{n}\unright a$ and $N( u_{n})
	\unright N(\alpha) $. Let $\varepsilon>N$. Then
	there exist $n_{1},n_{2}\in \mathbb{N}$ such that $%
	\left \vert a_{n}-a\right \vert <\varepsilon/2$, for $n\geq n_{1}$ and  $%
	\left \vert N(u_{n}) +N(\alpha) \right \vert
	=\left
	\vert N(u_{n}) -N(\alpha) \right \vert
	<\varepsilon/2$ for $n\geq n_{2}$. Let $n_{0}=\max \left \{ n_{1},n_{2}\right \} $. Then for $%
	n\geq n_{0}$ we have $$\left \vert u_{n}-\alpha \right \vert =\left \vert
	a_{n}+N(u_{n}) -a+N(\alpha) \right \vert \leq
	\left
	\vert a_{n}-a\right \vert +\left \vert N(u_{n}) -N(\alpha) \right \vert <\frac {\varepsilon}{2}+\frac {\varepsilon}{2}=\varepsilon.$$ Hence $u_{n}\unright \ \alpha$.
	
	Conversely, assume that $u_{n}\unright \alpha$. Let $\varepsilon>N$. Then
	there exist $n_{1},n_{2}\in \mathbb{N}$ such that  $%
	\left \vert u_{n}-\alpha \right \vert <\varepsilon$ for $n\geq n_{1}$ and  $\left \vert u_{n}-\alpha \right \vert <\varepsilon/2$ for $n\geq n_2$. Then, for $n \geq n_1$,%
	\begin{equation*}
	\left \vert a_{n}-a\right \vert \leq \left \vert a_{n}+N(u_{n})
	-a+N(\alpha) \right \vert =\left \vert u_{n}-\alpha \right
	\vert <\varepsilon
	\end{equation*}
	and, for $n \geq n_2$, 
	\begin{equation*}
	\left \vert N(u_{n}) -N(\alpha) \right \vert =\left
	\vert u_{n}-u_{n}+\alpha-\alpha \right \vert \leq \left \vert u_{n}-\alpha
	\right \vert +\left \vert u_{n}-\alpha \right \vert <\frac{\varepsilon}{2}+%
	\frac{\varepsilon}{2}=\varepsilon \text{.}
	\end{equation*}
	
	\noindent Hence $a_{n}\unright a$ and $N(u_{n}) \unright N(\alpha) $.
\end{proof}

\begin{proposition} \label{seq representatives well defined}
	Let $(u_n)$ be a flexible sequence. Let $(a_n),(b_n)$ be precise sequences and $(A_n)$ be a sequence of neutrices such that for all $n \in \N$ it holds that $u_n=a_n+A_n$ and $u_n=b_n+A_n$. Let $\alpha=a+A$ be an external number and $N$ be a neutrix such that $u_n \unright \alpha$. Then $a_n \unright a$ if and only if $b_n \unright a$.  
\end{proposition}

\begin{proof}
	Suppose that $a_n \unright a$. Let $\varepsilon >N$. Then there exists $n_1 \in \N$ such that for $n \geq n_1$ it holds that $|a_n-a|<\varepsilon/2$. Also, since $u_n \unright \alpha$, by Proposition \ref{converComponents} it holds that $A_n \unright A$, so there exists $n_2 \in \N$ such that for $n \geq n_2$ it holds that $|A_n-A|<\varepsilon/2$. Let $n_0=\max\{n_1,n_2 \}$. Observe that for each $n \in \N$ we have $b_n-a_n \in A_n$ because $a_n$ and $b_n$ are both representatives of the same external number. Then for $n \geq n_0$ it holds that 
	\begin{alignat*}{2}
	|b_n-a|=&|b_n-a_n+a_n-a|\leq |b_n-a_n|+|a_n-a|<A_n+\frac{\varepsilon}{2} \\
	\leq & |An-A|+ \frac{\varepsilon}{2} 	< \frac{\varepsilon}{2}+\frac{\varepsilon}{2}=\varepsilon.
	\end{alignat*}
	
	Hence, $b_n \unright a$.
\end{proof}

The absolute value preserves $N$-convergence.
\begin{proposition}
	\label{Proposition abs conv}Let $(u_{n})$ be a flexible sequence
	and $N$ be a neutrix. If $u_{n}\unright \alpha$ then $%
	\left \vert u_{n}\right \vert \unright \left \vert \alpha
	\right \vert $.
\end{proposition}

\begin{proof} By \eqref{DT2} we have that $\left \vert \left \vert u_{n}\right \vert -\left \vert \alpha \right \vert
	\right \vert \leq \left \vert u_{n}-\alpha \right \vert$. The result is an immediate consequence of this fact.
\end{proof}

Let $N$ be a neutrix. By the results above most cases of $N$-convergence to an external number $\alpha$ reduce to the case where $N=N(\alpha)$, so we can talk about convergence instead of $N$-convergence and write simply $u_n \longrightarrow \alpha$ as indicated in Definition \ref{Definite convergentie}. 
Next theorem states that convergence implies strong convergence except in the case where the limit is precise. The proof will be given in Section \ref{subsequences}.

\begin{theorem}	\label{Stelling sterke convergentie onbegrensd}
	Let $(u_n)$ be a flexible sequence and $\alpha$ be an external number with $N(\alpha)\neq 0.$ Then $u_n{\longrightarrow} \alpha$ if and only if  $u_n \rightsquigarrow \alpha$.
\end{theorem}

A well-known theorem of nonstandard analysis says that a
standard sequence $(u_n)$ which converges to a limit $a$ satisfies $u_{n}\simeq
a$ for all unlimited $n\in \mathbb{N}$. Reformulated in the sense of Definitions \ref{Definite convergentie} and \ref{def strong convergence}, this property implies that if the standard sequence $(u_n)$ is
$\oslash$-convergent to $a+\oslash $, then it converges strongly to $a+\oslash $. The property is of interest, for instance because operations with classical convergence 
involve epsilontics, and operations with the nonstandard criterium of convergence use simple
algebraic properties of infinitesimal numbers. Theorem \ref{Stelling sterke convergentie onbegrensd} extends this property to all external definable sequences and all external numbers. 
Again the theorem is of interest in relation to behavior under operations: operations with $N$-convergence involve epsilontics, and operations with strong convergence may profit from the algebraic properties of external numbers \cite{dinisberg 2011} (see Section \ref{Section Operations Limits}).

\begin{example} \rm
	Consider again the flexible sequence $(u_n)$ defined by $u_n=\frac{1}{n}+\oslash$. Then $ u_n \rightsquigarrow \oslash.$ In fact $ u_n\subseteq \oslash $ for all unlimited $ n \in \N  $.
\end{example}

\subsection{Convergence with respect to an initial segment\label{Subsection initial segment}}

Theorem \ref{Stelling sterke convergentie onbegrensd} states that a flexible
external sequence $(u_n)$ which converges to $\alpha \in \mathbb{E}$ reaches $\alpha $ in a finite
time. Put%
\begin{equation*}
C=\{n\in \mathbb{N} : \exists p>n(u_{p}\notin \alpha ) \}.
\end{equation*}%
Then $C\subset 
%TCIMACRO{\U{2115} }%
%BeginExpansion
\mathbb{N}
%EndExpansion
$ is an initial segment of $%
%TCIMACRO{\U{2115} }%
%BeginExpansion
\mathbb{N}
%EndExpansion
$. If $C\neq \emptyset $, one may say that $(u_n)$ converges to $\alpha $
already with respect to $C$: afterwards the terms $u_{n}$ stay within the
limit set $\alpha $. This property remains true if one truncates the sequence up to an index $\nu
\in $ $%
\mathbb{N}
\setminus C$. So one may define convergence also for sequences with finitely many terms. This
gives also the liberty to consider convergence for initial parts of a
sequence, that afterwards may exhibit different behavior, a situation that
is quite natural in the context of dynamical systems (see Section \ref{Section applications}).

Below we will extend the notions of convergence of Definitions \ref{Definite
	convergentie} and \ref{def strong convergence} to convergence with respect to an initial segment of $
\mathbb{N}$, note that by the above we may as well define this notion for local sequences. Again convergence implies strong convergence, as stated in
Theorem \ref{Convergence initial segment}. However in some cases there is a difference
with infinite sequences, meaning that the strong convergence may not happen
on the initial segment, but just beyond, along internal bridges.

\begin{definition}
	\label{Definition convergence local}Let $S,C\subseteq 
	%TCIMACRO{\U{2115} }%
	%BeginExpansion
	\mathbb{N}
	%EndExpansion
	$ be initial segments of $%
	%TCIMACRO{\U{2115} }%
	%BeginExpansion
	\mathbb{N}
	%EndExpansion
	$ with $C\subseteq S$, and $D=S\setminus C$. Let $u:S \rightarrow %
	%TCIMACRO{\U{211d} }%
	%BeginExpansion
	\mathbb{R}
	%EndExpansion
	$ be an external local sequence. Let $\alpha =a+A\in \mathbb{E}$. We say that $u$ 
	\emph{converges} to $\alpha $ with respect to $C$ if 
	\begin{equation*}
	\forall \varepsilon >A\exists m\in C\forall n\in S (n\geq
	m\Rightarrow \left \vert u_{n}-\alpha \right \vert < \varepsilon ).
	\end{equation*}%
	We say that $u$ \emph{strongly converges} to $\alpha $ with respect
	to $C$ if there exists a nonempty final segment $T$ of $S$ such that $%
	D\subseteq T$ and 
	\begin{equation*}
	\forall n\in S(n\in T\Rightarrow u_{n}\in \alpha ).
	\end{equation*}%
	In the first case we write $\lim_{n\rightarrow D}u_{n}=\alpha $ and in the
	second case \textrm{Lim}$_{n\rightarrow D}u_{n}=\alpha $. 
\end{definition}

Definition \ref{Definition convergence local} states that the local
convergence has to happen on the initial segment $C$\ of the domain $S$\ of
the sequence, but in the remaining part of $S$\ the sequence may not
deviate. In the case of strong convergence, observe that if $C=S$, we have
already $u_{n}\in \alpha $\ on a final segment of $C$. If $C\subset S$, we
have $u_{n}\in \alpha $\ for all $n\in S\setminus C$. So the strong
convergence already happens on $C$, or immediately beyond, and must continue
on the whole of $S$.

The property that $(u_n)$ already strongly converges to $\alpha $ within the
initial segment $C$, or just beyond, depends on the nature of the external
sets $C$ and $A$. They can either be internal, a galaxy or a halo. If the strong
convergence happens just beyond $C$, either $C$ is a galaxy, while $A$ is a
halo, or $C$ is a halo, while $A$ is a galaxy.
This is illustrated by the next example.

\begin{example}\label{Voorbeelden begin}
		\emph{Let} $u:\N	\rightarrow\R		$\emph{\ be defined by }$u_{n}=1/n$\emph{. If }$C={}^{\sigma }%
		%TCIMACRO{\U{2115} }%
		%BeginExpansion
		\mathbb{N}
		%EndExpansion
		$\emph{\ and }$D=\centernot{\infty}\cap 
		%TCIMACRO{\U{2115} }%
		%BeginExpansion
		\mathbb{N}
		%EndExpansion
		$\emph{, we have }$\lim_{n\rightarrow D}u_{n}=\oslash $\emph{\ and Lim}$%
		_{n\rightarrow D}u_{n}=\oslash $\emph{, because }$u_{n}\simeq 0$\emph{\ for
			all }$n\simeq \infty $\emph{. Let }$\omega \in 
		%TCIMACRO{\U{2115} }%
		%BeginExpansion
		\mathbb{N}
		%EndExpansion
		$\emph{\ be unlimited. If }$C^{\prime }=\oslash \omega $\emph{\ and }$%
		D^{\prime }=%
		%TCIMACRO{\U{2115} }%
		%BeginExpansion
		\mathbb{N}
		%EndExpansion
		\setminus C^{\prime }$\emph{, we have }$\lim_{n\rightarrow D^{\prime
		}}u_{n}=$\emph{Lim}$_{n\rightarrow D^{\prime }}u_{n}=%
		\pounds /\omega $\emph{. The limits are not unique, because also }$%
		\lim_{n\rightarrow D^{\prime }}u_{n}=\emph{Lim}%
		_{n\rightarrow D^{\prime }}u_{n}=\oslash $\emph{. However }$\pounds /\omega $%
		\emph{\ is minimal with respect to the latter limits, for if }$\alpha \in 
		\mathbb{E}$ \emph{such that }$\alpha \subset $\emph{\ }$\pounds /\omega $\emph{\
			there exists }$\nu \in \N
		$\emph{\ with }$\alpha <1/\nu <$\emph{\ }$\pounds /\omega $\emph{, so }$%
		u_{\nu }\notin \alpha $ \emph{(see Proposition \ref{suhoitucuaneutrixlonhon}). }
		
	\emph{Restricting the sequence }$(u_n)$ 
		\emph{to }$%
		S=\{0,...,\omega \}$\emph{, with a similar argument one shows that the sequence strongly converges to $\oslash $ with respect to $%
		^{\sigma }\N
		$ and to $\pounds /\omega $ with respect to $\oslash \omega $.}
\end{example}
Next theorem resumes all possibilities of strong convergence of precise sequences with respect to an initial segment.
The proof will be given in Section \ref{subsequences}.
\begin{theorem}
	\label{Convergence initial segment} Let $S,C\subseteq \mathbb{N}$ be nonempty initial segments of 
	$	\mathbb{N}	$ with $C\subseteq S$. Let $ D =\mathbb{N}\setminus C$ be nonempty and $\alpha =a+A\in \mathbb{E}$. Let $u:S\rightarrow 
	\mathbb{R}	$ be a local flexible sequence which converges to $\alpha$ on $S$ with respect to $C$. 
	Then $ u $ converges strongly to $ \alpha $ with respect to $C$ if $C \subset S$ or if $C=S$ and $C$ and $A$ are both prehalos or both pregalaxies.  
	
	If $C$ is a galaxy, while $A$ is a halo, or $C$ is a halo, while $A$ is a galaxy, for every internal bridge $ b:K\rightarrow \mathbb{R} $ from $C$ to $ D $ along $ u $, with $ K\subseteq \N$, there exists $ \nu \in D $ such that $ b_{\upharpoonright{K\cap \{0,..., \nu \}}} $ is strongly convergent with respect to $ K\cap C$. 
	\end{theorem}

		\section{Proofs of the strong convergence theorems\label{subsequences}}
	
%In this section we always assume that the sequences are definable.
	
	\subsection{Strong convergence of infinite sequences\label{global sequences}}
	We prove here Theorem \ref{Stelling sterke convergentie onbegrensd}, saying that definable flexible convergent sequences are strongly convergent. To do this, we prove some preliminary properties on cofinality, extend the classical theorem on the existence of a convergent subsequence for bounded sequences to external precise sequences, and prove separately the strong convergence theorem for precise sequences and sequences of neutrices. 
	We start by relating convergence and divergence of a sequence to corresponding properties of subsequences. 
	
	\begin{proposition}\label{suhoitucuamoidaycon}
		Let $N$ be a neutrix, $\alpha \in \mathbb{E}$ and let $(u_{n})$ be a flexible sequence. Then $u_{n} \unright \alpha 
		$ if and only if $u_{n\upharpoonright_{J}} \unright  \alpha$ for every set $J$ cofinal with $\mathbb{N}$.
	\end{proposition}
	
	\begin{proof}
		Assume first that $(u_{n}) $ is $N$-convergent to $\alpha$. Let $J$ be cofinal with $\mathbb{N}$ and let $\varepsilon>N$%
		. Then there exists $n_{0}\in \mathbb{%
			N}$ such that for all $n\geq n_{0}$, we have $|u_{n}-\alpha|<\varepsilon.$ Because $J$ is cofinal with $\mathbb{N}$, there exists $m_0\in J$ such that $m_0\geq n_0$. So for all $m\in J$ such that $m\geq m_0\geq n_0$ we have $|u_{m_{n}}-\alpha|<\varepsilon$. We conclude	that $u_{n\upharpoonright_{J}} \unright  \alpha$. The other implication is	obvious because $\left( u_{n}\right) $ is a subsequence of itself.
	\end{proof}

	\begin{proposition}\label{bounded convergent}
		Every bounded precise sequence admits an internal convergent subsequence. 
	\end{proposition}
	\begin{proof} Let $(a_n)$ be a bounded precise sequence. 
		By Corollary \ref{exist internal subsequence} there exists an internal subsequence $(a_{n_m})$, of $(a_n)$ which is clearly also bounded. So there exists a convergent subsequence $(b_n)$ of $(a_{n_m})$, which is also a subsequence of $(a_n)$.
	\end{proof} 

\begin{remark} Let $\alpha=a+N$ and $(u_n)$ be a flexible sequence. Then $(u_n)$ is $ N $-convergent to $\alpha$ if and only if $(u_n-a)$ is $ N $-convergent to $N$. So when investigating $ N $-convergence, without restriction of generality we may consider $\alpha$ to be a neutrix. 
\end{remark}
	The next lemma is crucial in proving Theorem \ref{Stelling sterke convergentie onbegrensd}, and says that an internal sequence which remains outside a neutrix, does not have this neutrix as a limit.
	\begin{lemma}\label{hoitutrong} 
		Let $(a_n)$ be an internal precise sequence. Let $N\not=0$ be a neutrix and $n_0 \in \N$ be such that $a_n\not\in N$ for all $n \geq n_0$. Then  $(a_n)$ is not $N$-convergent to $N$. 
	\end{lemma}
	
	\begin{proof}  By Proposition \ref{Proposition abs conv} we may assume that $a_n>N$ for all $n\geq n_0$.  We consider two cases, (i) the sequence  $(a_n)$ is convergent to some $a\in \R$, (ii) the sequence $(a_n)$ is divergent.
		
		(i) We have $a\not\in N$. Indeed, suppose  that $a\in N$. Let $\varepsilon\in N$ be such that $\varepsilon>0$. Then there exists $n_1\in \N$ such that $|a_n-a|<\varepsilon$ for all $n\geq n_1$. So, $|a_n|<|a|+\varepsilon$ for all $n\geq n_2$, where $n_2=\max \{n_0,n_1\}$. This means that $|a_n|<|a|+\varepsilon\in N$. By our assumption and by the convexity of $N$ it follows that $a_n\in N $ for all $ n\geq n_2$, and in particular $a_{n_2}\in N$, which contradicts the assumption. 
		We prove now that $a>N$. Suppose towards a contradiction that $a< N$. Then $a<0$ and so $|a|=-a>N$. Let $\eta=|a|/2$. Clearly $|a|/2>N$. Then there exists $n_3\in \N$ such that $|a_n-a|<\eta$ for $n\geq n_3$. Let $n_4=\max \{ n_0,n_3\}$. This implies that  $a_n<a+\eta=a/2<0$  for all $n\geq n_4$, a contradiction. Hence $a>N$.
		Suppose  that $a_n \unright N$. Let $\varepsilon=a/2$. Clearly $\varepsilon>N$. Then there exists $n_5 \in \N$ such that for $n \geq n_5$ it holds that $|a_n-N|< \varepsilon$. This means that $a/2=a-\varepsilon<a_n<a +\varepsilon $ for all $ n\geq n_5$.
		Observe that since $a \notin N$ and $a/4>N$ then $-a/4<N$ and we have that $a/4=a/2-a/4<a/2+N$.		
		It follows that $|a_n-N|=a_n+N>a/2+N> a/4>N$ for all $ n\geq n_6$, where $n_6=\max \{n_0,n_5 \}$, which is a contradiction. Hence $N$ is not a $N$-limit of $(a_n)$.
		
			(ii) Suppose towards a contradiction that $a_n \unright N$. Since $(a_n)$ is an internal real sequence, it follows that $(a_n)$ is bounded. So there exists a subsequence $(a_{m_n})\subset (a_n)$ such that $(a_{m_n})$ has a limit $b\in \R$. By (i) the subsequence $(a_{m_n})$ does not $N$-converge to $N$, in contradiction with Proposition \ref{suhoitucuamoidaycon}. Hence $N$ is not a $N$-limit of $(a_n)$. 		
	\end{proof}
	
	Next proposition proves Theorem \ref{Stelling sterke convergentie onbegrensd} for sequences of neutrices.
	\begin{proposition}\label{strong conver of neutrix sequence}
		Let $(A_n)$ be a sequence of neutrices and $N\not=0$ be a neutrix. If $A_n\longrightarrow N$ then   $A_n \rightsquigarrow N$. 
	\end{proposition}
	
	\begin{proof}
		Suppose that for all $n\in \N$ there exist $p_n\geq n$ such that $N\subset A_{p_n}$. Let $D=\{(n, A_n\setminus N): N\subset A_n\}$. Then $D$ is cofinal with $A_n$. By Proposition \ref{exists internal subsequence2}, there exists an internal sequence $(b_{k})$ such that $b_k\in A_k\setminus N$ for all $k\in \pi(D)$. Since $A_n \longrightarrow N$, we have $b_k \unright N$, in contradiction with Lemma \ref{hoitutrong}. Then there exists $n_0\in \N$ such that  $A_n\subseteq N$ for all $n\geq n_0$. Hence $A_n \rightsquigarrow N$. 
	\end{proof} 
	
	We are now able to prove Theorem \ref{Stelling sterke convergentie onbegrensd}. 
		
	\begin{proof}[Proof of Theorem \ref{Stelling sterke convergentie onbegrensd}] Without restriction of generality we may suppose that $\alpha =N(\alpha)=A$. Clearly, if there exists $n_0 \in \N$ such that for all $n\geq n_0$ we have $u_n\subseteq A$, then $u_n{\longrightarrow} A$.
		
		Conversely, assume that $u_n {\longrightarrow} A$. For all $ n\in N $, put $ A_n=N(u_{n}) $. Then $ (A_{n}) $ is a definable sequence of neutrices, which by Proposition \ref{converComponents} converges to $ A $. Then $ (A_{n}) \rightsquigarrow A$ by Proposition \ref{strong conver of neutrix sequence}. Hence there exists $m \in \N$ such that $A_n\subseteq A$ for all $n\geq m$. Suppose $ u $ does not converge strongly to $ A $. Then $ C=\left\lbrace n \in \N : u_{n}\nsubseteq N\right\rbrace  $ is cofinal in $ \mathbb{N} $. Put $ D= \left\lbrace n\in C : n\geq m\right\rbrace $. Then $ u_{n}\cap A =\emptyset $ for all $n \in D $. By Proposition \ref{exists internal subsequence2} there exists an internal set $ J\subseteq D $ cofinal with $ \mathbb{N} $ and an internal sequence $ (a_{n})_{n\in J} $ such that $ a_{n} \in u_{n} $ for all $ n\in J $. By Lemma \ref{hoitutrong} the sequence $ (a_{n})_{n\in J} $ does not converge to $ A $, hence also $ (u_{n})_{n\in J} $ does not converge to $ A $, in contradiction with Proposition \ref{suhoitucuamoidaycon}. Hence $ u $ converges strongly to $ A $.
	\end{proof}

	\subsection{Strong convergence with respect to initial segments\label{subsection initial segments}}
	Theorem \ref{Convergence initial segment} on strong convergence of precise sequences with respect to an initial segment will be a consequence of the next results, each of which dealing with a special case.
	
	We start with the case where the sequence is convergent with respect to an initial segment which is smaller than the initial segment on which it is defined.
	
	\begin{theorem}
		\label{Hoofdstelling bevat}Let $S\subseteq \N$ be an internal initial segment and $C\subset S$ be a possibly external
		initial segment of $\N
		$. Let $u:S \rightarrow %
		%TCIMACRO{\U{211d} }%
		%BeginExpansion
		\mathbb{E}
		%EndExpansion
		$ be a local flexible sequence. Let $\alpha =a+A\in \mathbb{E}$. Assume that $u$
		converges to $\alpha $ with respect to $C$. Then $u$ converges strongly to $%
		\alpha $ with respect to $C$.
	\end{theorem}
	
	\begin{proof}
		Without restriction of generality we may assume that $\alpha =A$. If $C$ is
		internal, let $m=\max C$. Then $u_{m}\subseteq A$, otherwise we see that $u$
		does not converge to $A$ with respect to $C$, by taking $\varepsilon \in 
		%TCIMACRO{\U{211d} }%
		%BeginExpansion
		\mathbb{R}
		%EndExpansion
		$ such that $A<\varepsilon <\left \vert u_{m}\right \vert $. If $C$ is
		external, then $C\subset S$. Let $\nu \in S\setminus C$. Suppose $%
		u_{\nu }\nsubseteq A$. Let $\varepsilon \in \R	$ be such that $A<\varepsilon <\left \vert u_{\nu }\right \vert $. Then $%
		A<\varepsilon /3$ and there exists $m\in C$ such that $\left \vert 
		u_{n}-A\right \vert \leq \varepsilon /3$ for all $n\geq m$. Then $\varepsilon 
		< \left \vert u_{\nu } \right \vert < \frac{2}{3} \varepsilon $, a contradiction. Hence $u_{\nu }\subseteq A$%
		. We conclude that $u$ strongly converges to $A$ with respect to $C$.
	\end{proof}
The second case deals with internal sequences, for which convergence is only known on an external initial segment $ C $ of their domain $ S $. It considers cases where strong convergence possibly does not happen on $ C $, but by conveniently restricting the domain $ S $, they are strongly convergent with respect to this initial segment.

	\begin{theorem}
		\label{Stelling sterke convergentie intern}Let $(C,D)$ be a cut of $%
		%TCIMACRO{\U{2115} }%
		%BeginExpansion
		\mathbb{N}
		%EndExpansion
		$ into an external initial segment $C$ and a final segment $D$. Let $S$ be an
		internal initial segment of $%
		%TCIMACRO{\U{2115} }%
		%BeginExpansion
		\mathbb{N}
		%EndExpansion
		$ such that $C\subset S$. Let $u:S\rightarrow %
		%TCIMACRO{\U{211d} }%
		%BeginExpansion
		\mathbb{R}
		%EndExpansion
		$ be an internal sequence and $\alpha =a+A\in \mathbb{E}$ be such that $%
		u_{\upharpoonright C}$ converges to $\alpha $ with respect to $C$. If $C$ is
		a pregalaxy and $A$ is a halo, or if $C$ is a halo and $A$ is a galaxy, there
		exists $\nu \in S\cap D$ such that the sequence $u$ converges strongly to $%
		\alpha $ on $\{0,...,\nu \}$ with respect to $C$.
	\end{theorem}
	
	\begin{proof}
		Without restriction of generality we may suppose that $\alpha =A$. Observe
		that $S\setminus C\neq \emptyset $, for $C$ is external and $S$ is
		internal. Assume first that $C$ is a galaxy and $A$ is a halo. Then there is
		a standard set $W$ and an internal function $a:W\rightarrow 
		%TCIMACRO{\U{211d} }%
		%BeginExpansion
		\mathbb{R}
		%EndExpansion
		^{+}$ such that $A=\bigcap _{\st (w)\in W}[-a_{w},a_{w}]$, and a standard set $Z$
		and an internal function $c:Z\rightarrow 
		%TCIMACRO{\U{2115} }%
		%BeginExpansion
		\mathbb{N}
		%EndExpansion
		$ such that $C=\bigcup _{\st(z)\in Z}\{0,...,c_{z}\}$. Then%
		\begin{equation*}
		\forall ^\mathrm{st fin}I\subseteq W\exists ^{\st}z\in Z\forall w\in I\forall n\in 
		%TCIMACRO{\U{2115} }%
		%BeginExpansion
		\mathbb{N}
		%EndExpansion
		(c_{z}\leq n\in C\rightarrow (\left \vert u_{n}\right \vert \leq a_{w}).
		\end{equation*}%
		Because $u$ is internal and $C$ is external, by the Cauchy Principle we
		obtain that 
		\begin{equation*}
		\forall ^\mathrm{st fin}I\subseteq W\exists ^{\st}z\in Z\exists f\in D\forall w\in
		I\forall n\in 
		%TCIMACRO{\U{2115} }%
		%BeginExpansion
		\mathbb{N}
		%EndExpansion
		(c_{z}\leq n\leq f\rightarrow (\left \vert u_{n}\right \vert \leq a_{w}).
		\end{equation*}%
		By Standardization, respectively the Saturation Principle, there exists a standard function $%
		\widetilde{z}:W\rightarrow Z$, respectively an internal function $\widetilde{%
			f}:W\rightarrow 
		%TCIMACRO{\U{2115} }%
		%BeginExpansion
		\mathbb{N}
		%EndExpansion
		$ such that $\widetilde{f}(w)\in D$ for all standard $w\in W$ and 
		\begin{equation*}
		\forall ^{\st}w\in W(\forall n\in 
		%TCIMACRO{\U{2115} }%
		%BeginExpansion
		\mathbb{N}
		%EndExpansion
		(c_{\widetilde{z}(w)}\leq n\leq \widetilde{f}(w)\rightarrow (\left \vert
		u_{n}\right \vert \leq a_{w}).
		\end{equation*}%
		Put $G=\bigcup _{\st(w)\in W}\{ \widetilde{f}(w),...,\max S\}$. Because $G$ is a
		pregalaxy included in the halo $S\cap D$, by the Fehrele Principle it is
		strictly included in $S\cap D$. Let $\nu \in (S\cap D)\setminus G$. Observe 
		$\bigcup _{\st(w)\in W}\{0,...,c_{\widetilde{z}(w)}\} \subseteq C$. Hence $u_n\in A
		$ at least for all $n\in \{0,...,\nu \} \cap D$. Then $u$ converges strongly
		to $A$ on $\{0,...,\nu \}$ with respect to $C$.
		
		Secondly, assume that $C$ is a halo and $A$ is a pregalaxy. Let $Z$ be a
		standard set and $c:Z\rightarrow 
		%TCIMACRO{\U{2115} }%
		%BeginExpansion
		\mathbb{N}
		%EndExpansion
		$ be an internal function such that $C=\bigcap _{\st(z)\in Z}\{0,...,c_{z}\}$; we may assume that $c$ is
		decreasing at least on $^{\sigma }Z$.  Suppose that there does not exist $\nu \in S\cap D$ such
		that $u$ converges strongly to $A $ on $\{0,...,\nu \}$ with respect to 
		$C$. Then%
		\begin{equation*}
		\forall ^{\st}z\in Z\exists n\in D
		(n<c_{z}\wedge \left \vert u_{n}\right \vert >A).
		\end{equation*}%
		By the Saturation Principle there exists an internal function $\widetilde{n}:Z\rightarrow 
		%TCIMACRO{\U{2115} }%
		%BeginExpansion
		\mathbb{N}
		%EndExpansion
		$ such that $\widetilde{n}(z)\in D, \widetilde{n}(z)<c_{z}$ and $\left \vert u_{\widetilde{n}%
			(z)}\right \vert >A$ for all standard $z\in Z$. By Idealization if $ A=0 $, and if $A$ is a
		galaxy, by the Fehrele Principle, because $\bigcap _{\st(z)\in
			Z}\left[0,\left \vert u_{\widetilde{n}(z)}\right \vert \right]$ is a prehalo, there exists $\varepsilon \in \R$ such that $A<\varepsilon <\left \vert u_{\widetilde{n}(z)}\right \vert $ for
		all standard $z\in Z$. Because $u_{\upharpoonright C}$ converges to $A$ with respect to $C$, there exists $ \overline{c}\in C $ such that $ \left \vert u_{n}\right \vert <\varepsilon$ for all $ n\in C, n \geq \overline{c} $. Choose standard $ \overline{z} \in Z$. Now
		\begin{equation*}
		\{\overline{c},...,{\widetilde{n}(\overline{z})}\}\cap D\subseteq  \{ n\in \mathbb{N}:	
		 \exists m\in 		\mathbb{N}		(\overline{c}\leq m\leq n\wedge \left \vert u_{m}\right \vert >\varepsilon)   \} .
		\end{equation*}
		By the Cauchy Principle the inclusion is strict. So there exists $n\in
		C$ such that $\left \vert u_{m}\right \vert >\varepsilon $, for some $ m $ with $\overline{c}\leq m\leq n$, a
		contradiction. Hence $u$ converges strongly to $A$ on some interval $%
		\{0,...,\nu \}$ with respect to $C$, with $\nu \in D\cap S$.
	\end{proof}	
Theorem \ref{Stelling sterke convergentie intern} is a sort of Matching Principle, enabling an overall description for the coexistence of different kinds of behavior for, say, solutions of singular perturbations (see Subsection \ref{Subsection Matching Principles}).
 
	 The case where the initial segment and the limit set are of the same nature is contained in Theorem \ref{Stelling zelfde vorm}. In this case the strong convergence already takes place on the initial segment. If the initial segment is a galaxy, or if the initial segment is a halo and the limit is external, the proof uses an internal choice function, which can be obtained through the Saturation Principle. If the initial segment is a halo and the limit is precise, such an argument is not available, for in general there is no direct proof for the existence of an internal choice function within an external environment (\cite{BergWallet} gives a proof for an internal choice function within a galaxy or a halo). So this case needs special care. The proof will be based on the fact that if an internal sequence has a limit on a (internal) finite initial segment, it strongly converges to this limit, i.e. will eventually assume this limit value. Next lemma extends this property to halos. Then we will prove it for precise sequences, then for sequences of neutrices, and finally for general flexible sequences.
	 
	\begin{lemma}\label{intern precies cofinal}
	Let $(C,D)$ be a cut of $%
	%TCIMACRO{\U{2115} }%
	%BeginExpansion
	\mathbb{N}
	%EndExpansion
	$ into a halo $C$ and a galaxy $D$. Let $u$ be an
	internal local sequence of real numbers such that $ dom(u) $ is cofinal in $ C $. Let $ H\subseteq C  $ be an external set cofinal in $ C $ and $a \in \mathbb{R}$ be such that $ u_{\upharpoonright H} $ converges to $ a $. Then there exists $ m \in \mathbb{N}$ such that $u_{n}=a$ for all $ n\in H,n\geq m $.

	\end{lemma} 
\begin{proof}
	Without restriction of generality we may assume that $ a=0 $. We may assume that $ u $ is finite. Put $ \Delta =dom(u) $. Then $ \Delta $ is internal, finite, cofinal in $ C $ and, as a consequence of the Cauchy Principle, coinitial in $ D $. For all $ k\in \mathbb{N} $, put
	\begin{equation*}
	B_{k}=\left\lbrace  n\in \Delta : \lvert u_{n}\rvert \leq \frac{1}{k} \right\rbrace.
	\end{equation*}
	Then $ \left( B_{k} \right)  $ is an internal non-increasing sequence of finite sets, which are nonempty because $ b_{\upharpoonright H} $ converges to zero. Let $ B=\bigcap_{k\in \mathbb{N}} B_{k}$. Then $ B=B_{\overline{k}} $ for some $ \overline{k}\in \mathbb{N} $. Observe that for all $ n\in B $ one has $ \lvert u_{n}\rvert \leq 1/k  $ for all $ k \geq \overline{k} $, hence $ u_{n}=0 $. Let $ m\in C $
	be such that $\lvert u_{n}\rvert \leq 1/\overline{k} $ for all $ n\in H, n\geq m $. Then for all $ n\in H, n\geq m$ it holds that $ u_{n}\in B_{\overline{k}}=B $, hence $ u_{n}=0 $. This proves the lemma.
\end{proof}
	\begin{corollary}\label{Stelling intern halo a}
	Let $(C,D)$ be a cut of $%
	%TCIMACRO{\U{2115} }%
	%BeginExpansion
	\mathbb{N}
	%EndExpansion
	$ into  a halo $C$ and a galaxy $D$. Let $S$ be an
	internal initial segment of $%
	%TCIMACRO{\U{2115} }%
	%BeginExpansion
	\mathbb{N}
	%EndExpansion
	$ such that $C\subseteq S$. Let $%
	u:S \rightarrow %
	%TCIMACRO{\U{211d} }%
	%BeginExpansion
	\mathbb{E}
	%EndExpansion
	$ be an internal sequence and $a \in \mathbb{R}$ be such that $u_{\upharpoonright C}$ 		converges to $ a $ with respect to $C$. Then $u_{\upharpoonright C}$ converges
	strongly to $a $ with respect to $C$.
	
\end{corollary}

	\begin{theorem}
		\label{Stelling precies halo a}
		Let $(C,D)$ be a cut of $%
		%TCIMACRO{\U{2115} }%
		%BeginExpansion
		\mathbb{N}
		%EndExpansion
		$ into a halo $C$ and a galaxy $D$. Let $%
		u:C \rightarrow %
		%TCIMACRO{\U{211d} }%
		%BeginExpansion
		\mathbb{R}
		%EndExpansion
		$ be a precise external sequence and $a \in \mathbb{R}$ be such that $u$
		converges to $ a $ with respect to $C$. Then $u$ converges strongly to $a $ with respect to $C$.
	\end{theorem}
	
	\begin{proof}
		Without restriction of generality we may assume that $ a=0 $. By \eqref{uunionQ} we may write $u=\bigcup _{\mathrm{st}(x)\in X}Q_{x}   $, where for all standard $x\in X$ the prehalo $Q_{x}$ is
		functional, with domain $ \Delta _{x}= \text{dom}(Q_{x}) $. Because $ C $ is a halo, by Lemma \ref{cofinal sets}  
		at least one of the $ \Delta _{x} $ is cofinal with $ C $, and therefore it is a halo. Let
		\begin{equation*}  
		Y:=\,^{\sigma}\left\lbrace \st(x) \in X : \Delta _{x} \text{ is not cofinal with } C\right\rbrace,
		\end{equation*}
		and
		\begin{equation*}  
		Z:=\,^{\sigma}\left\lbrace \st(x) \in X : \Delta _{x} \text{ is  cofinal with } C\right\rbrace.
		\end{equation*}
		Again by by Lemma \ref{cofinal sets} there exists $ \overline{c} \in C $ such that for all $ n \in C, n\geq \overline{c}$ it holds that $ n \in  	\Delta  _{x} $ 	for some standard $ x\in Z $. 
		Let $ x\in Z $ be standard. By Lemma \ref{Lemma internal extension}, there exist $K$ with $\Delta  _{x}\subseteq K\subseteq \mathbb{N}	$ and an internal sequence $b:K\rightarrow 	\mathbb{R}$ such that $b_{\upharpoonright \Delta  _{x}}=Q_{x}$, implying that $ b $ is an internal bridge from $ \Delta _{x} $ to $ D $ along $ u $. Because $ u_{\upharpoonright \Delta _{x}} $ converges to $ 0 $, the sequence $ b_{\upharpoonright \Delta _{x}} $ converges to $ 0 $ with respect to $ \Delta _{x}$. By Lemma \ref{intern precies cofinal} there exists $ m \in \Delta _{x} $ such that $b_{n}=0 $ for all $ n \in \Delta _{x},n \geq m $. Because $u_{n}=b_{n} $ for all $ n \in \Delta _{x}$, it holds that $u_{n}=0 $ for all $ n \in \Delta _{x},n \geq m $. By the Saturation Principle there exists an internal function $ \widetilde{m}:X\rightarrow 	\mathbb{N} $ such that for all standard $ x \in X $ one has $ \widetilde{m}(x)\in C $ and $u_{n}=0 $ for all $ n \in \Delta _{x},n \geq \widetilde{m}(x) $. By Lemma \ref{cofinal sets} there exists $ \mu \in C$ such that $ \mu > \widetilde{m}(x) $ for all standard $ x \in Z$; we may assume that $ \mu \geq \overline{c} $. Let $ n \in C,n \geq \mu $. Then $ n \in \Delta _{x} $ for some standard $ x \in Z $. Hence $ u_{n}=0 $. We conclude that $ u $ converges strongly to $0$
		with respect to $C$.
	\end{proof}
	
	To deal with sequences of neutrices converging to $ 0 $ with respect to a halo, we recall some terminology from \cite{BergWallet}.
	\begin{definition}
		Let $ k\in \mathbb{N} $ be standard and $\emptyset \subset E \subseteq \R^{k}$ be external. We
		say that $E$ is \emph{strongly open} if there exists $r > 0$ such that $B(u, r) \subseteq E$ for all $u \in E$, where $B(u, r) = \left\{x \in \R^{k} : \lvert x -u\rvert \leq r \right\}$.
	\end{definition}
\begin{definition}
	Let $ k\in \mathbb{N} $ be standard and $\emptyset \subset E \subset \R^{k}$ be external. We
	say that $E$ has \emph{external borders} if for every internal path $ \phi:[0,1]\rightarrow \R^{k} $ such that $\phi(0)\in E $ and $\phi(1)\in \R^{k}\setminus E $ the set $ \{t\in [0,1] :  \forall s(0\leq s\leq t \Rightarrow \phi(s)\in E) \} $ is external. We
say that $E$ is \emph{border-external}, if both $ E $ and $ \R^{k}\setminus E  $ have external borders.
	\end{definition}
	
	A possibly external subset of $ \mathbb{R}^{k} $ is said to be \emph{precompact} if it is contained in an (internal) compact set.
	
	We also recall two results from \cite{BergWallet}. Proposition 22 states that a clopen set $ E $ such that $\emptyset \subset E \subset \R^{k}$  is border-external. Theorem 28 states that if in addition $ E $ is precompact, such a border-external set is strongly open. The following theorem is an immediate consequence of these two results.
	\begin{theorem}\label{stelling sterk open}
		Let $ V\subset \mathbb{R}^{k} $	be a precompact galaxy or halo. If both $ V $ and its complement are open, then $ V $ is strongly open.
	\end{theorem}
	
	\begin{lemma}
		\label{Convergence neutrices same form} Let $(C,D)$ be a cut of $%
		%TCIMACRO{\U{2115} }%
		%BeginExpansion
		\mathbb{N}
		%EndExpansion
		$ into a halo $C$ and a galaxy $D$. Let $%
		U:C \rightarrow  \mathbb{E}$ be a definable sequence of neutrices converging to $0 $ with respect to $C$. Then $U$ converges
		strongly to $0 $ with respect to $C$.
	\end{lemma}
	
	\begin{proof}
		We may define $ U'_{n}= \bigcup{\left\lbrace U_{p} : p\in C \wedge p\geq n\right\rbrace } $ for all $ n\in C $. Then $ U' $ is non-increasing, and still converging to $ 0 $ with respect to $ C $. So without restriction of generality we may assume that $ U $ is non-increasing. We may write the graph of $ U $ in the form $\Gamma (U)=\bigcup _{\mathrm{st}(x)\in X}\bigcap_{\mathrm{st}(y)\in Y} I_{xy}   $, where $ I:X\times Y \rightarrow  \mathcal{P}(\mathbb{N} \times\mathbb{R}) $ is internal. Also without loss of generality we may suppose that for each $ n\in C $ the sets $ I_{xy}(n)  $ are convex, compact and symmetrical with respect to $ 0 $, and moreover are closed under finite intersections. For all $x\in X$ we put $H_{x}=\bigcap_{\mathrm{st}(y)\in Y}I_{xy}$; then also $H_{x}(n)  $ is convex and symmetrical with respect to $ 0 $ for all $ n\in C $. 
		
		Suppose towards a contradiction that there exists $ \overline{c}\in C $ such that $ U_{n} $ is a halo for all $ n\geq \overline{c} $; for simplicity we assume that $ \overline{c}=0  $. Then every neutrix $ U_{n},n\in C $ is external and stable under at least some shifts in the second coordinate. Note that $ C $ being an external initial segment of $\mathbb{N}  $, it is stable under the addition of limited numbers ($ c\in C\Rightarrow c+k\in C $ for all limited $ k\in \mathbb{N} $). We will associate to $ U $ a clopen set $ V\subset \mathbb{R} \times \mathbb{R} $ as follows. First we define
		\begin{equation*}
		\overline{U}= \left\lbrace (s,t)\in \mathbb{R} \times \mathbb{R} : t\in  U_{\left[ s\right]}  \right\rbrace,
		\end{equation*} 
		where $\left[ s\right]$ denotes the integer part of the real number $s$.
		We associate to $ \overline{U} $ a set $ V $ which is stable under infinitesimal shifts in the first coordinate, by putting
		\begin{equation*}
		V= \left\lbrace (s,t)\in \mathbb{R} \times \mathbb{R} : \exists\sigma \simeq \lvert s\rvert  \left((\sigma,t)\in\overline{U} \right)\right\rbrace. 
		\end{equation*}
		Then $ \pi (V) =C\cup\oslash$. 
		Both $ V $ and its complement are open in $ \mathbb{R} \times \mathbb{R} $. Indeed, let $ (s,t) \in V$. Because $ U_{s} $ is an external neutrix, there exists $ \delta >0 $ such that $ (s,t+\tau) \in U_{t}  $ for all $ \lvert\tau\rvert <\delta $. Also $ (\sigma,\tau) \in U_{\sigma}\subseteq V$ for all $ \sigma \simeq s $. Hence $ V $ is open. With a similar argument we prove that the complement of $ V $ in $ \mathbb{R} \times \mathbb{R} $ is open. Being a precompact halo, by Theorem \ref{stelling sterk open}, it holds that $ V $ is strongly open. Hence there exists $ \varepsilon >0 $ such that $ (n, \varepsilon) \in U_{n} $ for all $ n\in C $. Hence $ U $ does not converge to $ 0 $ with respect to $ C $, a contradiction.

		As a consequence $ U $ has a cofinal sequence of pregalaxies. Put
		\begin{equation*}
		F= \left\lbrace n \in C : U_{n} \text{ is a pregalaxy} \right\rbrace.
		\end{equation*}
		Let $ n \in F $. From the representation 
		\begin{equation*}
		U_{n}=\bigcap_{\mathrm{st}(x)\in X}H_{x}(n)
		\end{equation*}
		we derive from \eqref{Uxyn} that there exists standard $ y\in Y $ such that
		\begin{equation*}
		U_{n}=\bigcup _{\mathrm{st}(x)\in X} I_{xy}(n).   
		\end{equation*}
		
		For all standard $ y\in Y $ we define
		\begin{equation*}
		C_{y}=\left\lbrace n\in F : U(n)=\bigcup _{\mathrm{st}(x)\in X} I_{xy}(n)\right\rbrace .   
		\end{equation*}
		Then $ C_{\overline{y}} $ is cofinal in the halo $ C $ for some standard $ \overline{y}\in Y $ by Lemma \ref{cofinal sets}.
		
		For all standard $ x\in X $ we define a function $ s_{x}:C_{\overline{y}}\rightarrow \mathbb{R} $ by
		\begin{equation*}
		s_{x}(n) = 		
		\begin{cases}
		\max I_{x\overline{y}}(n) &  \text{if it is defined}\\
		0 &  \text{otherwise.}
		\end{cases} 
		\end{equation*}
		Then $ s_{x} $ is a definable precise function such that $s_{x}(n)\in U_{n}  $ for all $ n\in C_{\overline{y}}$. Hence $ s_{x} $ converges to $ 0 $ with respect to $ C $ and by Theorem
		\ref{Stelling precies halo a} it converges strongly to $ 0 $ with respect to $ C $. So there exists $ \gamma \in C $ such that $ s_{x}(n)=0 $ for all $ n \in C_{\overline{y}},n\geq \gamma $. Again by Lemma \ref{cofinal sets} we find $ \overline{\gamma}\in C$ such that $ s_{x}(n)=0 $ for all standard $x\in X$ and all $n \in C_{\overline{y}}$ such that $  n\geq\overline{\gamma} $. 
		
		Because $ C_{\overline{y}} $ is cofinal in $ C $, there exists $ m\in C_{\overline{y}}, m\geq\overline{\gamma} $. Let $ v\in U_{m}  $. Then there exists standard $ x\in X $ such that $ v\in  I_{x\overline{y}}(m)  $. So $ \lvert v\rvert\leq s_{x}(m)=0 $, implying that $ v=0 $. Hence $ U_{m}=\left\lbrace 0\right\rbrace  $. Because $ U $ is non-increasing on $ C $, it holds that  $ U_{n}={0} $ for all $ n \in C, n\geq m$.
		We conclude that $ U $ converges strongly to $ 0 $ with respect to $ C $.
		
	\end{proof}
	\begin{theorem}\label{flexible halo a}
		Let $(C,D)$ be a cut of $%
		%TCIMACRO{\U{2115} }%
		%BeginExpansion
		\mathbb{N}
		%EndExpansion
		$ into a nonempty initial segment $C$ and a nonempty final segment $D$. Let $%
		u:C \rightarrow  \mathbb{E}$ be a flexible sequence and $a\in \mathbb{R}$ be such that $u$
		converges to $a$ with respect to $C$. If $ C $ is a halo, the sequence $u$ converges
		strongly to $a $ with respect to $C$.
	\end{theorem}
	\begin{proof}
		Without restriction of generality we may assume that $ a=0 $. We obtain a definable sequence of neutrices $ U $ by putting $ U_{n}=N(u_{n}) $ for all $ n\in C $. Since $ u $ converges to $ 0 $ with respect to $ C $, then $ U $ must also converge to $ 0 $ with respect to $ C $. By Theorem \ref{Convergence neutrices same form} the sequence $ U $ converges strongly to $ 0 $ with respect to $ C $. Then there exists $ c\in C $ such that $ U_{n}=0 $ for all $ n\in C,n\geq c $. Put $ C'= \left\lbrace n\in C: n\geq c \right\rbrace  $. Because $ u $ is definable its graph has the form $  
		\Gamma(u) =\bigcup _{\mathrm{st} (x)\in X}\bigcap _{\mathrm{st}(y)\in Y}I_{xy},  $
		where $X$ and $Y$ are standard sets and $I:X\times Y \rightarrow\mathcal{P}( \N \times\R)$ is an internal mapping. Then
		\begin{equation*}
		\Gamma(u_{\upharpoonright C'}) =\bigcup _{\mathrm {st} (x)\in X}\bigcap _{\mathrm{st}(y)\in Y}((I_{xy}\cap C')\times \mathbb{R}),  
		\end{equation*}
		hence $ u_{\upharpoonright C'} $ is definable. Being precise, it strongly converges to $ 0 $ with respect to $ C' $ by Theorem \ref{Stelling precies halo a}. Hence $ u$  strongly converges to $ 0 $ with respect to $ C $.
	\end{proof}
	
	\begin{theorem}
		\label{Stelling zelfde vorm}Let $(C,D)$ be a cut of $%
		%TCIMACRO{\U{2115} }%
		%BeginExpansion
		\mathbb{N}
		%EndExpansion
		$ into a nonempty initial segment $C$ and a nonempty final segment $D$. Let $%
		u:C \rightarrow %
		%TCIMACRO{\U{211d} }%
		%BeginExpansion
		\mathbb{E}
		%EndExpansion
		$ be a flexible sequence and $\alpha =a+A\in \mathbb{E}$ be such that $u$
		converges to $\alpha $ with respect to $C$. If $C$ and $A$ both are
		pregalaxies, or both prehalos, the sequence $u$ converges
		strongly to $\alpha $ with respect to $C$.
	\end{theorem}
	
	\begin{proof}
		Without restriction of generality we may suppose that $\alpha =A$. Assume
		first that $C$ is internal. Let $M=\max C$. Suppose $u_{M} \nsubseteq A$. Let $%
		\varepsilon \in 
		%TCIMACRO{\U{211d} }%
		%BeginExpansion
		\mathbb{R}
		%EndExpansion
		,\varepsilon >A$ be such that $\left \vert u_{M}\right \vert >\varepsilon $. Then $%
		\forall m\in C\exists n\in C(n\geq m\wedge \left \vert u_{n}\right \vert
		>\varepsilon )$, in contradiction with Definition \ref{Definition convergence
			local}. Hence $u_{M}\subseteq A$.
			
		Secondly, we assume that $C$ is a galaxy; and therefore $A$ is a pregalaxy. Then there exist
		a standard set $Z$ and an internal function $c:Z\rightarrow \N
		$ such that $C=\bigcup _{\st(z)\in Z}\{0,...,c_{z}\}$. Suppose that $u$ does not
		converge strongly to $A$ with respect to $C$. Then for all standard $z\in Z$
		there is $n\in C,n>c_{z}$ such that $\left \vert u_{n}\right \vert>A$. By the Saturation Principle we find an
		internal sequence $\widetilde{n}:Z\rightarrow 
		%TCIMACRO{\U{2115} }%
		%BeginExpansion
		\mathbb{N}
		%EndExpansion
		$ such that for all standard $z\in Z$ it holds that $\widetilde{n}(z)\in C,%
		\widetilde{n}(z)>c_{z}$ and $\left \vert u_{\widetilde{n}(z)}\right \vert >A$.  For all standard $ z\in Z $ there exists $ v\in \mathbb{R} $ such that $ A<v $ and $ v<\left \vert u_{\widetilde{n}(z)}\right \vert$, because $ \left \vert u_{\widetilde{n}(z)}\right \vert $ is an external number. So by the Saturation Principle we find an
		internal sequence $\widetilde{v}:Z\rightarrow 
		%TCIMACRO{\U{2115} }%
		%BeginExpansion
		\mathbb{R}
		%EndExpansion
		$ such that for all standard $z\in Z$ it holds that $\widetilde{v}(z)<  \left \vert u_{\widetilde{n}(z)}\right \vert $.  Then $ \bigcup_{\st(z) \in Z}\left[ \widetilde{v}(z),\infty\right) $ is a pregalaxy, disjunct from the pregalaxy $A$. Then  
		by Idealization if $A$ is internal, and by the Fehrele Principle, if $A$ is
		external, there exists $\varepsilon \in 
		%TCIMACRO{\U{211d} }%
		%BeginExpansion
		\mathbb{R}
		%EndExpansion
		$, such that $A<\varepsilon $ and $\varepsilon <\widetilde{v}(z)<\left \vert u_{\widetilde{n}(z)}\right \vert$ for all
		standard $z\in Z$. Then $u$ does not converge to $A$ with respect to $C$, a
		contradiction. Hence $u$ converges strongly to $A$ with respect to $C$.
		
Finally, we assume that $C$ is a halo. The case where $ A =0$  is contained in Theorem \ref{flexible halo a}.
		In the remaining case $ A $ is a halo. Then there exist a standard set $W$ and an internal function $a:W\rightarrow 
		%TCIMACRO{\U{211d} }%
		%BeginExpansion
		\mathbb{R}
		%EndExpansion
		^{+}$ such that $A=\bigcap _{\st(w)\in W}[-a_{w},a_{w}]$. For all standard $w\in W$
		there exists $m\in C$ such that $\left \vert u_{n}\right \vert < a_{w}$ for
		all $n\in C,n\geq m$. By the Saturation Principle we find an internal function $\widetilde{%
			m}:W\rightarrow 
		%TCIMACRO{\U{2115} }%
		%BeginExpansion
		\mathbb{N}
		%EndExpansion
		$ such that for all standard $w\in W$ it holds that $\widetilde{m}(w)\in C$
		and that $\left \vert u_{n}\right \vert < a_{w}$ for all $n\in C,n\geq 
		\widetilde{m}(w)$. By the Fehrele Principle there exists $\mu \in C$ such
		that $\widetilde{m}(w)<\mu $ for all standard $w\in W$. Then for all $n\in
		C,n\geq \mu $ we have that $\left \vert u_{n}\right \vert < a_{w}$ for all
		standard $w\in W$, i.e. $u_{n}\in A$. Hence $u$ converges strongly to $A$
		with respect to $C$.
	\end{proof}

\begin{proof}[Proof of Theorem \ref{Convergence initial segment}]
	If $C\subset S$, the theorem follows from Theorem \ref{Hoofdstelling bevat}. If $C$ and $A$ are both pregalaxies or both prehalos, the result follows from Theorem \ref{Stelling zelfde vorm}.
	
	In the remaining cases $C$ is a galaxy and $A$ is a halo, or $C$ is a halo and $A$ is a galaxy. We may assume that $ \alpha =A $. Let $ b $ be an internal bridge from from $C$ to $D$ along $u$, such bridges exist by Theorem \ref{Stelling flexible brug}. Put $ K=\text{dom} (b)$, $ K'=\{n\in \N : \exists k\in K(n\leq k) \} $ and let $ b':K'\rightarrow \mathbb{R} $ be defined by $ b'_n=b_n $ for $ n\in K $ and $ b'(n)=0 $ for $ n\notin K $. Then $b' $ converges to $ A $ with respect to $ C$. By Theorem \ref{Stelling sterke convergentie intern} there exists $ \nu \in D $ such that $ b'_{\upharpoonright{\{0,..., \nu \}}} $ is strongly convergent to $ A $ with respect to $ C$. Then $ b_{\upharpoonright{K\cap \{0,..., \nu \}}} $ is strongly convergent to $ A $ with respect to $ K\cap C$.		
\end{proof}
	
	\section{Operations on flexible sequences}
	
	\label{Section Operations Limits}
In this section we show that the usual properties of the convergence of sequences are still valid or may be adapted to the context of flexible sequences.	
	
\subsection{Boundedness and monotonicity}\label{Boundedness and monotonicity}

Let $N$ be a neutrix. In this subsection we study the relation between
boundedness and $N$-convergence. 

\begin{definition}
	We say that a flexible sequence $(u_{n})$ is \emph{%
		bounded} if there exists $\alpha \in \mathbb{E}$ such that $\alpha \neq 
	\mathbb{R}$ and $\left \vert u_{n}\right \vert \leq \alpha$, for all $n\in 
	\mathbb{N}$.
\end{definition}

Clearly the element $\alpha$ in the previous definition may be supposed to be precise.

It is well-known that convergent sequences are bounded. This result is no longer true for flexible sequences, as shown by Example \ref{example not bounded}. However, it is possible to give an adapted version of this result, proving that a $N$-convergent flexible sequence is bounded for large enough indices. We will call such sequences \emph{eventually bounded}. 

\begin{definition} 
	A flexible sequence $(u_n)$ is said to be \emph{eventually bounded} if there exist $n_0\in \N$ and $\alpha\not=\R$ such that $|u_n|\leq \alpha$ for all $n\geq n_0$.
\end{definition}

\begin{example} \label{example not bounded}
	Let $(u_n)$ be the flexible sequence defined by $u_n=\begin{cases} \R &\text{if }n\in \pounds\\
	1 &\text{if } n\not\in \pounds\end{cases}$.
	It is easy to see that this flexible sequence is convergent to  $1$ but not bounded. 
\end{example}

The sequence in Example \ref{example not bounded} above is eventual bounded but not bounded, so the two notions of boundedness are not equivalent in general. However the two notions do coincide if the sequence is precise. 

\begin{proposition} \label{propboundedfinite}
A precise sequence $(a_n)$ defined on a finite set is bounded. 
\end{proposition}

\begin{proof}
If $a$ is internal the result is obvious. If $ a $ is not internal, by the Representation Theorem, there exist a standard set $ X $, an internal family of (internal) sequences $ \left( u_{x}\right) _{x \in X} $ and a family of prehalos  $ \left( \Delta_{x}\right) _{x \in X} $ such that $a=\bigcup _{\mathrm{st}(x)\in X}u_{x\upharpoonright\Delta_{x}}$. 
	For all $x$ put $S_x= \Delta_x \cap [0,...,k]$, then $u_{x\upharpoonright S_x}$ has a maximum. By Lemma \ref{cofinal sets} there exists $M$ such that $M >u_x$, for all standard $x$. Hence $a_{\upharpoonright [0,...,k]}$ has a maximum. As a consequence $a$ has a maximum.
\end{proof}

\begin{corollary}\label{Propboundedprecise}
	A precise sequence $(a_n)$ is bounded if and only if it is eventually bounded.
\end{corollary}

\begin{proof}
		Because $a$ is eventually bounded there exist $k\in \N$ and $M_0 \in \R$ such that $|a_n|<M_0$, for $n \geq k$. By Proposition \ref{propboundedfinite} there exists $M_1$ such that $a_{\upharpoonright [0,...,k]}<M_1$. Hence $a$ is bounded.
		The other implication is obvious.
\end{proof}

\begin{theorem}
	\label{e-conv is bounded}Every $N$-convergent sequence is eventually bounded.
\end{theorem}

\begin{proof}
	Let $N$ be a neutrix and $(u_n) $ be a flexible sequence such that
	$u_n\unright  \alpha$. Let $\varepsilon>N$. Then there is
	$n_{0}\in\N
	$ such that for $n\geq n_{0}$ we have $|u_n|-|\alpha|\leq |u_n-\alpha|<\varepsilon$ for all $n\geq n_0$. Hence $|u_n|\leq \varepsilon+|\alpha|$ for all $n\geq n_0$ and we conclude that $(u_n)$ is eventually bounded. 
\end{proof}

Sequences which are eventually constant and equal to some $\alpha \in 
\mathbb{E}$ are convergent to $\alpha$.

\begin{proposition}
	\label{constant seq e-conv}Let $(u_{n})$ be a flexible sequence and $\alpha \in \E$. If there exists $n_{0}\in \mathbb{N}$ such that $u_n= \alpha$ for $n\geq
	n_{0}$, then $u_{n}\longrightarrow\alpha $.
\end{proposition}

\begin{proof}
	Let $\varepsilon >N(\alpha)$. Then for $n\geq n_{0}$
	we have $\left\vert u_{n}-\alpha \right\vert =\left\vert \alpha -\alpha \right\vert
	=N(\alpha) <\varepsilon$. Hence $u_{n}\longrightarrow \alpha $.
\end{proof}

Let $(u_{n})$ be an internal sequence of real numbers and let $n_{0}\in \N$.
It is well-known that if  $u_{n}\geq0$ for $n\geq n_{0}$ and $%
u_{n}\longrightarrow a$ for some $a\in \mathbb{R}$ then $a\geq0$. The following
proposition gives an adapted version of this result for flexible
sequences.

\begin{proposition}
	\label{positive N-limit} Let $N$ be a neutrix, and $\alpha=a+A \in \E$. Let $(u_n)$ be a flexible sequence such that $u_n\unright  \alpha$. If there exists $n_{0}\in\N$ such that $N\leq u_n$, for $n\geq n_{0}$, then $N\leq \alpha+N$.
\end{proposition}

\begin{proof}
	Observe that $N(\alpha)=A\subseteq N$, so $\alpha+N=a+N$ and $\alpha+N$ is a $N$-limit of the sequence $(u_n)$. If $\alpha+N$ is a neutrix, then $\alpha+N=N$ and we are done. If not, we may assume that $N$ and $\alpha+N$ are disjoint and therefore $N \leq \alpha+N$ is equivalent to $N< \alpha+N$. Suppose that $\alpha+N=a+N<N.$ Let $\varepsilon=-\frac{a}{2}>N$. Then there exists $k_0\in \N, k_0\geq n_0 $ such that for $n\geq k_0$ it holds that $|u_n-a+N|< \varepsilon$. So $a-\varepsilon<u_n+N<a+\varepsilon$ for all $n\geq k_0$. This means that $\frac{3}{2}a< u_n<\frac{a}{2}< N$ for all $n\geq k_0,$ in contradiction with the assumption. Hence $N<\alpha+N$.
	Combining the two cases we conclude that $N\leq \alpha+N$.
\end{proof}

\begin{proposition}
	\label{Proposition monotony} Let $N$ be a neutrix and $(u_n),(v_n)$ be two flexible sequences  such that $u_n\unright  \alpha,$ $v_{n}\unright  \beta$, for some
	$\alpha,\beta\in \E$. If  there is $n_{0}\in\N 
	$ such that $ u_n\leq v_{n}$  for all $n\geq n_{0}$, then $\alpha \leq  \beta$. 
\end{proposition}

\begin{proof}  We may assume that $\alpha=a+N$ and $\beta=b+N$, for some $a,b \in \R$. By Proposition \ref{converComponents} we may assume that $\alpha$ and $\beta$ are precise. Suppose that $\alpha \not\leq  \beta$. This means that $\beta < \alpha$. It follows that $\varepsilon
	= \frac{\alpha-\beta}{2}>N$. So there exist $n_{1},n_{2}\in\N$ such that  $\left \vert u_n-\alpha \right \vert
	<\varepsilon$ for $n\geq n_{1}$ and  $\left \vert v_{n}-\beta \right \vert
	<\varepsilon$ for $n\geq n_{2}$. Let $n_{3}=\max \left \{  n_{0},n_{1},n_{2}\right \}  $. Then for $n\geq n_{3}$ we have
	\[
	v_{n}\leq v_n+N<\beta+\varepsilon=\beta+\frac{\alpha-\beta}{2}=\alpha-\frac{\alpha-\beta
	}{2}=\alpha-\varepsilon<u_n+N.
	\] So $v_n < u_n$ for all $n\geq n_3,$ which is a contradiction. Hence $\alpha \leq \beta$.
\end{proof}
\begin{remark} By similar arguments one shows that the conclusions in Propositions \ref{positive N-limit} and \ref{Proposition monotony} remain true if one replaces $ \leq $ by $\geq$.
\end{remark}

We show next an adapted version of the well-known result saying that the
product of a bounded sequence with an infinitesimal sequence (sequence which
converges to zero) is an infinitesimal sequence. This may also be seen as a special case of the product of two sequences, which will be dealt with in Theorem \ref{operationsNconv}.

\begin{proposition}
	\label{inf.bounded=inf}Let $N$ be a neutrix and let $(u_{n}),(v_{n})$ be flexible sequences. If $u_{n} \unright 0$ and $(v_n)$ is eventually bounded by $ \alpha \in \mathbb{E}$,
	%there exists $\alpha \in \mathbb{E}$ such that $\alpha \neq \mathbb{R}$ and $\left \vert v_{n}\right \vert <\alpha$, for all $n\in 
	%\mathbb{N}$, 
	then $\left( u_{n}v_{n}\right) \underset{\alpha N}{\longrightarrow} 0$.
\end{proposition}

\begin{proof}
	Since $(v_n)$ is eventually bounded, there exist $n_1 \in \N$ and $\alpha \in \mathbb{E}$ such that $\alpha \neq \mathbb{R}$ and $\left \vert v_{n}\right \vert <\alpha$, for all $n \geq n_1$. Without restriction in generality we assume that $\alpha$ is precise. Let $%
	\varepsilon>\alpha N$. Then there exists $n_{2}\in \mathbb{N}$ such that  $\left \vert u_{n}\right \vert <\varepsilon/\alpha$ for $%
	n\geq n_{0}$. Let $n_0=\max\{n_1,n_2\}$. Then, for $n \geq n_0$ we have that  
	\begin{equation*}
	\left \vert u_{n}v_{n}\right \vert =\left \vert u_{n}\right \vert \left
	\vert v_{n}\right \vert < \frac{\varepsilon}{\alpha}\alpha=\varepsilon \text{.}
	\end{equation*}
	Hence $\left( u_{n}v_{n}\right) \underset{\alpha N}{\longrightarrow} 0$.
\end{proof}

We finish this subsection by showing a version of the squeeze theorem. As in the classical case, this theorem enables one to calculate the $N$-limit of a sequence $\left(
v_{n}\right) $ by comparison with two other flexible sequences whose flexible limits
are equal and already known.

\begin{theorem}
	[Squeeze theorem]\label{Squeeze thm}Let $M,N$ be neutrices and let $(u_n),(v_n)$ and $(w_n)$ be flexible sequences such that $u_n\unright \alpha,$ $w_{n}  \underset{M}{\longrightarrow}  \alpha$, for some $\alpha \in \E$. If
	there is $n_{0}\in\N $ such that $u_n\leq v_{n}\leq w_{n}$ for $n\geq n_{0}$, then
	$v_{n} \underset{N+M}{\longrightarrow}  \alpha$. In particular, if $N=M$ then
	$v_{n}\unright  \alpha$.
\end{theorem}

\begin{proof}
	Assume that there exists $n_{0}\in\N$ such that  $u_n\leq v_{n}\leq w_{n}$ for $n\geq n_{0}$. Let $\varepsilon\in \R$ be such that $\varepsilon>N+M$.  Then $\varepsilon>N$ and $\varepsilon>M$. There exist
	$n_{1},n_{2}\in\N 
	$ such that for $n\geq n_{1}$ we have $\left \vert u_n-\alpha \right \vert
	<\varepsilon$ and for $n\geq n_{2}$ we have $\left \vert w_{n}-\alpha \right \vert
	<\varepsilon$. Let $n_{3}= \max \left \{  n_{0},n_{1},n_{2}\right \}  $.   Then
	for $n\geq n_{3}$  we have%
	\[
	-\varepsilon<u_n-\alpha\leq v_{n}-\alpha\leq w_{n}-\alpha<\varepsilon
	\]
	So  $\left \vert v_{n}-\alpha \right \vert <\varepsilon$. Hence, by Remark \ref{equiv flexible}, one concludes that
	$v_{n} \underset{N+M}{\longrightarrow}  \alpha$.
\end{proof}

\subsection{Operations}\label{Operations}
 The usual rules of operations of limits of sequences are valid in the context of definable flexible sequences. This is a consequence of Theorem \ref{strongconv} and Theorem \ref{Stelling sterke convergentie onbegrensd}.

	\begin{definition}
		A flexible sequence $(u_{n})$ is said to be \emph{zeroless} if $u_{n}\neq
		N(u_{n})$, for all $n\in \N$, i.e. if $0\notin u_{n}$, for all $n\in \N$.
	\end{definition}

	\begin{theorem} \label{strongconv}
		  Let $c\in \mathbb{R}$. Let $(u_{n}),(v_{n})$ be flexible sequences such that $u_{n}\rightsquigarrow \alpha $ and $%
		v_{n}\rightsquigarrow \beta $ for some $\alpha ,\beta \in \mathbb{E}.$ Then
		
		\begin{enumerate}
			\item  $u_n+v_n \rightsquigarrow \alpha+\beta$.
			
			\item  $u_n-v_n \rightsquigarrow \alpha-\beta$.
			
			\item $cu_{n}\rightsquigarrow c\alpha $.
			
			\item  $(u_{n}v_{n}) \rightsquigarrow \alpha \beta$.
			
			\item if $(v_n)$ and $\beta$ are zeroless then $u_n/v_n \rightsquigarrow \alpha/\beta $.
		\end{enumerate}
	\end{theorem}
	
	\begin{proof}
		 By the assumptions there exist $n_{1},n_{2}\in \N	$ such that for all $n\geq n_{1}$, it holds that $u_{n}\subseteq \alpha $
		and for all $n\geq n_{2}$, it holds that $v_{n}\subseteq \beta $. Then   $u_n+v_n \subseteq \alpha+\beta$, $u_n-v_n \subseteq \alpha - \beta$,  $cu_{n} \subseteq c\alpha$, $%
		u_{n}v_{n}\subseteq \alpha \beta $ and $u_n/v_n \subseteq \alpha/\beta$ for all $n\geq n_{0}$, where $n_{0}=\max
		\{n_{1},n_{2}\}$. This implies the proposition.
	\end{proof}
	
	We prove next an analogue of Theorem \ref{strongconv} for $N$-convergence. But first we need the following preparatory results.
	\begin{lemma}[{\protect \cite[p.151]{koudjetivandenberg}, \protect \cite[p. 89]%
			{koudjetithese} }]
		\label{lemma division}Let $\alpha ,\beta \in \mathbb{E}$ be zeroless and
		such that $\alpha =a+N(\alpha )$ with $a\in \R$. Then
		\begin{enumerate}
			\item \label{1/alpha}$\frac{1}{\alpha}=\frac{\alpha}{a^{2}}$,
			
			\smallskip
			
			\item \label{beta/alpha}$\frac{\beta}{\alpha}=\frac{\alpha \beta}{a^{2}}$.
		\end{enumerate}
	\end{lemma}
	
	\begin{lemma}
		\label{Lemma an>a/2}Let $N$ be a neutrix and let $(a_{n})$ be a zeroless
		precise sequence such that $a_{n}\underset{N}{\longrightarrow }a$, for some $%
		a\in \R, |a|>N$. Then there exists $n_{0}\in \N	$ such that for $n\geq n_{0}$ it holds that $|a|/2<\left \vert
		a_{n}\right \vert <2\left \vert a\right \vert $.
	\end{lemma}

\begin{proof}
		Let $\varepsilon =\left \vert a\right \vert /2$. Clearly $\left \vert a\right \vert
		/2>N$. Then there exists $n_{0}\in \N$ such that for all $n\geq n_{0}$ one has $|a_{n}-a|<\varepsilon =|a|/2.$ It
		follows that $|a|-|a|/2<|a_{n}|<|a|+|a|/2=\frac{3}{2}|a|.$ Hence $$\frac{|a|}{2}<\left \vert a_{n}\right \vert <\frac{3}{2}\left \vert a\right \vert\leq 2|a|,$$
		for all $n\geq n_{0}$.
	\end{proof}

	\begin{theorem} \label{operationsNconv}
		Let $N,M$ be neutrices and let $c\in \R$. Let $(u_{n}),(v_{n})$ be flexible sequences such that $u_{n}\underset{N%
		}{\longrightarrow }\alpha $ and $v_{n} \underset{M}{\longrightarrow }
		\beta $ for some $\alpha ,\beta \in \mathbb{E}$ . Then
		
		\begin{enumerate}
			\item \label{addition} $\left( u_{n}+v_{n}\right) \underset{N+M}{\longrightarrow}
			\left( \alpha +\beta \right)$.
			
			\item \label{subtraction} $\left( u_{n}-v_{n}\right) \underset{N+M}{\longrightarrow}
			\left( \alpha -\beta \right)$.
			
			\item \label{prodprecise} $\left( cu_{n}\right) \underset{cN}{%
				\longrightarrow}\left( c\alpha \right) $.
				
			\item \label{Theorem product}  $(u_{n}v_{n})\underset{K}{\longrightarrow }\alpha \beta $, with $%
		K=\alpha M+\beta N+MN.$	
		
		\item \label{division} if $(u_n)$ and $\alpha$ are zeroless, then $(1/u_{n})$ is $(N/a^{2})$-convergent to $%
		1/\alpha$.
		\end{enumerate}
		
		In particular, if $M=N$, then $\left( u_{n}+v_{n}\right) \underset{N}{%
			\longrightarrow} \left( \alpha +\beta \right) $ and $\left(
		u_{n}-v_{n}\right) \underset{N}{\longrightarrow} \left( \alpha -\beta
		\right) $.
	\end{theorem}

	\begin{proof} 
	
	\ref{addition}. Let $\varepsilon >N+M$. Then $\varepsilon >N$ and $\varepsilon >M$. So, there exist $n_{1},n_{2}\in 
		\N$ such that $\left \vert u_{n}-\alpha \right \vert <\varepsilon /2$ for $n\geq
		n_{1}$ and $\left \vert v_{n}-\beta \right \vert <\varepsilon /2$ for $n\geq
		n_{2}$. Let $n_0=\max \left \{ n_{1},n_{2}\right \} $. Then, for $n\geq n_0$ 
		\begin{equation*}
			\left \vert \left( u_{n}+v_{n}\right) -\left( \alpha +\beta \right)
			\right \vert \leq \left \vert u_{n}-\alpha \right \vert +\left \vert v_{n}-\beta
			\right \vert <\frac{\varepsilon }{2}+\frac{\varepsilon }{2}=\varepsilon .
		\end{equation*}%
		Hence $\left( u_{n}+v_{n}\right) \underset{N+M}{\longrightarrow }\left(
		\alpha +\beta \right) $.
		
		If $N=M$, then $N+M=N$. Hence $\left( u_{n}+v_{n}\right) \underset{N}{%
			\longrightarrow} \left( \alpha +\beta \right) $.

	\ref{subtraction}. The proof is analogous to the proof of Theorem \ref{operationsNconv}.\ref{addition}.
	
	\ref{prodprecise}. If $c=0$, then $cu_{n}=0 $ for
		all $n\in \N$. Hence $(cu_{n})$ converges to $0$ by Proposition \ref{constant seq e-conv}.
		Assume $c\neq 0$. Let $\varepsilon >cN$. Then $\varepsilon /\left \vert
		c\right \vert >N$. So, there exists $n_{0}\in \N	$ such that $\left \vert u_{n}-\alpha \right \vert <\varepsilon /\left \vert
		c\right \vert $ for all $n\geq n_{0}$. Because $c\in \R$, distributivity
		holds. Hence 
		\begin{equation*}
			\left \vert cu_{n}-c\alpha \right \vert =\left \vert c\left( u_{n}-\alpha
			\right) \right \vert =\left \vert c\right \vert \left \vert u_{n}-\alpha
			\right \vert <\left \vert c\right \vert \frac{\varepsilon }{\left \vert
				c\right \vert }=\varepsilon \text{.}
		\end{equation*}%
		We conclude that $\left( cu_{n}\right) \  \underset{cN}{\longrightarrow }%
		\left( c\alpha \right) $.
	
	\ref{Theorem product}. Let $\alpha=a+A, \beta=b+B$. We will consider three cases: (i) $M,N$ are both nonzero, (ii) exactly one of the neutrices $M,N$ is zero and (iii) $M,N$ are both zero.

(i)	By Proposition \ref{tinhduynhatcuagioihansailechneutrix} we have that $u_{n}\underset{N%
		}{\longrightarrow }a+N $ and $v_{n}  \underset{M}{\longrightarrow }
		b+M $. Then $u_{n}\longrightarrow a+N $ and $v_{n}\longrightarrow b+M$. Hence $u_n \rightsquigarrow a+N$ and $v_n \rightsquigarrow b+N$ by Theorem \ref{Stelling sterke convergentie onbegrensd} and the result follows from Theorem \ref{strongconv}.
	
	For the remaining cases we put $u_{n}=a_{n}+A_{n}$ and $v_{n}=b_{n}+B_{n}$, where $(a_{n}),(b_{n})$ are
		sequences of representatives and for all $n\in \N$, $A_{n}=N(u_{n})$ and $B_{n}=N(v_{n})$.
	
	(ii) Without loss of generality we
		assume that $M=0.$   By Proposition \ref{equiv e-conv} we may assume that $\alpha =a+N$ and $\beta =b\in \R
		$. So, $K=bN$.
		Let $\varepsilon >K$. By Theorem \ref{Stelling sterke convergentie onbegrensd}
		there exists $n_{1}\in \N$ such that $u_{n}\subseteq \alpha $ for all $n\geq n_{1}$%
		. Then, for all $n\geq n_{1}$ we have 
		\begin{equation}
			|u_{n}v_{n}-\alpha \beta |=|u_{n}v_{n}-\alpha b|\leq |\alpha v_{n}-\alpha
			b|=|a||v_{n}-b|+Nv_{n}+Nb.  \label{danhgia0}
		\end{equation}%
		We show that there exists $n_{2}\in \N$ such that 
		\begin{equation}
			|a||v_{n}-b|<\varepsilon /3,  \label{danhgia1}
		\end{equation}%
		for all $n\geq n_{2}$. If $a=0$ the inequality \eqref{danhgia1} is trivial.
		Otherwise, because $\varepsilon >K$, it holds that $\frac{\varepsilon }{3|a|}>0$.
		Since $v_{n}\longrightarrow b$, there exists $n_{2}\in \N	$ such that $|a||v_{n}-b|<\varepsilon /3$ for all $n\geq n_{1}$. 
		
		To estimate
		the term $Nv_{n}$ we consider two cases: $b=0$ and $b\not=0$.
		
		If $b\not=0$, there exists $n_{3}\in \N$ such that $|v_{n}|<2|b|$, for all $n\geq n_{3}$. It follows that 
		\begin{equation}
			Nv_{n}\leq N \cdot(2|b|)=K<\varepsilon /3.  \label{danhgia2}
		\end{equation}%
		From \eqref{danhgia0}-\eqref{danhgia2} we obtain $|u_{n}v_{n}-\alpha \beta
		|<\varepsilon $, for all $n\geq n_{4}$, where $n_{4}=\max \{n_{1},n_{2},n_{3}\}$%
		. Hence $u_{n}v_{n}\underset{K}{\longrightarrow }\alpha \beta .$
		
		If $b=0$, then $K=0$. Because $N\not=\R		$, there exists $\delta >0$ such that $\varepsilon /\delta >N.$ Also, since $%
		v_{n}\longrightarrow 0$, there exists $n_{5}\in \N	$ such that $|v_{n}|<\delta /3$ for all $n\geq n_{5}$. It follows that 
		\begin{equation}
			v_{n}N<\frac{\delta }{3}N<\frac{\varepsilon }{3},  \label{danhgia4}
		\end{equation}%
		for all $n\geq n_{5}$. Put $n_{0}=\max \{n_{1},n_{2},n_{5}\}$. Then, from %
		\eqref{danhgia0},\eqref{danhgia1}, \eqref{danhgia4} we obtain that $%
		|u_{n}v_{n}-\alpha \beta |<\varepsilon $ for all $n\geq n_{0}$. Hence $%
		u_{n}v_{n}\underset{K}{\longrightarrow }\alpha \beta .$
		
		(iii) By Proposition \ref{converComponents} it holds that $a_n \longrightarrow a$, $A_n \longrightarrow 0$, $b_n \longrightarrow b$ and $B_n \longrightarrow 0$. Since $(b_n)$ is precise it is bounded by Proposition \ref{Propboundedprecise} and Theorem \ref{e-conv is bounded}. I.e. there exists $M \in \R$ such that $|b_n|<M$. Let $\varepsilon >0$. Then there exist $n_1,n_2,n_3,n_4 \in \N$ such that $|a_n-a|<\varepsilon/(4(|M|+1))$, for $n \geq n_1$, $|b_n-b|<\varepsilon/(4(|a|+1))$, for $n \geq n_2$, $A_n<\varepsilon/(8(|b|+1))$, for $n \geq n_3$ and $B_n<\varepsilon/(8(|a|+1))$, for $n \geq n_4$. 
		 Because $u_{n}\longrightarrow a$ and $v_{n}%
		\longrightarrow b$, there exists $n_{5}\in \N$ such that $|u_{n}|<2(|a|+1)$, for all $n \geq n_5$ and $|v_{n}|<2(|b|+1)$, for all $n\geq n_{6}$.
		Let $n_0=\max \{n_1,n_2,n_3,n_4,n_5,n_6 \}$. Then, for $n \geq n_0$ we have 
		\begin{align*}
			|u_{n}v_{n}-ab|=& |a_{n}b_{n}-ab|+A_{n}v_{n}+B_{n}u_{n}  \notag  \label{pt0}
			\\
			\leq & |b_{n}||a_{n}-a|+|a||b_{n}-b|+A_{n}v_{n}+B_{n}u_{n} \\
			<& \frac{\varepsilon}{4}+\frac{\varepsilon}{4}+\frac{\varepsilon}{4}+\frac{\varepsilon}{4}=\varepsilon.
		\end{align*}%
	Hence, $u_{n}v_{n}\longrightarrow ab$.

	\ref{division}. Assume that $u_{n}=a_{n}+A_{n}$ for all $n\in \N$. Let $\varepsilon >N/a^{2}$. Then $\frac{a^{2}}{b}\varepsilon >N$ for all appreciable $b$.
		Because $(u_{n})$ is $N$-convergent to $\alpha $, it holds that $(A_{n})$ is 
		$N$-convergent to $N$ and $(a_{n})$ is $N$-convergent to $a$. Then there
		exists $n_{0}\in \N$ such that for all $n\geq n_{0}$ we have $|A_{n}-N|<\frac{a^{2}}{3}	\varepsilon$, that is, $|A_{n}|\leq \frac{a^{2}}{3}\varepsilon$ and $%
		|a-a_{n}|<\frac{a^{2}}{6}\varepsilon$.
		By Lemma \ref{Lemma an>a/2} there exists $n_1 \in \N$ such that $%
		|a|/2<|a_n|<2|a|$, for all $n\geq n_1$. Let $k=\max \{n_0, n_1\}$. Then by Lemma \ref{lemma division}.\ref{1/alpha}, for all $n\geq k$ we have that
		\begin{equation*}
		\begin{split}
			\left \vert \frac{ 1}{u_{n}}-\frac{1}{\alpha }\right \vert = &\left \vert 
			\frac{u_{n}}{a_{n}^{2}}- \frac{\alpha}{a^2 }\right \vert=\frac{|a^2u_n
				-a_n^2\alpha|}{a_n^2a^2} = \frac{|a^2a_n -a_n^2a|}{a_n^2a^2} +\frac{a^2A_n
				+a_n^2N}{a^2a_n^2}  \\
			\leq & \frac{|a -a_n|}{a_n a} +\frac{a^2A_n +4|a|^2N}{a^4/4} \leq \frac{%
				|a -a_n|}{a^2/2} +\frac{A_n +N}{a^2} < \frac{\varepsilon}{3} +\frac{2\varepsilon}{3}%
			=\varepsilon. 
			\end{split}
		\end{equation*}
		Hence $(1/u_n)$ is $(N/a^2)$-convergent
		to $1/\alpha$.
	\end{proof}

We end with an example that ilustrates how the neutrices change under the operations.
	
	\begin{example}
		\label{example product}
		\begin{enumerate}
		
		\item	 Let $(u_n),(v_n),(w_n)$ be flexible sequences defined respectively by $u_n=\frac{1}{n}+ \oslash$, $v_n=\frac{1}{n^2}+\varepsilon \pounds$ and $w_n=\omega^2+\omega \pounds$, where $\varepsilon$ is a fixed infinitesimal and $\omega$ a fixed infinitely large real number. We have that $u_nv_n= \frac{1}{n^3}+\frac{\varepsilon \pounds}{n}+\frac{\oslash}{n^2}+ \varepsilon \oslash$, which is $(\varepsilon \oslash)$-convergent to $0$ and $u_nw_n=\frac{\omega^2}{n}+ \frac{\omega\pounds}{n}+\omega^2\oslash$ wich is $(\omega^2 \oslash)$-convergent to $0$.

	\item Let $\omega $ be unlimited and let $(w_{n})$ be the constant sequence from the previous example. We have that $w_{n}\underset{\omega \pounds }{\longrightarrow }\omega ^{2}$ and $(w_{n}w_{n})\underset{\omega ^{3}\pounds }{%
			\longrightarrow }\omega ^{4}$. Consider the two constant representatives $a=\omega^2$ and $b=\omega^2 +\omega$. Then $ab-a^2=\omega^3$ which is bigger than the neutrix of convergence $\omega \pounds$. So the sequence is $\omega \pounds$-divergent.	
		\end{enumerate}
	\end{example}

	\section{Cauchy flexible sequences}
	
	\label{Section Cauchy}
	
	%In classical analysis many sequences are not explicitly determined by a formula. So we cannot verify whether that sequence is convergent or not by using the definition. In this case the Cauchy characterization of convergence  for sequences, which says that a sequence is convergent if and only if it is a Cauchy sequence, is a useful tool. 
	In this section we prove that the Cauchy characterization of sequences,
	i.e. the fact that a sequence is convergent if and only if it is a Cauchy
	sequence, still holds for the convergence of flexible sequences.
	
	Intuitively, a sequence is a Cauchy sequence if the terms of the sequence
	become arbitrarily close to each other as the sequence progresses. In other
	words the difference between terms of the sequence converges to $0$. For flexible sequences this means that the terms of the sequence are close to each other up to a neutrix $N$, i.e. the difference between terms of the sequence $N$-converges to $0$.

	\subsection{Global Cauchy sequences}
	
	\begin{definition}
		\label{def cauchy sequence} Let $N$ be a neutrix and $(u_{n})$ be a flexible
		sequence. We say that $(u_{n})$ is $N$-\emph{Cauchy} if and only
		if 
		\begin{equation}
			\forall \varepsilon >N\exists k\in 
			%TCIMACRO{\U{2115} }%
			%BeginExpansion
			\mathbb{N}
			%EndExpansion
			\forall m\in 
			%TCIMACRO{\U{2115} }%
			%BeginExpansion
			\mathbb{N}
			%EndExpansion
			\forall n\in 
			%TCIMACRO{\U{2115} }%
			%BeginExpansion
			\mathbb{N}
			%EndExpansion
			(k\leq m\wedge k\leq n\Rightarrow \left \vert u_{m}-u_{n}\right \vert <\varepsilon ).
			\label{N-Cauchy}
		\end{equation}%
		If $N=0$ we simply say that $(u_{n})$ is a \emph{Cauchy sequence}.
	\end{definition}
	
	\begin{definition}
		Let $N$ be a neutrix and $(u_{n})$ be a flexible sequence. We say that $%
		(u_{n})$ is \emph{strongly $N$-Cauchy} if and only if 
		\begin{equation}
			\exists k\in 
			%TCIMACRO{\U{2115} }%
			%BeginExpansion
			\mathbb{N}
			%EndExpansion
			\forall n\in 
			%TCIMACRO{\U{2115} }%
			%BeginExpansion
			\mathbb{N}
			%EndExpansion
			\forall m\in 
			%TCIMACRO{\U{2115} }%
			%BeginExpansion
			\mathbb{N}
			%EndExpansion
			(k \leq m\wedge k \leq n\Rightarrow u_{n}-u_{m}\subseteq N).  \label{Strong Cauchy}
		\end{equation}
	\end{definition}
	
	Reasoning as in Remark \ref{equiv flexible} it is possible to show
	that, in (\ref{N-Cauchy}), the element $\varepsilon$ can be taken precise.
	
	Let $N$ be a neutrix and let $(u_{n})$ be a $N$-Cauchy sequence. Clearly,
	for all $\varepsilon >N$ and for all $k\in \N	$ there exists $p\in \N$ such that $\left \vert u_{n+k}-u_{n}\right \vert <\varepsilon $ for $n>p$.
	
	\begin{example}
		\label{Example e-Cauchy}Let $N$ be a neutrix. Let $%
		(u_{n})$ be the flexibble sequence defined by $u_{n}=s_{n}+N$, where $\left(
		s_{n}\right) $ is a Cauchy sequence. We show that $\left( u_{n}\right) $ is $N$-Cauchy. Let $\varepsilon >N$. Then there exists $n_{0}\in \N$ such that for $m,n>n_{0}$ it holds that $\left \vert s_{m}-s_{n}\right \vert
		<\delta $, for $0<\delta <N$. Then 
		\begin{align*}
			\left \vert u_{m}-u_{n}\right \vert & =\left \vert s_{m}+N-\left(
			s_{n}+N\right) \right \vert  \\
			& =\left \vert s_{m}-s_{n}+N\right \vert \leq \left \vert
			s_{m}-s_{n}\right \vert +N\leq \delta +N\leq N+N=N<\varepsilon \text{.}
		\end{align*}%
		Hence $\left( u_{n}\right) $ is an $N$-Cauchy sequence.
	\end{example}
	
	Given a neutrix $N$, the Cauchy criterion holds for both $N$-convergence and
	strong convergence.
	
	\begin{theorem}
		\label{Theorem Cauchy completeness} Let $N$ be a neutrix and $(u_{n})$ be a flexible sequence.
		Then
		
		\begin{enumerate}
			\item \label{tc cauchy1} $(u_{n})$ is $N$-convergent if and only if it
			is $N$-Cauchy.
			
			\item \label{tc cauchy2} Let $\alpha \in \mathbb{E}$. Then $(u_{n})$
			strongly converges to $\alpha $ if and only if it is strongly $%
			N(\alpha )$-Cauchy.
		\end{enumerate}
	\end{theorem}
	
	As a consequence we obtain that if $N$ is a nonzero neutrix, then conditions %
	\eqref{N-Cauchy} and \eqref{Strong Cauchy} are equivalent.
	
	\begin{corollary}
		Let $(u_n)$ be a flexible sequence and $N\not=0$ be a neutrix. Then $(u_n)$
		is $N$-Cauchy if and only if it is strongly $N$-Cauchy.
	\end{corollary}
	
	\begin{proof}
		By Theorems \ref{Theorem Cauchy completeness}.\ref{tc cauchy1}, \ref%
		{Stelling sterke convergentie onbegrensd} and \ref{Theorem Cauchy
			completeness}.\ref{tc cauchy2} a sequence $(u_n)$ is $N$-Cauchy if and only
		if it is $N$-convergent if and only if it is $N$-strongly convergent if and
		only if it is strongly $N$-Cauchy.
	\end{proof}
	
	Before we prove Theorem \ref{Theorem Cauchy completeness} we study some
	properties of $N$-Cauchy sequences and strongly $N$-Cauchy sequences and
	present some auxiliary results.
	
	\begin{proposition}
		\label{N-Cauchy bounded} Let $N$ be a neutrix. Every $N$-Cauchy flexible sequence is eventually
		bounded.
	\end{proposition}
	
	\begin{proof}
		Let $(u_{n})$ be a $N$-Cauchy sequence. Let $\varepsilon >N$.
		Then there exists $k\in \N$ such that $|u_{n}-u_{m}|<\varepsilon $ for all $m,n\geq k$. So for $m,n\geq k$
		we have 
		\begin{equation*}
			|u_{n}|\leq |u_{n}-u_{m}|+|u_{m}|<|u_{m}|+\varepsilon .
		\end{equation*}%
		Taking $m=k$ one concludes that $(u_{n})$ is eventually bounded.
	\end{proof}
	
	The components of a flexible $N$-Cauchy sequence are also $N$%
	-Cauchy and the components of a strongly $N$-Cauchy sequence are strongly $N$-Cauchy
	sequences.
	
	\begin{proposition}
		\label{Cauchy-component} Let $N$ be a neutrix. Let $(a_n)$ be a precise sequence and $(A_n)$ be a sequence of neutrices. Let $(u_{n})$ be a flexible
		sequence such that $u_{n}=a_{n}+A_{n}$ for all $n\in \N$. Then 
		
		\begin{enumerate}
			\item \label{cauchy coponent1} the sequence $(u_{n}) $ is $N$-Cauchy if and
			only if $(a_{n})$ and $( A_{n}) $ are both $N$-Cauchy sequences;
			
			\item \label{Cauchy component 2} the sequence $(u_n)$ is strongly $N$-Cauchy	if and only $(a_n)$ and $(A_n)$ are both strongly $N$-Cauchy sequences.
		\end{enumerate}
	\end{proposition}
	
	\begin{proof}
		\ref{cauchy coponent1}. Assume that $(u_n)$ is $N$-Cauchy. Let $\varepsilon>N$%
		. Then there exists $n_0 \in \N$ such that $|a_n+A_n-(a_m+A_m)|<\varepsilon$ for all $%
		m,n\geq n_0$. It follows that $|a_n-a_m|<\varepsilon$ and $|A_n-A_m|<\varepsilon$
		for all $m, n\geq n_0$. Hence $(a_n), (A_n)$ are both $N$-Cauchy.
		
		Conversely, assume that $(a_{n}),(A_{n})$ are both $N$-Cauchy. Let $\varepsilon >N$.
		Then there exists $n_{0}\in \N	$ such that $|a_{n}-a_{m}|<\varepsilon /2$ and $|A_{n}-A_{m}|<\varepsilon /2$ for
		all $m,n\geq n_{0}.$ It follows that for all $n,m\geq n_{0}$
		\begin{equation*}
			|a_{n}+A_{n}-(a_{m}+A_{m})|=|a_{n}-a_{m}|+|A_{n}-A_{m}|<\varepsilon /2+\varepsilon
			/2=\varepsilon.
		\end{equation*}%
		 Hence $(u_{n})$ is $N$-Cauchy.
		
		\ref{Cauchy component 2}. The result follows from the fact that $%
		a_n+A_n-a_m+A_m\subseteq N$ if and only if $a_n-a_m\in N$ and $%
		A_n+A_m\subseteq N.$
	\end{proof}
	
	\begin{lemma}
		\label{conver of neutrix sequence} \label{converge-neutrixs} Let $N$ be a
		neutrix and $(A_{n}) $ be a flexible sequence of neutrices. If $(A_{n}) $ is 
		$N$-Cauchy then $A_{n} \longrightarrow N.$
	\end{lemma}
	
	\begin{proof}
		Let $\varepsilon >N$. Then $\varepsilon /2>N$. Because $(A_{n})$ is $N$ -Cauchy
		there exists $n_{0}\in 
		%TCIMACRO{\U{2115} }%
		%BeginExpansion
		\mathbb{N}
		%EndExpansion
		$ such that $|A_{n}-A_{n+p}|<\varepsilon /2$ for all $n\geq n_{0}$ and for all $%
		p>0$. It follows that for all $n \geq n_0$ and all $p>0$
		\begin{equation*}
			|A_{n}-N|\leq |A_{n+p}+A_{n}-N|=|A_{n+p}-A_{n}|+N<\frac{\varepsilon}{2}+\frac{\varepsilon}{2}=\varepsilon .
		\end{equation*}
		Hence $A_n \longrightarrow N$.
	\end{proof}
	
	\begin{lemma}
		\label{convergence of internal cauchy sequence} Let $N$ be a neutrix and $%
		(a_n)$ be an internal precise sequence. If $(a_n)$ is $N$-Cauchy then $(a_n)$
		is $N$-convergent.
	\end{lemma}
	
	\begin{proof}
		Because $(a_{n})$ is precise and $N$-Cauchy, we have that $(a_{n})$ is
		bounded. Also, since $(a_{n})$ is internal, it admits a subsequence $(a_{m_{n}})
		$ convergent to some $%
		a\in \R$. Let $\varepsilon >N$. Then $\varepsilon /2>N$. There exist $n_{1},n_{2}\in 
		%TCIMACRO{\U{2115} }%
		%BeginExpansion
		\mathbb{N}
		%EndExpansion
		$ such that for all $m_{n}\geq n_{1}$ we have $|a_{m_{n}}-a|<\varepsilon /2$
		and for all $m,n\geq n_{2}$ and for all $p\geq 0$ we have $%
		|a_{m}-a_{n}|<\varepsilon /2.$ Let $n_{0}=\max \{n_{1},n_{2}\}$. Then for all $%
		n\geq n_{0}$ we have 
		\begin{equation*}
		|a_{n}-a|=|a_{n}-a_{m_{n}}+a_{m_{n}}-a|\leq
		|a_{n}-a_{m_{n}}|+|a_{m_{n}}-a|<\varepsilon .
		\end{equation*}
		Hence $(a_n)$
		is $N$-convergent.
	\end{proof}
	
	\begin{lemma}
		\label{convergent of N cauchy} Let $N$ be a neutrix and let $(a_{n})$ be a precise $N$-Cauchy sequence. If $(a_{n})$ has a $N$-convergent subsequence, then it is $N$-convergent.
	\end{lemma}
	
	\begin{proof}
		Let $J \subseteq  \Gamma(a)$ be cofinal with $\Gamma(a)$. Assume that $a_{n{\upharpoonright{\pi(J)}}}\underset{N}{\longrightarrow }a_{0}$ for some $a_0 \in 
		%TCIMACRO{\U{211d} }%
		%BeginExpansion
		\mathbb{R}
		%EndExpansion
		$. Let $\varepsilon >N$. Then $\varepsilon /2>N.$ There exists $n_{1}\in 
		%TCIMACRO{\U{2115} }%
		%BeginExpansion
		\mathbb{N}
		%EndExpansion
		$ such that for all $k\in \pi(J)$, such that $k\geq n_{1}$ we have $|a_{k}-a_0|\leq \varepsilon /2$. Also, since $(a_{n})$ is 
		$N$-Cauchy, there exists $n_{2}\in 
		%TCIMACRO{\U{2115} }%
		%BeginExpansion
		\mathbb{N}
		%EndExpansion
		$ such that $|a_{m}-a_{n}|<\varepsilon /2$ for all $m,n\geq n_{2}$. Let $%
		n_{0}=\max \{n_{1},n_{2}\}$. Because $J$ is cofinal with $(a_n)$, there exists $%
		m_{0}\in \pi(J)$ such that $m_{0}\geq n_{0}$. So for all $n\in 
		%TCIMACRO{\U{2115} }%
		%BeginExpansion
		\mathbb{N}
		%EndExpansion
		,n\geq m_{0}\geq n_{0}$ we have 
		\begin{equation*}
		|a_{n}-a_0|=|a_{n}-a_{m_{0}}+a_{m_{0}}-a_0|%
		\leq |a_{m_{0}}-a_{n}|+|a_{m_{0}}-a_0|<\varepsilon .
		\end{equation*}
		 Hence $(a_{n})$ $N$%
		-converges to $a_0$.
	\end{proof}

	\begin{proof}[Proof of Theorem \protect \ref{Theorem Cauchy completeness}]
		\ref{tc cauchy1}. Assume that $u_{n}\underset{N}{\longrightarrow }\alpha $,
		for some $\alpha \in \mathbb{E}$. Then there exists $n_{0}\in 
		%TCIMACRO{\U{2115} }%
		%BeginExpansion
		\mathbb{N}
		%EndExpansion
		$ such that for $n\geq n_{0}$ we have $\left \vert u_{n}-\alpha \right \vert
		<\varepsilon /2$. So for $m,n \geq n_{0}$ 
		\begin{align*}
			\left \vert u_{m}-u_{n}\right \vert & \leq \left \vert u_{m}-u_{n}+N\left(
			\alpha \right) \right \vert =\left \vert u_{m}-\alpha -u_{n}+\alpha
			\right \vert  \\
			& \leq \left \vert u_{m}-\alpha \right \vert +\left \vert u_{n}-\alpha
			\right \vert <\frac{\varepsilon }{2}+\frac{\varepsilon }{2}=\varepsilon .
		\end{align*}%
		Hence $\left( u_{n}\right) $ is $N$-Cauchy.
		
		Conversely, we assume that $(u_{n})$ is $N$-Cauchy. Then $(a_{n})$ and $%
		(A_{n})$ are also $N$-Cauchy, by Proposition \ref{Cauchy-component}.\ref{cauchy coponent1}. By Corollary \ref{exist internal subsequence},
		$(a_{n})$ has an internal subsequence $(a_{n_{m}})$ which is also $N$-Cauchy. Then $(a_{n_{m}})$ is $N$-convergent by Lemma \ref{convergence of internal cauchy
			sequence} and we conclude that $(a_{n})$ is $N$-convergent by Lemma \ref%
		{convergent of N cauchy}. Also, $A_{n}%
		\longrightarrow N$, by Lemma \ref{conver of neutrix sequence}.
		Hence $(u_{n})$ is $N$-convergent by Proposition \ref{converComponents}.
		
		\ref{tc cauchy2}. Assume that $(u_{n})$ is strongly convergent to some $\alpha =a+A\in \mathbb{E}$. Then there exists $n_{0}\in \mathbb{N}$ such that $%
		u_{n}\subseteq \alpha $ for all $n\geq n_{0}$. It follows that $%
		u_{n}-a\subseteq A$ for all $n\geq n_{0}$. So for all $n,m\geq n_{0}$ we
		have $u_{n}-u_{m}=u_{n}-a+a-u_{m}\subseteq A+A=A$. Hence $(u_{n})$ is
		strongly $A$-Cauchy.
		
		Conversely, assume that $N$ is a neutrix and $(u_{n})$ is strongly $N$%
		-Cauchy. If $N=0$, there exists $n_{0}\in 
		%TCIMACRO{\U{2115} }%
		%BeginExpansion
		\mathbb{N}
		%EndExpansion
		$ such that $u_{n}=a$ for all $n\geq n_{0}$. So $u_{n}\rightsquigarrow a.$
		If $N\not=0$ then the sequence $(u_{n})$ is $N$-Cauchy. By Theorem \ref{Theorem Cauchy completeness}.\ref{tc
			cauchy1} the sequence $(u_{n})$ is $N$-convergent and by Theorem \ref{Stelling
			sterke convergentie onbegrensd} it is strongly convergent.
	\end{proof}
	
	\subsection{Local Cauchy sequences}
	\begin{definition}\label{Definition local Cauchy}
		Let $S,C\subseteq 
		%TCIMACRO{\U{2115} }%
		%BeginExpansion
		\mathbb{N}
		%EndExpansion
		$ be initial segments of $%
		\mathbb{N}
		$ with $C\subseteq S$. Let $u:S \rightarrow %
		\R
		$ be a local flexible sequence and $N\subseteq 
		\R
		$. We say that $(u_n)$ is $N$-Cauchy on $S$ with respect to $C$ if 
		\begin{equation*}
			\forall \varepsilon >N\exists c\in C\forall m,n\in S(m \geq c \wedge  n \geq c \Rightarrow \left \vert u_{n}-u_{m}\right \vert \leq \varepsilon ).
		\end{equation*}
	\end{definition}

	We recall that the external numbers satisfy a generalized form of Dedekind Completeness \cite[Cor. 4.25]{vdbnaa} (see also \cite[Thm. 3.2]{dinisberg 2016}). The element immediately above (resp. below) the interval is called weak supremum (resp. weak infimum).

	\begin{theorem}[Generalized Completeness]
	Let $C$ be a convex subset of $\R$. Then there exist two unique additive convex subgroups $K$ and $L$ of $\R$ and two real numbers $a$ and $b$ such that $C$ takes exactly one of the following forms:
	\begin{enumerate}
	\item $C= [a,b] \cup \left(a+L^-\right) \cup \left(b+K^+\right)$;
	\item $ C= [a,b[ \cup \left(a+L^-\right) \setminus \left(b+K^-\right)$;
     \item $C= \left([a,b] \setminus \left(a+L^+\right)\right) \cup \left(b+K^+\right)$;
     \item $C= \left([a,b] \setminus \left(a+L^+\right)\right) \setminus \left(b+K^-\right)$;
	\end{enumerate}
	where $L^-$ and $L^+$ denote respectively the negative and positive elements of $L$, and similarly to $K$.
	\end{theorem}
	
	\begin{theorem}
		\label{Stelling Cauchy limiet}Let $S,C\subseteq 
		\mathbb{N}
		$ be initial segments of $ \N$ such that $C\subseteq S$. Let $N$ be a neutrix and $u:S \rightarrow %
		\mathbb{E}
		$ be a $N$-Cauchy sequence on $S$ with respect to $C$. Then there
		exists $a\in 
		\mathbb{R}
		$ such that $(u_n)$ converges to $a+N$ on $S$ with respect to $C$.
	\end{theorem}
	
	\begin{proof}
		The case where $ C= \mathbb{N} $ is contained in Theorem \ref{Theorem Cauchy completeness}. So we may assume that $ C $ is bounded in $  \mathbb{N} $. If $ C $ is internal, it has a maximum, say $\mu  $. Then it follows from Definitions \ref{Definition local Cauchy} and \ref{Definition convergence local} that $(u_n) $ converges to $ u_{\mu}+N $ with respect to $ C $. If $ C $ is external and $ C\subset S $, let $ \nu\in S\setminus C $ be arbitrary. Again it follows from Definitions \ref{Definition local Cauchy} and \ref{Definition convergence local} that $(u_n) $ converges to $ u_{\nu}+N $ with respect to $ C $.
		  
		From now on we suppose that $ C $ is external and $ C=S $. Put $ D=\mathbb{R}\setminus C $. Assume first that there exists $k\in C$ such that $u_{n}-u_{m}\subseteq N$ for all 
		$m,n\in C$ with $m,n\geq k$. Then $(u_n)$ converges (strongly) to $u_{k}+N$.
		In the remaining case, for all $m\in C$ there exists $n\in C,n>m$ such that $u_{n}-u_{m} \nsubseteq N$; then either $ u_{m}\subset u_{n}$, or  $ u_{m}+N< u_{n}$, or $ u_{m}+N>u_{n}$. Observe that in the first and second cases there exists a real number $ \eta $ such that $%
		u_{m}+N<\eta\leq u_{n}$, and in the third case there exists a real number $ \eta $ such that $%
		u_{n}<\eta< u_{m}+N$; in the latter case we may also assume that $ \eta $ is such that $ u_{n}+N<\eta\leq u_{m}+N $. So it may happen that (i) there exists $ k\in C $ such that for all $m\in C, m\geq k$ there exists $n\in C,n>m$ and there exists a real number $ \eta $ such that $%
		u_{m}+N<\eta\leq u_{n}$, or (ii) there exists $ k\in C $ such that for all $m\in C,m\geq k$ there exists $n\in C,n>m$ such that  $u_{n}+N<\eta\leq u_{m}+N $.
		
		In case (i), without restriction of generality we may assume that $ k=0 $. Put
		\begin{equation*}
			L=\bigcup \left\lbrace \left[ s,t\right] :s\in u_{0}\wedge t\in u_{m}\wedge s\leq t \wedge m\in C \right\rbrace  .
		\end{equation*} 
		By Generalized
		Completeness $L$ has a weak supremum of the form $\alpha =a+B$ with $a\in \R$ and $B$ a neutrix. We show that $B=N$ and $\alpha
		=\lim_{n\rightarrow D}u_{n}$.
		
		Suppose $B\subset N$. Let $\delta \in \R$ be such that $B<\delta \leq N$. Then there exists $m\in C$ such that $%
		a-\delta \leq u_{m} $. Let $n\in C,n>m$ and $ \eta \in \R $ be such that $	u_{m}+N<\eta\leq u_{n}$. Then 
		\begin{equation*}
			a+B\leq u_{m} + \delta +B\leq u_{m}+N+B=u_{m}+N<\eta\leq u_{n},
		\end{equation*}
		in contradiction with the fact that $\alpha $ is the weak
		supremum of $L$. Hence $N\subseteq B$.
		
		Suppose that $N\subset B$.  Assume first that $ a\in L $. Then $a\leq u_{m}$ for some $m\in C$. Let again $n\in C,n>m$ and $ \eta \in \R $ be such that $	u_{m}+N<\eta\leq u_{n}$. Then $ a+B \leq u_{m}+N<\eta\leq u_{n}$. So $ a+B $ cannot be the supremum of $ L $, a contradiction. Secondly, assume that $a \notin L$. Let $ \varepsilon\in \R $ and $N<\varepsilon \leq B$ and $p \in C$ be such that $\left \vert u_{n }-u_{m }\right \vert <\varepsilon $ for all $n,m \in C$ with $n,m\geq p$. Then $ u_{p}< a+B $, hence also $ u_{p}+B < a+B $. Then $u_{n}< u_{p}+2\varepsilon$ for all $n\in C,n\geq p$, and because $u_{p}+2\varepsilon\leq u_{p}+B< a+B $  
		again $ a+B $ cannot be the supremum of $ L $. We conclude that 
		$N=B$. 
		
		Finally, let $\zeta \in \R$ be such that $N<\zeta $. Because $N<\zeta /2$, by the $ N $-Cauchy property there exists $q\in C$ such that $\left \vert u_{n}-u_{m}\right \vert \leq \zeta /2$ for all $n\in C,n\geq q
		$, and because $\alpha $ is the weak supremum of $ L $, there exists $r\in C$ such that $%
		r\geq q$ and $\left \vert u_{r}-\alpha \right \vert <\zeta /2$. Then for all
		for all $n\in C,n\geq r$%
		\begin{equation*}
			\left \vert u_{n}-\alpha \right \vert \leq \left \vert u_{n}-u_{r}\right \vert
			+\left \vert u_{r}-\alpha \right \vert <\frac{\zeta}{2}+\frac{\zeta}{2}=\zeta .
		\end{equation*}%
		Hence $\alpha =\lim_{n\rightarrow D}u_{n}$.
		
		Case (ii) is analogous to the former case, now working with
		a weak infimum. This completes the proof.
	\end{proof}

	\section{Applications}
	
	\label{Section applications} 
		\subsection{Nonstandard Borel-Ritt Theorem}
	Expansions in $ \varepsilon $-shadow of numbers have been introduced by Francine Diener as a nonstandard alternative to asymptotic expansions of functions, in fact the two notions coincide for standard functions, see \cite	{Diener-Reeb}. Given an asymptotic expansion, the classical Borel-Ritt theorem shows the existence of a real function having this expansion. This theorem is particularly useful for divergent series: they are not empty, in the sense that they are always the expansion of some function. Proofs of existence of real numbers having a prescribed expansion in $ \varepsilon $-shadow are given in  \cite	{Diener-Reeb}, and \cite{vdbnaa}, they form a sort of nonstandard Borel-Ritt theorems.
	
	Let $\varepsilon \simeq 0$. We interpret here an expansion in $ \varepsilon $-shadow in terms of a $N$-Cauchy sequence. Here $N$ is the set of all real numbers having shadow expansion identicaly zero; i.e. the so-called \emph{microhalo}	$ m_{\varepsilon}:=\bigcap _{\st(n)\in \mathbb{N}} [-\varepsilon^{n},\varepsilon^{n}]$. The microhalo is a neutrix built from unlimited powers of $ \varepsilon $, and may be denoted by $ m_{\varepsilon}=\pounds \varepsilon ^{\centernot{\infty}} $. By Theorem \ref{Stelling Cauchy limiet} such a $N$-Cauchy sequence has a limit. This yields a new proof of such a nonstandard Borel-Ritt theorem, in fact it defines the set of real numbers having a given expansion in $ \varepsilon $-shadow in the form of an external number. 
	
	We recall that the \emph{shadow} of a limited real number $x$, denoted by $^{\circ}x$, is the unique standard real number infinitely close to $x$.
	\begin{definition}
		Let $\varepsilon$ be a positive infinitesimal real number, $(a_{n})$ be a standard sequence of real numbers and $b\in \R$. Then the formal series $ \sum_{n=0}^{\infty}a_{n}\varepsilon ^{n} 	$ is called an \emph{expansion in $\varepsilon $-shadow} of $b$ if 
		\begin{equation*}
			^{^\circ}\left( \frac{b-\sum_{k=0}^{n}a_{k}\varepsilon ^{k}}{\varepsilon ^{n+1}}%
			\right) =a_{n+1}
		\end{equation*} for all standard $ n\in \mathbb{N} $. 
	\end{definition}

	\begin{theorem}[Nonstandard Borel-Ritt Theorem]
		\label{Stelling Borel-Ritt}Let $\varepsilon$ be a positive infinitesimal real number and $(a_{n})$ be a standard sequence of real numbers. Then there exists $b\in 
		\R$ such that $\sum_{n=0}^{\infty}a_{n}\varepsilon ^{n}$ is an expansion in $\varepsilon $-shadow of $b$.
	\end{theorem}
	
	\begin{proof}
		Put $s_{n}=\sum_{k=0}^{n}a_{k}\varepsilon ^{k}$ for $n\in 
		\N$. Let $\eta >\pounds \varepsilon ^{\centernot{\infty}}$. Then there exists a standard $%
		m\in \N$ such that $\varepsilon ^{m}<\eta $. Then for standard $n\in \N$ with $n\geq m+1$ we have 
		\begin{equation*}
		\left \vert s_{n}-s_{m}\right \vert =\left \vert \sum_{k=m+1}^{n}a_{k}\varepsilon
		^{k}\right \vert 
		\subset \pounds \varepsilon ^{m+1}<\eta 
		\end{equation*}%
		Hence $s$ is a $\pounds \varepsilon ^{\centernot{\infty}}$-Cauchy sequence with
		respect to $\pounds $. By Theorem \ref{Stelling Cauchy limiet} there exists $%
		b\in 
		%TCIMACRO{\U{211d} }%
		%BeginExpansion
		\mathbb{R}
		%EndExpansion
		$ such that $\lim_{n\rightarrow \centernot{\infty}}s_{n}=b+\pounds \varepsilon ^{\centernot%
			{\infty}}$. Let $n\in 
		%TCIMACRO{\U{2115} }%
		%BeginExpansion
		\mathbb{N}
		%EndExpansion
		$ be standard.
		Then there exists standard $p\in 
		%TCIMACRO{\U{2115} }%
		%BeginExpansion
		\mathbb{N}
		%EndExpansion
		,p>n$ such that $b-s_{p}\in \oslash \varepsilon ^{n+1} $. Then 
		\begin{equation*}\begin{array}{lll}
				\frac{b-\sum_{k=0}^{n}a_{k}\varepsilon ^{k}}{\varepsilon ^{n+1}} &=&\frac{%
				b-\sum_{k=0}^{p}a_{k}\varepsilon ^{k}+\sum_{k=n+1}^{p}a_{k}\varepsilon ^{k}}{%
				\varepsilon ^{n+1}} \\
			&\in &\frac{\oslash \varepsilon ^{n+1} +a_{n+1}\varepsilon ^{n+1}+\oslash \varepsilon ^{n+1}}{\varepsilon
				^{n+1}} \\
			&\subseteq &a_{n+1}+\oslash.
			\end{array}
		\end{equation*}%
		Hence $^{^\circ}\left( \frac{b-\sum_{k=0}^{n}a_{k}\varepsilon ^{k}}{\varepsilon ^{n+1}}%
		\right) =a_{n+1}$, meaning that $\sum_{n=0}^{\infty}a_{n}\varepsilon ^{n}$ is an expansion in $\varepsilon $-shadow of $b$.
	\end{proof}

	\subsection{Matching Principles}\label{Subsection Matching Principles}
	
	Theorem \ref{Stelling sterke convergentie intern} on internal sequences may
	be seen as a form of "matching". Indeed the
	internal sequence exhibits two approximative behaviors: being outside a limit set, while approaching it, and being inside this limit set. Observe that we have only information on the sequence outside the limit set, and that the theorem adds information on the sequence within the limit set: it enters and stays there for some time. Matching is
	particularly relevant to singularly perturbated differential equations. Usually one has separate information on slow behavior of solutions close to a slow curve, and fast behavior when approaching it. These separate behaviors have to be "matched" into an overall description of the behavior of a solution, showing how 
	the transition from fast behavior to slow behavior is made, and vice-versa. 
	
		Classical tools are based on the Du Bois-Reymond Lemma \cite{Hardy} on
	persistence of convergence outside a sequence of convergent functions or Van
	Dyke's Principle \cite{VanDyke} on equality of terms in two locally valid
	asymptotic expansions. The practice shows that sometimes a general result
	can be applied, such as the Extension Principle of Eckhaus \cite{Eckhaus},
	but mostly matching has to be proved case-by-case, with {\it ad hoc} methods.
	
    Somewhat surprisingly, Nonstandard Analysis happened to yield a general method which is based on logic, \textit{in casu} the Fehrele Principle mentioned in Section \ref{Section_nature_external_sequences} saying that no halo is a galaxy \cite{Diener-Van den Berg}. Alternatively one may say also that an external universal formula of the form $ \forall^{\st}x\Phi(x,y) $ can never be equivalent to an external existential formula of the form $ \exists^{\st}x\Psi(x,y) $, here $ \Phi,\Psi $ are internal; more precisely, if they were equivalent, they would be equivalent to an internal formula. The well-known Robinson's Lemma \cite{Robinson} (discovered before) is seen to be an instance of the Fehrele Principle, and illustrates how it is used in matching. Indeed, let $ (u_{n}) $ be an internal sequence. Suppose that it is proved that $ u_{n}\simeq 0 $ for all standard $ n\in \N $. Then there exists some unlimited $ \nu\in \N $ such that still $ u_{\nu}\simeq 0 $, meaning that the behavior "being infinitesimal" matches partly with the (possibly different) behavior of the sequence outside the set $ ^{\sigma}\N $ where it originally was proved. As for the proof, note that the galaxy $ ^{\sigma}\N $ must be stricly included in the prehalo $ H=\{n: u_{n}\simeq 0\}$, by the Cauchy Principle if $ H $ is internal and by the Fehrele Principle if $ H $ is external.  

	We start by adapting the definitions on convergence and strong convergence
	of sequences to functions, and relating convergence and strong convergence
	of sequences and functions, which are assumed to be continuous.
	
	\begin{definition}
		Let $(C,D)$ be a cut of $\R$ into a nonempty initial segment $C$ and a nonempty final segment $D$. Let $%
		S$ be an initial segment of $\R$ such that $C\subseteq S$. Let $f:S\rightarrow \R$ be an internal function and $\alpha =a+A\in \mathbb{E}$. We say that $f$ 
		\emph{converges} to $\alpha $ on $S$ with respect to $C$ if%
		\begin{equation*}
		\forall \varepsilon >A\exists c\in C\forall y\in S(y\geq
		c\Rightarrow \left \vert f(y)-\alpha \right \vert < \varepsilon ).
		\end{equation*}%
		We say that $f$ \emph{strongly converges} to $\alpha $ on $S$ with respect
		to $C$ if there exists a nonempty final segment $T$ of $S$ such that $%
		D\subseteq T$ and 
		\begin{equation*}
		\forall y\in S(y\in T\Rightarrow f(y)\in \alpha ).
		\end{equation*}%
		In the first case we write $\lim_{x\rightarrow D}f(x)=\alpha $ and in the
		second case \emph{Lim}$_{x\rightarrow D}f(x)=\alpha $. If in the above $C=S$%
		, or $S=\R$, we simply speak of (strong) convergence with respect to $C$.
	\end{definition}
	
	\begin{lemma}
		\label{Lemma rij functie}Let $(C,D)$ be a cut of $\R^{+}$ into an external initial segment $C$ and a final segment $D$. Let $S$ be an (internal) initial segment of $%
		%TCIMACRO{\U{211d} }%
		%BeginExpansion
		\mathbb{R}
		%EndExpansion
		$ such that $C\subset S$. Let $f:S\rightarrow %
		%TCIMACRO{\U{211d} }%
		%BeginExpansion
		\mathbb{R}
		%EndExpansion
		$ be an internal continuous function and $\alpha =a+A\in \mathbb{E}$. For every $%
		t\in 
		%TCIMACRO{\U{211d} }%
		%BeginExpansion
		\mathbb{R}
		%EndExpansion
		,t>0$, define $%
		S_{t}=\left \{ n\in \N: nt\in S\right \} $ and $u_{t}:S_{t}\rightarrow \mathbb{R}	$ by $u_{t}(n)=f(tn)$. Then
		
		\begin{enumerate}
			\item \label{Lemma rij functie1}$f$ converges to $\alpha $ on $S$ with respect to $C$ if and only if $%
			u_{t}$ converges to $\alpha $ on $S_{t}$ with respect to $C/t$ for all $t\in 
			%TCIMACRO{\U{211d} }%
			%BeginExpansion
			\mathbb{R}
			%EndExpansion
			,t>0$ be such that $c+t\in C\Leftrightarrow c\in C$ for all $c\in 
			%TCIMACRO{\U{211d} }%
			%BeginExpansion
			\mathbb{R}
			%EndExpansion
			$.
			
			\item \label{Lemma rij functie2} $f$ converges strongly to $\alpha $ on $S$ with respect to $C$ if and
			only if $u_{t}$ converges to $\alpha $ on $S_{t}$ with respect to $C/t$ for
			all $t\in 
			%TCIMACRO{\U{211d} }%
			%BeginExpansion
			\mathbb{R}
			%EndExpansion
			,t>0$ be such that $c+t\in C\Leftrightarrow c\in C$ for all $c\in 
			%TCIMACRO{\U{211d} }%
			%BeginExpansion
			\mathbb{R}
			%EndExpansion
			$.
		\end{enumerate}
	\end{lemma}
	
	\begin{proof}
		Without restriction of generality we may suppose that $\alpha =A$. 
		
		\ref{Lemma rij functie1}. Assume first that $f$ converges to $\alpha $ on $S$ with respect to $%
			C$.  Let $\varepsilon \in 
			%TCIMACRO{\U{211d} }%
			%BeginExpansion
			\mathbb{R}
			%EndExpansion
			,\varepsilon >A$. Then there is $x\in C$ such that $\left \vert f(y)\right \vert
			<\varepsilon $ for all $y\in S,y\geq x$. Let $%
			t\in 
			%TCIMACRO{\U{211d} }%
			%BeginExpansion
			\mathbb{R}
			%EndExpansion
			,t>0$ be such that $c+t\in C\Leftrightarrow c\in C$ for all $c\in 
			%TCIMACRO{\U{211d} }%
			%BeginExpansion
			\mathbb{R}
			%EndExpansion
			$. Let $m\in 
			%TCIMACRO{\U{2115} }%
			%BeginExpansion
			\mathbb{N}
			%EndExpansion
			$ be minimal such that $mt\geq x$. Then $\left \vert u_{t}(n)\right \vert
			<\varepsilon $ for all $n\in S_{t},n\geq m$. Hence $u_{t}$ converges to $A$ on $%
			S_{t}$ with respect to $C/t$.\newline
			As for the converse, let $\varepsilon \in 
			%TCIMACRO{\U{211d} }% 
			%BeginExpansion
			\mathbb{R}
			%EndExpansion
			,\varepsilon >A$. Let $s\in S$ be such that $C<s$. By uniform continuity of $f$
			on $[0,s]$ there exists $t>0$ such that $c+t\in C\Leftrightarrow c\in C$
			for all $c\in 
			%TCIMACRO{\U{211d} }%
			%BeginExpansion
			\mathbb{R}
			%EndExpansion
			$ and $\left \vert f(v)-f(w)\right \vert <\varepsilon /2$ for all $v,w\in \lbrack
			0,s]$ with $\left \vert v-w\right \vert <t$. Because $u_{t}$ converges to $A$
			on $S_{t}$ with respect to $C/t$, there exists $m\in C/t$ such that $%
			\left \vert u_{t}(n)\right \vert <\varepsilon /2$ for all $n\in S_{t},n\geq m$.
			Let $x\in \lbrack 0,s]$ be such that $x>mt$ and $n\in C/t$ be minimal such
			that $x\leq nt$. Then%
			\begin{equation*}
			\left \vert f(x)\right \vert =\left \vert f(x)-f(nt)\right \vert +\left \vert
			u_{t}(n)\right \vert <\frac{\varepsilon}{2}+\frac{\varepsilon}{2}=\varepsilon .
			\end{equation*}%
			Hence $f$ converges to $A$ on $[0,s]$ with respect to $C$. Because $s$ is
			arbitrary $f$ converges to $A$ on $S$ with respect to $C$.
			
			\ref{Lemma rij functie2}. The proof is similar to the proof of Lemma \ref{Lemma rij functie}.\ref{Lemma rij functie1}.
	\end{proof}
	
	\begin{theorem}[Matching for sequences]
		Let $S$ be an initial segment of $%
		%TCIMACRO{\U{2115} }%
		%BeginExpansion
		\mathbb{N}
		%EndExpansion
		$. Let $u:S		\rightarrow		\mathbb{R}		$ be an internal sequence and $\alpha =a+A\in \mathbb{E}$, with $A \neq 0$. Assume $u_{0}\notin \alpha $. Let
		\begin{equation}
		C=\left \{ n\in S : \forall m(0\leq m\leq n\Rightarrow u_{m}\notin
		\alpha )\right \}.   \label{Formule C u}
		\end{equation}%
		If $u_{\upharpoonright C}$ converges to $\alpha $ with respect to $C$, then there
		exists $\nu \in S,\nu >C$ such that $u$ converges strongly to $\alpha $ with
		respect to $C$ on $\{0,...,\nu \}$.
	\end{theorem}
	
	\begin{proof}
		Without restriction of generality we may suppose that $\alpha =A$. We put $D=S\setminus C$.
		
		We show first that $C$ is external. Indeed, if  $C$ is internal it must be finite
		as a consequence of Theorem \ref{Stelling sterke convergentie onbegrensd}. Let $\mu :=\min \{ \left \vert u_{m}\right \vert : m\in	C \wedge u_{m}\notin A \}$. Then $\mu >A$, hence $u_{\upharpoonright C}$ does
		not converge to $A$ with respect to $C$, a contradiction. Hence $C$ is external. As a consequence $C\subset S$
		and we conclude that $D\neq \emptyset $.
		
		It 	follows from (\ref{Formule C u}) that if $A$ is a halo, the set $C$ is a complement of a prehalo, hence, being external it is a galaxy, and similarly, if $ A$ is a galaxy 
		the set $C$ is a halo. Then, knowing that $u_{\upharpoonright C}$ converges to 
		$A$ with respect to $C$, by Theorem \ref{Stelling sterke convergentie intern}
		there exists $\nu \in S,\nu >C$ such that $u$ converges strongly to $\alpha $
		with respect to $C$ on $\{0,...,\nu \}$.
	\end{proof}
	
	\begin{theorem}[Matching for functions]\label{Matching for functions}
		Let $S$ be an initial segment of $%
		%TCIMACRO{\U{211d} }%
		%BeginExpansion
		\mathbb{R}
		%EndExpansion
		^{+}$. Let $f:S\rightarrow %
		%TCIMACRO{\U{211d} }%
		%BeginExpansion
		\mathbb{R}
		%EndExpansion
		$ be an internal continuous function and $\alpha =a+A\in \mathbb{E}$, with $A
		$ external. Assume $f(0)\notin \alpha $. Let
		\begin{equation}
		C=\left \{ y\in S : \forall x(0\leq x\leq y\Rightarrow f(x)\notin
		\alpha ) \right \}.   \label{Formule C}
		\end{equation}%
		If $f_{\upharpoonright C}$ converges to $\alpha $ with respect to $C$, then there
		exists $s\in S,s>C$ such that $f$ converges strongly to $\alpha $ with
		respect to $C$ on $[0,s]$.
	\end{theorem}
	
	\begin{proof}
		Without restriction of generality we may suppose that $\alpha =A$ and $%
		f(0)>A $. We put $D=S\setminus C$.
		
		We show first that $D\neq \emptyset $. If $C=S$, the internal set $f(S)\cap
		\lbrack 0,f(0)]$ is contained in the external set $(A,f(0)]$. By the Cauchy
		Principle $f(S)\subset (A,f(0)]$ Hence there exists $\varepsilon \in 
		%TCIMACRO{\U{211d} }%
		%BeginExpansion
		\mathbb{R}
		%EndExpansion
		,\varepsilon >A$ such that $f(x)\geq \varepsilon $ for all $x\in S$.
		Hence $f$ does not converge to $A$, a contradiction. Hence $C\subset S$ and
		we conclude that $D\neq \emptyset $. Choose $ d\in D $, then $f$ is
		uniformly continuous on $[0,d ]$. Hence there exists $\tau >0,\tau \leq t$
		such that $\left \vert f(v)-f(w)\right \vert \in A$ for all $v,w\in
		\lbrack 0,d ]$ with $\left \vert v-w\right \vert <\tau $.
		
		Let $S_{\tau}=\left \{ n\in \N	: n\tau\in S \right \} $ and define $u_{\tau}:S_{\tau}\rightarrow \R$
		by $u_{\tau}(n)=f(\tau n)$. Then the local sequence $u_{\tau_{\upharpoonright C/\tau}}$
		converges to $A$ with respect to $C/\tau$. It follows from (\ref{Formule C}) and the continuity of $f$ that $C$ is a
		galaxy if $A$ is a halo, and $C$ is a halo if $A$ is a galaxy. Then by Theorem \ref{Stelling sterke convergentie
			intern} there exists $\nu \in D/\tau$ such that the sequence $u_{\tau}$ converges
		strongly to $A $ on $\{0,...,\nu \}$ with respect to $C/\tau$. Let $x\in
		\lbrack 0,\nu \tau ]\cap D$. Let $n\in \{0,...,\nu \}$ be minimal
		such that $x\leq n\tau$. Then 
		\begin{equation*}
		\left \vert f(x)\right \vert =\left \vert f(x)-f(n\tau)\right \vert +\left \vert
		u_{\tau}(n)\right \vert \in A+A=A.
		\end{equation*}%
		Hence $f$ converges strongly to $A$ on $[0,\nu \tau ]$ with respect to $C
		$.
	\end{proof}
	
	We present here two examples of matching in the setting of singular
	perturbations. Consider the differential equation%
	\begin{equation}
	\varepsilon \frac{dy}{dt}=f\left( t,y\right)   \label{singulierestoring}
	\end{equation}%
	with $\varepsilon >0,\varepsilon \simeq 0$. We assume that $f:%
	%TCIMACRO{\U{211d} }%
	%BeginExpansion
	\mathbb{R}
	%EndExpansion
	^{2}\rightarrow 
	%TCIMACRO{\U{211d} }%
	%BeginExpansion
	\mathbb{R}
	%EndExpansion
	$ is internal, continuously differentiable (of class $C^{1}$) and limited,
	and $S$-continuously $S$-differentiable at limited points (of class $S^{1}$, see \cite{Diener-Reeb}). Let us write 
	\begin{equation*}
	\frac{dy}{dt}=\frac{f\left( t,y\right) }{\varepsilon }.
	\end{equation*}%
	We observe two types of behavior of solutions $\lambda (t)$: \emph{fast}
	behavior at points where the derivative $\frac{f\left( t,\lambda (t)\right) 
	}{\varepsilon }$ is unlimited, and \emph{slow} behavior at points where the
	derivative $\frac{f\left( t,\lambda (t)\right) }{\varepsilon }$ is limited.
	We define the \emph{slow curve} $L$ by%
	\begin{equation*}
	L=\text{}^{\circ}\left \{ \left( t,y\right) \in 
	%TCIMACRO{\U{211d} }%
	%BeginExpansion
	\mathbb{R}
	%EndExpansion
	^{2}: f\left( t,y\right) \simeq 0 \right \} .
	\end{equation*}%
	Typically the slow curve consists of one or more standard differentiable
	curves; to fix ideas we consider the case of only one differentiable curve, say $ g $,
	which we take to be identically zero, and defined on a standard interval $ [a,b] $ with $ a<b $. Consider the halo $H:= [a,b]\times \oslash	$; here the derivative of the solution is no longer of order $1/\varepsilon $.
	Observe that the galaxy $G:= \left \{ \left( t,y\right) \in [a,b]\times 
	%TCIMACRO{\U{211d} }%
	%BeginExpansion
	\mathbb{R}
	%EndExpansion
	: f\left( t,y\right) \in \pounds \varepsilon \right \} $ is
	strictly included in $H$. So $H\setminus G$ may be seen as an intermediate
	region, where the transition between fast behavior
	and slow behavior takes place.
	
	All these regions are external. To obtain a global description of the
	behavior of a solution, which is internal, the behaviors of the solution on
	those external regions need to be matched.
	
    The slow curve $ g $ is said to be \emph{attractive}
	if%
	\begin{equation*}	
	f_{2}^{\prime }\left( t,0\right) \lnapprox 0 , 
	\end{equation*}%
	for $  a\leq t\leq b 	 $. Observe that using the $ S $-continuity of $ f_{2}^{\prime } $ one proves that there exists standard $c\in 
	\R	,c>0$ such that for all $ t $ with $  a\leq t\leq b 	 $ and $\left \vert y\right \vert \leq c $
	\begin{equation}
		\begin{cases}
	f\left( t,y\right)\lnapprox 0 &   y\gnapprox 0 \\ 
	f\left( t,y\right)\simeq 0 &   y\simeq 0 \\ 
	f\left( t,y\right)\gnapprox 0 &   y\lnapprox 0%
	\end{cases}%
	,  \label{voorwaardenf}
	\end{equation}%
	meaning that trajectories approach $g$ at points at appreciable distance to $%
	g$, less than $c$.
	
	Next theorem gives a first matching result: solutions approaching the slow curve become infinitely close to the slow curve.	
	\begin{theorem}
		\label{tyhknonov}Consider the equation (\ref{singulierestoring}), with $\varepsilon
		\simeq 0,\varepsilon >0$ and $f$ of class $C^{1}\cap S^{1}$. Let $a,b$ be
		limited and such that $a\lnapprox b$ and let $g:[a,b]\rightarrow 
		\mathbb{R}
		$, defined by $g(t)=0$, be an attractive slow curve. Let $\lambda $ be a
		solution such that $\lambda (a)\not \simeq 0$. Then there exists $\tau
		>a$ such that $\lambda (\tau)\simeq 0$.
	\end{theorem}
	
	\begin{proof}
		Without loss of generality we may suppose that $a=0$, that the value of $c$
		such that (\ref{voorwaardenf}) holds on $[a,b]\times \lbrack -c,c]$ is equal
		to $1$, and that $\lambda (0)=1$. Let 
		\begin{equation*}
		C=\left \{ t\geq 0: \forall s (0\leq s\leq t \wedge\lambda (s)\gnapprox
		0 )\right \} .
		\end{equation*}%
		Since possible singular points only may occurr at infinitesimal
		values, as long as $ t$ is such that $\lambda (s)> \oslash  $ for all $ s $ with $ 0\leq t $, by the Existence Theorem for Differential Equations $ \lambda (s) $ is defined up to $ t $. Then by the Mean Value Theorem there exists 
		$\overline{t}$ such that $0\leq \overline{t}\leq t $ and
		\begin{equation}\label{labda1}
		\lambda (t )-1=\lambda (t )-\lambda (0)=\lambda ^{\prime }\left(\overline{t}\right)t
		=\frac{f\left(\overline{t},\lambda \left(\overline{t}\right)\right)}{\varepsilon }t.
		\end{equation}
		We show first that $\lambda $ converges
		to $\oslash $ with respect to $C$. Let $0\lnapprox \eta \leq 1$. Suppose that $ \lambda (t )> \eta $ for all $ t\in [0,b] $. Then it follows from \eqref{labda1} that $\lambda (t )-1 $, hence also $\lambda (t ) $, is negative unlimited for all appreciable $ t\in [0,b] $, a contradiction. Hence there exists $t
		_{\eta }\in [0,b] $, such that $\lambda (t _{\eta })=\eta $. Because $\lambda $ remains decreasing for times $t\in C$, one
		has $\lambda (t)< \eta $ as long as $t\in C,t> t _{\eta } $. Hence $\lambda $ converges
		to $\oslash $ with respect to $C$.
		
		By Theorem \ref{Matching for functions} the solution $ \lambda $ strongly converges to $%
		\oslash $ with respect to $ C $ on some interval $[0,\tau ]$.
		\end{proof}

	The well-known Theorem of Tykhonov \cite{Tykhonov} (see \cite{Sari} for a nonstandard proof) says that a solution remains close to the
	slow curve, at least until meeting a singular point $(s,0)$, where $%
	f_{2}^{\prime }\left( s,0\right) =0$. Typically $f_{2}^{\prime }\left(
	t,0\right) >0$ for $t>s$, so the slow curve turns repulsive, and nearby
	solutions may leave. Also, solutions become closer to the slow curve, in fact the distance becomes of order $ \varepsilon \pounds $; this refinement of Theorem \ref{tyhknonov} can also be proved using strong convergence.

	\begin{theorem}
		\label{tyhknonov epsilon}Consider the equation (\ref{singulierestoring}), with $%
		\varepsilon \simeq 0,\varepsilon >0$ and $f$ of class $C^{1}\cap S^{1}$. Let $a,b$
		be limited with $a\lnapprox b$ and let $g:[a,b]\rightarrow 
		%TCIMACRO{\U{211d} }%
		%BeginExpansion
		\mathbb{R}
		%EndExpansion
		$, defined by $g(t)=0$, be an attractive slow curve. Let $\lambda $ be a
		solution such that $\lambda (a)\not \simeq 0$. Then there exists $\tau
		>a$ such that $\lambda (\tau )\in \varepsilon \pounds $.
	\end{theorem}
	
	\begin{proof}
		Without loss of generality we may suppose that $a=0$, that the value of $c$
		such that (\ref{voorwaardenf}) holds on $[a,b]\times \lbrack -c,c]$ is equal
		to $1$, and that $\lambda (0)=1$. Let 
		\begin{equation*}
		C=\left \{ t\geq 0: \forall s (0\leq s\leq t \wedge \lambda (s)>\varepsilon 
		\pounds )\right \} .
		\end{equation*}%
		Noting that $ \frac{f(t,\lambda (t))}{\varepsilon }$ is negative unlimited for all $ t\in C $, the proof that $\lambda $ converges to $\varepsilon \pounds $ with respect to $C$ is similar to the argument used in the proof of Theorem \ref{tyhknonov}. By Theorem \ref{Matching for functions} the convergence is strong.
	    	\end{proof}

	\subsection{External recurrence relations and near stability}\label{subsection stability}
	We apply some aspects of the theory of flexible sequences to external recurrence relations and properties of near stability. To that end, we need to be able to define flexible sequences by recursion as well. This will be done in Definition \ref{recurrentsolution} below, through internal representatives. To see that some care is needed we consider the example of sequences $(\alpha_n)$ defined by $ \alpha_0=1 $ and 
	\begin{equation}  \label{pt1}
		\alpha_n=\oslash \alpha_{n-1}.
	\end{equation}
	Observe that we may not write naively $ \alpha_n=\oslash^{n} $ for $ n\in \mathbb{N} $, in the sense of
	\begin{equation*}
	\alpha_n=\{x_n : x_n=\varepsilon x_{n-1} \wedge \varepsilon \simeq 0\},
	\end{equation*}
	for we would obtain that
	\begin{equation*}
	\alpha_n=%
	\begin{cases}
	\oslash & \mbox{ if } n=2k-1 \\ 
	\oslash^+ & \mbox{ if } n=2k%
	\end{cases}.%
	\end{equation*}

	We would not get a sequence of external numbers, and obtain only part of the solution. Indeed, for even standard $ n\in \mathbb{N} $ we have also $ \alpha_n=\oslash $. To see this, observe that the External Induction Principle \cite{Nelson} applied to the rules of multiplication of external numbers yields the local sequence $(\alpha_n)_{\mathrm{st} (n)\in \mathbb{N}}$ defined by
	\begin{equation}\label{alfa standaard}
	\alpha_n=\oslash,
	\end{equation}
	for, if $ \alpha_n=\oslash $, we have $ \alpha_{n+1}=\oslash \alpha_n=\oslash\cdot \oslash =\oslash$.
		
	Also, formula \eqref{pt1} is external, so {\it a priori} it does not satisfy the classical Induction Principle, and we cannot as such extend the local sequence $(\alpha_n)_{\mathrm{st} (n)\in \mathbb{N}}$ defined by \eqref{alfa standaard} to a global sequence. 
	
	Again because the classical Induction Principle is not valid for external formulas, for unlimited $ \omega\in \mathbb{N} $ the expression $ \oslash^{\omega} $ is not directly defined. Due to the presence of negative numbers, also the definition of $ \oslash^{\omega} $ through exponential and logarithmic functions is not straightforward, in \cite{koudjetithese} an adapted, symmetrical logarithmic function is used. Even within this interpretation the sequence $ (\alpha_n) $ defined by $ \alpha_n=\oslash^{n} $ is not the only solution of \eqref{pt1}, for a second solution is given by the sequence $(\beta_n)$ defined by
 	\begin{equation*}\label{alfa omega}\beta_n=
	\begin{cases}
	\oslash^{n}& \mbox{ if $ n $ is limited} \\ 
	\oslash^{\omega+n} & \mbox{ if $ n $ is unlimited }
	\end{cases}.
	\end{equation*}
	We will now consider a general class of external recursive relations, for which we define admissible solutions, through sequences of internal representatives.
	\begin{definition}
		Let $ k\in \mathbb{N} $ be standard and $ f:\mathbb{R}^{k}\rightarrow \mathbb{R} $ be internal. Let $ \alpha_1,...,\alpha_k $ be external numbers. We define $ f(\alpha_1,...,\alpha_k ) $ by
		\begin{equation*}
		f(\alpha_1,...,\alpha_k ) =\{f(a_1,...,a_k )	: a_1\in	\alpha_1 \wedge ... \wedge a_k\in \alpha_k\}.
		\end{equation*}	   
	\end{definition}
	\begin{definition}\label{recurrentsolution}
		Let $ k\in \mathbb{N} $ be standard and $ f:\mathbb{R}^{k+1}\rightarrow \mathbb{R} $ be internal. Let $ f:\mathbb{R}^{k+2}\rightarrow \mathbb{R} $ be internal and $ \alpha_1,...,\alpha_k $ be external numbers. A \emph{flexible recurrent relation} is a recurrent relation of the form
		\begin{equation}\label{recurrent}
		u_{n+1}=f(n,u_{n},\alpha_1,...,\alpha_k ).
		\end{equation}
		The solution of \eqref{recurrent} is defined to be the sequence of external sets $ (u_{n}) $, where $u_{0}:=\alpha_0 $ is an external number, and for every $ n\in \N $
		\begin{equation}\label{solutionrecurrent}\begin{split}
			u_{n+1}=\{&t_{n+1}	: \exists a_1,...,a_k\forall i\in\{1,...,k  \}(a_i:\mathbb{N}\rightarrow \alpha_i, a_i \mbox{ internal }\wedge\\
		& t_{n+1}=f(n,u_{n}, a_1(n),...,a_k(n)))\} .
		\end{split}
		\end{equation}
				We call the solution \emph{admissible} if all $u_{n}  $	   are external numbers.
	\end{definition}
 Proposition \ref{solpt1} determines an admissible solution to the equation \eqref{pt1} in the sense of Definition \ref{recurrentsolution}. It follows from the proof that in \eqref{solutionrecurrent}, at least in some cases, it suffices to consider representative functions with constant absolute value.

\begin{proposition}\label{solpt1}
	The solution of \eqref{pt1} is given by $\alpha_n=\pounds e^{-n\centernot\infty}$ for all $n\in \N$.
\end{proposition}
\begin{proof}
Let $n\in \N$. Let $ \alpha_{n} $ be the solution of \eqref{pt1} in the sense of Definition \ref{recurrentsolution}. We write $ \overline{\oslash}^{n}=\{{\varepsilon}^{n}: \varepsilon \in \oslash\wedge \varepsilon\geq 0  \}$. We show first that $  \alpha_{n}^{+}=\overline{\oslash}^{n}$. Both sets contain $ 0 $. As for the non-zero elements, let $ y\in \alpha_{n},y>0 $. Then there exists an internal sequence of positive infinitesimals $ (\varepsilon_k)_{1\leq k\leq n } $ such that $ \prod_{1\leq k\leq  n}\varepsilon_k =y$.
 Let $ \varepsilon_{-}=\min\left\{\lvert\varepsilon_k\rvert :1\leq k\leq n \right\} $ and $ \varepsilon_{+}=\max\{\lvert\varepsilon_k\rvert:1\leq k\leq n \} $. Then $ \varepsilon_{-},\varepsilon_{+}\simeq 0 $, so $ \varepsilon_{-}^{n},\varepsilon_{+}^{n} \in \alpha_{n}^{+}$. By continuity there exists $  \varepsilon$ with $ \varepsilon_{-}\leq \varepsilon\leq \varepsilon_{+} $ such that $ \varepsilon^{n}=y $, clearly $ \varepsilon \simeq 0,\varepsilon > 0  $. Hence $  \alpha_{n}^{+}\subseteq \overline{\oslash}^{n}$. The other inclusion is obvious. Hence $  \alpha_{n}^{+}=\overline{\oslash}^{n}$.

Secondly, we prove that $  \overline{\oslash}^{n}=\pounds^{+} e^{-n\centernot\infty}$. Clearly $ 0 $ is contained in both sets. Let $ \delta \simeq 0,\delta>0 $. Then $\log (\delta) \in -\centernot\infty$. Hence 
\begin{equation*}
	\delta^{n} \in \pounds^{+} \delta^{n}= \pounds^{+} e^{n\log (\delta)} \subseteq \pounds^{+} e^{-n\centernot\infty}, 
\end{equation*}
implying that $  \overline{\oslash}^{n}\subseteq\pounds^{+} e^{-n\centernot\infty}$.
The other inclusion is a consequence of the fact that a number of the form $ e^{-\omega n} $ with $ \omega $ unlimited is equal to $ \varepsilon^{n}$, with $\varepsilon$  infinitesimal and equal to $e^{-\omega } $. We conclude that $ \alpha_{n}^{+} =\overline{\oslash}^{n}=\pounds^{+} e^{-n\centernot\infty}$. Then $ \alpha_{n} =\pounds e^{-n\centernot\infty}$ by symmetry.
\end{proof} 

	Below we adapt some classical notions of stability like Lyapunov stability and asymptotic stability to our setting. In this manner one may model equilibria which are stable on a macroscopic scale, but on a microscopic scale are subject to small uncertainties or instabilities. For the classical notions we refer to \cite{Anishchencko et al.}. 

	\begin{definition}
		\label{N-lyapunov stability} Let $ N $ be a neutrix. A solution  $ (u_{n}) $ of the flexible recurrence relation \eqref{recurrent} is called \emph{$ N$-stable} if for every solution $ (v_{n}) $ such that $ v_{0}-u_{0}\subseteq N $ it holds that $ v_{n}-u_{n}\subseteq N $ for all $ n\in \N $.
	\end{definition}
\begin{definition}
	\label{N-asymptotic stability} Let $ N $ be a neutrix. A solution $ (u_{n}) $ of the flexible recurrence relation \eqref{recurrent} is called \emph{$ N$-asymptotically stable} if there exists $ \varepsilon >N$ such that every solution $ (v_{n}) $ with $ v_{0}-u_{0}< \varepsilon$ satifies $ \lim\limits_{n\to \infty}v_{n}-u_{n}= N $. It is called \emph{strongly $ N$-asymptotically stable} if for all such solutions  $ (v_{n}) $ the difference $ (v_{n}-u_{n}) $ has $ N $ as a strong limit.
\end{definition}
Next proposition shows that Definitions \ref{N-lyapunov stability} and \ref{N-asymptotic stability} are natural extensions of nonstandard equivalences of the classical notions of (asymptotic) stability.
\begin{proposition}
	Consider a standard recurrence relation. Then a stable standard solution $ (u_{n}) $ is $ \oslash$-stable and an asymptotically stable standard solution $ (u_{n}) $ is $ \oslash$-asymptotically stable.
\end{proposition}
\begin{proof}
	Without restriction of generality we may assume that $ (u_{n}) $ is identically zero. Also, by Definition \ref{recurrentsolution} we only need to consider internal solutions.
	As regards to stability, let $ (v_n) $ be an internal solution such that $ v_{0}\simeq 0 $. Suppose that there exist a standard $\varepsilon>0 $ and  $ n\in N $ such that $ \lvert v_{n}\rvert>\varepsilon  $. However there exists $ \delta>0 $ such that for all solutions $ v $ with $ \lvert v_{0}\rvert <\delta$ it holds that $ \lvert v_{n}\rvert<\varepsilon  $ for all  $ n\in \N $. By the Transfer Principle $ \delta>0 $ may be considered standard, a contradiction. Hence $ v_{n}\simeq 0 $ for all $ n\in \N $, meaning that $ 0 $ is $ \oslash$-stable.
	
	As regards to asymptotic stability, there exists $\eta>0 $ such that for every internal solution $ v $ with $ \lvert v_{0}\rvert < \eta$ it holds that $ \lim\limits_{n\to \infty}v_{n}= 0 $. By the Transfer Principle $ \eta>0 $ may be considered standard, i.e. $ \eta>\oslash $. Because $ \lim\limits_{n\to \infty}v_{n}= 0 $ for all solutions $ v $ with $ \lvert v_{0}\rvert < \eta$, clearly $ \lim\limits_{n\to \infty}v_{n}=\oslash  $. Hence $ 0 $ is $ \oslash$-asymptotically stable.
	\end{proof}
	
The following definition is adapted from \cite{Bergfamilysolutions}, and permits to describe a class of difference equations with a  $ \oslash $-stable solution, which is not $ \oslash $-asymptotically stable.

\begin{definition}\label{definitieafwatering}
	Consider a standard difference equation $D: x_{n+1}=f(n, x_{n})$, where $ f $ is of class $ C^{1} $. A standard solution $ u $ is called a \emph{drain} if it is eventually positive or eventually negative, and for all unlimited $\omega \in \mathbb{N}$ the following hold.
	\begin{enumerate}
			\item \label{subrational} For all appreciable $ p $ such that $p\omega \in \mathbb{N}$ it holds that
		\begin{equation*}
		\frac{u_{pn}}{u_{n}}\simeq 1. 
		\end{equation*}
		\item \label{asymp1}For every solution $v$ such that $v_{\omega }/u_{\omega }\simeq 1$ it holds
		that $v_n/u_n\simeq 1$ for all $n\simeq \infty $.
	\end{enumerate}
\end{definition}
Theorems 3.1 and 3.4 of \cite{Bergfamilysolutions} give conditions of existence and characterization for drains, and also some examples. By Definition \ref{definitieafwatering}.\ref{subrational} the growth of a drain must be less than $ n^{r} $ for all rational $ r>0 $. If a drain $ (u_{n}) $ is, say, positive from the beginning, and has a non-zero limit $ L $, the drain is $ \oslash $-stable, but not $ \oslash $-asymptotically stable. Indeed, using the $ S $-continuity of $ f $, one shows that all solutions $ v $ such that $ v_0\simeq u_0 $ satisfy $ v_n\simeq u_n $ for all standard $ n\in \N $. By Robinson's Lemma this property persists up to some unlimited $ \omega \in N$. Note that $ L $ is appreciable, hence all terms $ u_{n},v_{n} $ with unlimited $ n $ satisfy $ u_{n}\simeq v_{n}\simeq L $, hence are also appreciable. Then it follows from  Definition \ref{definitieafwatering}.\ref{asymp1} that $v_n\simeq u_n $ for all $ n\in \N $. Hence the drain is $ \oslash $-stable. It is not $ \oslash $-asymptotically stable, because if a solution $ w $ satisfies $ w_{\nu}\not\simeq 0 $ for some unlimited $ \nu \in \N $, there cannot exist $ n \in \N,n>\nu  $ such that $ w_{n}\simeq 0 $.

With the help of the results mentioned above, in \cite{Bergfamilysolutions} it was verified that for standard $ a>1  $ the solution $%
\overline{u}_{n}=1+\frac{(-1)^{n}}{n^{a}}$ of
\begin{equation*}
u_{n+1}=\left( -1+\frac{1}{n^{a}}\right) u_{n}+2-\frac{1}{n^{a}}%
+(-1)^{n}\left( \frac{1}{n^{a}}-\frac{1}{\left( n+1\right) ^{a}}-\frac{1}{%
	n^{2a}}\right) 
\end{equation*}%
is a drain. Hence $ (\overline{u}_{n}) $ is $ \oslash $-stable, but not $ \oslash $-asymptotically stable indeed.

As a consequence of Theorem \ref{Stelling sterke convergentie onbegrensd} solutions nearby an $ N$-asymptotically stable sequence  satisfy a strong convergence property, as given by the next proposition.
\begin{proposition}
	\label{strong N-asymptotic stability} Let $ N \neq 0 $ be a neutrix. If a solution $ (u_{n}) $ of the flexible recurrence relation \eqref{recurrent} is $ N$-asymptotically stable, then it is strongly $ N$-asymptotically stable. 
	\end{proposition}

\begin{proof}
	There exists $ \varepsilon >N$ such that for every solution $ (v_{n}) $ such that $ v_{0}-u_{0}< \varepsilon$ it holds that $ (v_{n}-u_{n})$ converges to $ N $. Let $(t_{n}) $ be an internal representative of $ (v_{n}-u_{n})$.  Then $(t_{n}) $ converges to $ N $. By Theorem \ref{Stelling sterke convergentie onbegrensd} it converges strongly to $ N $. Hence $ (v_{n}-u_{n})$ strongly converges to $ N $. Hence $ (u_{n}) $ is strongly $ N$-asymptotically stable. 
\end{proof}	
We end with a class of difference equations exhibiting asymptotically stable solutions with respect to some neutrix, which can be arbitrary.
\begin{proposition}
	Let $N \neq 0  $ be a neutrix and $ \alpha\in \E, |\alpha|<1 $. Consider the recurrent relation
	\begin{equation*}
	u_{n+1} =\alpha u_{n}+N. 
	\end{equation*}  
	Then $ 0 $ is a strongly $ (N/(1-|\alpha|)) $-asymptotically stable solution.
\end{proposition}
\begin{proof}
		Let $ (u_{n}) $ be a solution and $ (t_{n}) $ be an internal representative of $ (u_{n}) $. Then there are internal sequences   $ (a_{n}) $ of elements of $ \alpha $ and $ (b_{n}) $ of elements of $ N $ such that $ t_{n+1}=a_{n}t_{n}+b_{n} $ for all $ n\in N $. There exists $q\in |\alpha| $ such that $ |a_{n}|<q $ and $ c\in N,c>0 $ such that $ |b_{n}|<c $ for all $ n\in \N $. Then
	\begin{equation*}
	|t_{n+1}| \leq q|t_{n}|+c 
	\end{equation*} 
	for all $ n\in \N $, hence also    
	\begin{equation*}
	|t_{n}|\leq \left(  |t_{0}|+\frac{c}{1-q}\right)  q^{n}+\frac{c}{1-q}.
	\end{equation*} 
	Hence $ 0 $ is a $ (N/(1-|\alpha|)) $-asymptotically stable solution and by Proposition \ref{strong N-asymptotic stability} it is strongly $ (N/(1-|\alpha|)) $-asymptotically stable.
\end{proof}

\end{document}